\documentclass[a4paper,12pt,reqno]{article}

\usepackage[top=2cm,right=2cm,left=2cm,bottom=2cm]{geometry}
\usepackage[utf8]{inputenc}
\usepackage[UKenglish]{babel}
\usepackage{amsmath,amsfonts,amssymb,amsthm}
\usepackage{cite}
\usepackage{esint}
\usepackage{graphicx}
\usepackage{color}
\usepackage{mathrsfs}
\usepackage{floatrow}
\usepackage{mathtools,empheq}
\usepackage{bbm}
\usepackage{enumitem}
\usepackage{hyperref}

\numberwithin{equation}{section}

\newcommand{\red}[1]{#1}
\newcommand{\blue}[1]{#1}

\newcommand{\framed}[2]{\begin{center}\framebox{\begin{minipage}{#1} #2 \end{minipage}}\end{center}}

\newcounter{theorem}
\def\openthm#1#2{\refstepcounter{theorem}\bigskip

{\noindent\bf#1~\thetheorem\if#2!{. }\else{ (#2).}\fi}
\it}

\def\openrepthm#1#2{\bigskip

{\noindent\bf#1~#2{. }}
\it}

\def\thmskip{}
\newenvironment{theorem}[1][!]{\openthm{Theorem}{#1}}{\thmskip}

\newenvironment{lemma}[1][!]{\openthm{Lemma}{#1}}{\thmskip}
\newenvironment{corollary}[1][!]{\openthm{Corollary}{#1}}{\thmskip}

\newenvironment{definition}[1][!]{\openthm{Definition}{#1}}{\thmskip}

\newcounter{remark}
\def\openrem#1#2{\refstepcounter{remark}\bigskip
{\noindent \it \bfseries#1~\theremark\if#2!{. }\else{ (#2). }\fi}}
\newenvironment{remark}[1][!]{\openrem{Remark}{#1} \pushQED{\qed}}{\popQED}

\def\<{\langle}
\def\>{\rangle}

\def\R{\mathbb{R}}
\def\Q{\mathbb{Q}}
\def\N{\mathbb{N}}
\def\Z{\mathbb{Z}}

\def\Pr{\mathbb{P}}
\def\Ex{\mathbb{E}}
\def\Leb{\mathcal{L}}

\def\CC{{\rm C}}
\def\HH{{\rm H}}

\def\LL{{\rm L}}

\def\pot{{\rm pot}}
\def\per{{\rm per}}

\def\dd{{\rm d}}

\def\dx{\,\dd x}

\def\dq{\,\dd q}
\def\dt{\,\dd t}
\def\ds{\,\dd s}
\def\dPr{\,\dd \Pr}

\def\to{\rightarrow}

\def\sep{\,|\,}
\def\bsep{\,\b|\,}
\def\Bsep{\,\B|\,}

\DeclareMathOperator{\sgn}{sgn}

\def\eps{\varepsilon}

\def\b{\big}
\def\B{\Big}
\def\bg{\bigg}

\def\H{\mathcal{H}}

\def\Dom{\mathcal{D}}

\def\GDir{{\Gamma_{\mathrm{D}}}}
\def\GNeu{{\Gamma_{\mathrm{N}}}}

\def\Vel{\mathcal{V}}

\def\div{\mathrm{div}}

\def\HBC{\HH^1_\GDir}
\def\HBCPrime{\left(\HBC\!\right)'}

\def\L2Pot{\LL^2_\pot}

\def\dps{\displaystyle}

\begin{document}

\title{Stochastic homogenization of a scalar viscoelastic model exhibiting stress--strain hysteresis}
\author{Thomas Hudson$^{1}$, Fr\'ed\'eric Legoll$^{2,3}$ and Tony Leli\`evre$^{2,4}$
\\
{\footnotesize $^1$ Mathematics Institute, Zeeman Building, University of Warwick, Coventry, CV4 7AL, United Kingdom}
\\
{\footnotesize $^2$ CERMICS, \'Ecole des Ponts ParisTech, 77455 Marne-La-Vall\'ee Cedex 2, France}
\\
{\footnotesize $^3$ Laboratoire Navier, \'Ecole des Ponts ParisTech, 77455 Marne-La-Vall\'ee Cedex 2, France}
\\
{\footnotesize $^4$ Inria Paris, MATHERIALS project-team, 2 rue Simone Iff, CS 42112, 75589 Paris Cedex 12, France}
}

\thanks{The work of TH was funded by a public grant overseen by the French National Research Agency (ANR) as part of the ``Investissements d'Avenir'' program (reference: ANR-10-LABX-0098). TH was also funded by an Early Career Fellowship awarded by the Leverhulme Trust (ECF-2016-526). The work of FL and TL was supported by the European Research Council under the European Union's Seventh Framework Programme (FP/2007-2013) / ERC Grant Agreement number 614492.}
    
\begin{abstract}
Motivated by rate--independent stress--strain hysteresis observed in filled rubber, this article considers a scalar viscoelastic model in which the constitutive law is random and varies on a lengthscale which is small relative to the overall size of the solid. Using a variant of stochastic two--scale convergence as introduced by Bourgeat, Mikelic and Wright, we obtain the homogenized limit of the evolution, and demonstrate that under certain hypotheses, the homogenized model exhibits hysteretic behaviour which persists under asymptotically slow loading. These results are illustrated by means of numerical simulations in a particular one--dimensional instance of the model.
\end{abstract}

% \keywords{Hysteresis, Stochastic homogenization, Viscoelasticity, Nonlinear time-dependent PDEs. MSC Classification: 74Q10, 35Q74}

\maketitle

\section{Introduction}

Hysteresis is the phenomenon of ``history--dependence'' in a physical system. In mechanical systems, \emph{stress--strain hysteresis} occurs when the stress observed during loading depends on the path taken by the system in order to arrive at a particular strain, and not simply on the value of the strain itself. It follows that such stresses are non--conservative fields, i.e.~they cannot be directly expressed as the gradient of a potential energy function. In such systems, mechanical energy is dissipated through a thermodynamic process.
  
The Second Law of Thermodynamics dictates the most basic mechanism for dissipating mechanical energy in a solid: some of the work done on the solid must always be lost as heat, increasing the entropy in the system~\cite{TN04}. Other less ubiquitous mechanisms may involve the storage of mechanical energy through magnetic effects~\cite{LU00} or through molecular rearrangement~\cite{BJC95}. Typically, thermal dissipation depends on the work rate, while the latter examples involve a stress--induced phase transition in some order parameter of the system, and are hence \emph{rate--independent}.
  
Filled rubbers are a class of materials in which stress--strain hysteresis is observed in experiments, and persists at very low strain rates~\cite{L96,MF04}, indicating a rate--independent mechanism. Such rubbers include the most common varieties which are produced for commercial and industrial applications. Typically, they are composed of a rubber matrix containing microscopic ``filler'' particles, added to improve the mechanical properties of the material. The matrix is formed of polymer chains which are bonded both to the surface of the filler particles and to each other via sulphur bonds formed during the process of vulcanization. Viewed through a microscope, the filler particles are seen to form a complex random network throughout the material (see for example Figure~25 of~\cite{K03} and Figure~\ref{fig:microstructure} below).
  
At present, the mechanism which causes rate--independent stress--strain hysteresis in filled rubbers is not well understood. Our motivation here is to propose and mathematically study a class of simple viscoelastic models which exhibit the features observed experimentally, and reflect the random structure of the underlying material. In particular, we propose a micromechanical model in which energy is dissipated by frictional stresses acting to counter local deformation of the polymer matrix.
  
This model has two key features. First, the complex microstructure of the material is modelled by assuming that material parameters are random, depend upon the material point and vary on a lengthscale much smaller than the overall size of the body. Second, energy is assumed to be dissipated by internal frictional mechanisms which act to counter local changes in strain, rather than rigid body translations and infinitesimal rotations, in keeping with the physical assumption of \emph{frame indifference}. Other similar constitutive approaches in the literature include~\cite{L96,KR98,MK00,M07}. The mathematical framework we exploit is sufficiently general to allow us to incorporate a form of internal ``dry friction'', which leads to hysteretic behaviour which persists at arbitrarily small strain rates.

\smallskip

After presenting the constitutive assumptions made in this model, we prove the long--time existence of solutions under time--varying Dirichlet loading conditions imposed on part of the boundary, which reflect the boundary conditions used in standard material testing experiments. Our general mathematical approach to this problem is through the application of the theoretical tools developed to study doubly--nonlinear evolution equations, which are described in detail in~\cite{Roubicek}, and are studied with particular reference to the modelling of hysteresis in~\cite{Visintin}.
  
Our main result is then to obtain a homogenization result for this model in the limit where the lengthscale for microscopic variation versus body size vanishes. Technically--speaking, this result corresponds to a stochastic homogenization result for a random stationary doubly--nonlinear evolution problem. Greater spatial heterogeneity and differing assumptions on the dissipation mean that the resulting model lies in a distinct (but related) class of models to the Prandtl--Ishlinski\u{\i} rheological models studied in Chapters~III and~VII of~\cite{Visintin}. In the limiting model we derive, the effect of the spatial heterogeneity is tracked via a ``corrector field'' which describes the microscopic oscillations of the strain, and is characterised as the solution of an explicit differential inclusion.

The basis for our proof of this homogenization result is the theory of two--scale stochastic convergence developed in~\cite{BMW94}. This theory is inspired by the definition of two--scale convergence proposed by Nguetseng in~\cite{N89} and subsequently explored in detail in~\cite{Allaire94} in a deterministic setting. The theory is sufficiently general such that it may be viewed as including both periodic homogenization for deterministic systems and homogenization for random systems which are piecewise constant on a lattice with a random shift of the origin. Such cases form the particular examples constructed in Section~\ref{sec:examples}. We also illustrate our results in the particular case of a periodic setting in Section~\ref{sec:periodic}.

Recently, a general theoretical framework for evolutionary $\Gamma$--convergence has been developed to describe the convergence of doubly nonlinear evolution equations. The work~\cite{MielNotes} describes several general results under which ``strong'' convergence of an evolutionary semigroup may be deduced. Our analysis provides an example of a case in which a sequence of evolution semigroups converges only in a weak sense (see~\cite{T90} for an example of another such result) but the limit can still be identified. As such, our results lie outside the remit of the general theory of evolutionary $\Gamma$-convergence. Nevertheless, we make use of many of the ideas underlying the development of this framework.

At the time of submission, we became aware of related recent works on the homogenization of hysteretic systems~\cite{HS16,Heida17}. Our results differ from the setting of~\cite{Heida17} in that we consider systems driven by time--dependent boundary conditions, inspired by cyclic loading experiments, which leads us to consider time--dependent dissipation potentials. Further, we note that the technical tools we use differ: while the results in the latter reference are based on the construction of a Palm measure and a notion of two--scale convergence proposed by Zhikov and Pyatnitskii in~\cite{ZP06}, as mentioned above, we use the theory developed by Bourgeat, Mikelic and Wright in~\cite{BMW94}.

\smallskip
  
The article is organized as follows. Section~\ref{sec:constitutive} provides a detailed introduction to the model which we consider and a statement of the main results (in Section~\ref{sec:main}), as well as a numerical study in a simple one--dimensional case and some results demonstrating the qualitative properties of the model. This is intended to be as self--contained as possible for the reader less interested in the technical proofs of the subsequent existence and homogenization results.
  
Section~\ref{sec:setup} presents the mathematical background required to precisely state our results: the existence and uniqueness of a solution to the highly oscillatory evolution problem (Theorem~\ref{th:eps_existence}) and the identification of a homogenized limit (Theorem~\ref{th:main}). We also recall and expand some key aspects of the theory of stochastic two--scale convergence introduced in~\cite{BMW94}. The proof of Theorem~\ref{th:eps_existence} is given in Section~\ref{sec:wellposedness}, while Section~\ref{sec:limit} is devoted to the proof of Theorem~\ref{th:main}.
   
\section{Constitutive assumptions and main results}
\label{sec:constitutive}

Before presenting our model in its full generality in Section~\ref{sec:general}, we describe a simple particular case in Section~\ref{sec:InitialExample} and some illustrative numerical simulations in Section~\ref{sec:numerics}. We next state our main results in Section~\ref{sec:main}. Section~\ref{sec:hysteresis} discusses the hysteretic behaviour of our model.

\subsection{A simple illustrative example}
\label{sec:InitialExample}
  
We begin by formulating a simple one--dimensional case as an illustration of the more general model we subsequently consider. In this case, our model is closely related to a Prandtl--Ishlinski\u{\i} model of Stop--type, described in Chapter~III of~\cite{Visintin}. Let $\Dom = [0,1]$ be the reference configuration for a material undergoing a time--dependent deformation. The displacement is described by the function $y:\Dom\times[0,+\infty)\to\R$, where the second independent variable represents time. Here and throughout the rest of the article, we write $\dot{y}$ to denote the partial derivative of $y$ with respect to time $t$, and $D_xy$ to denote the gradient of $y$ with respect to the variable $x$, i.e. the strain. We consider the material under a loading experiment in which we assume that ``rigid'' boundary conditions are enforced, i.e. that $y(0,t) = 0$ and $y(1,t) = \ell(t)$ for any $t\in[0,T]$, where $\ell(t)$ is the elongation of the body (at time $t$, the length of the system is thus $1+\ell(t)$). The function $\ell$ is assumed to be periodic in time, as is common in experimental settings studying hysteresis.
  
We suppose that the internal stress is composed of three additive components: an elastic stress $\sigma_\text{e}$, a ``wet'' viscous frictional stress $\sigma_\text{w}$, and a ``dry'' frictional stress $\sigma_\text{d}$. The parameters relating these stresses to the current deformation are assumed to be random, and to vary rapidly on a lengthscale, denoted $\eps$, which is much shorter than the body itself. More precisely, at a given material point $x\in\Dom$ and time $t\in[0,\infty)$, the former two stress components are assumed to be functions of the strain and strain rate respectively, namely
\begin{equation*}
  \sigma_\text{e}(t,x) = A\left(\omega,\frac{x}{\eps}\right)D_xy(t,x)
  \quad\text{and}\quad
  \sigma_\text{w}(t,x) = \nu\left(\omega,\frac{x}{\eps}\right)D_x\dot{y}(t,x).
\end{equation*}
For the dry frictional stress, we assume that
\begin{equation*}
  \sigma_\text{d}(t,x) = \begin{cases}
    \dps \phantom{-}\mu\left(\omega,\frac{x}{\eps}\right)& \text{if $D_x\dot{y}(t,x)>0$},\\[3mm]
    \dps -\mu\left(\omega,\frac{x}{\eps}\right)& \text{if $D_x\dot{y}(t,x)<0$},
  \end{cases}
\end{equation*}
and if $D_x\dot{y}(t,x)=0$, then $\sigma_\text{d}(t,x)$ takes on some value in $\dps \left[-\mu\left(\omega,\frac{x}{\eps}\right),\mu\left(\omega,\frac{x}{\eps}\right)\right]$ in order to satisfy a force balance. Here, $\omega\in\Omega$ denotes a realization of the constitutive relation drawn from an appropriate probability space.
  
These constitutive assumptions are motivated by our understanding of the microstructure of filled rubber, illustrated on Figure~\ref{fig:microstructure} below. As mentioned in the introduction, filled rubbers are made up of two principal components: large filler particles and a polymer matrix. Elastic stresses are induced by the polymer matrix acting to increase the entropy of the polymer chains at fixed temperature~\cite{Treloar}. The ``wet'' frictional stress corresponds to the action of viscous dissipation via thermal vibration of the polymer matrix, while the ``dry'' frictional stress represents the opposition to motion by friction between the filler particles and the polymer chains as they move in order to allow deformation. We also note that, since all stresses depend only on the strain and strain rate, they are frame independent: rigid body motions (i.e. translations in this one--dimensional setting) do not affect their definition.
  
\begin{figure}[tbp]
  \includegraphics[width=0.6\textwidth]{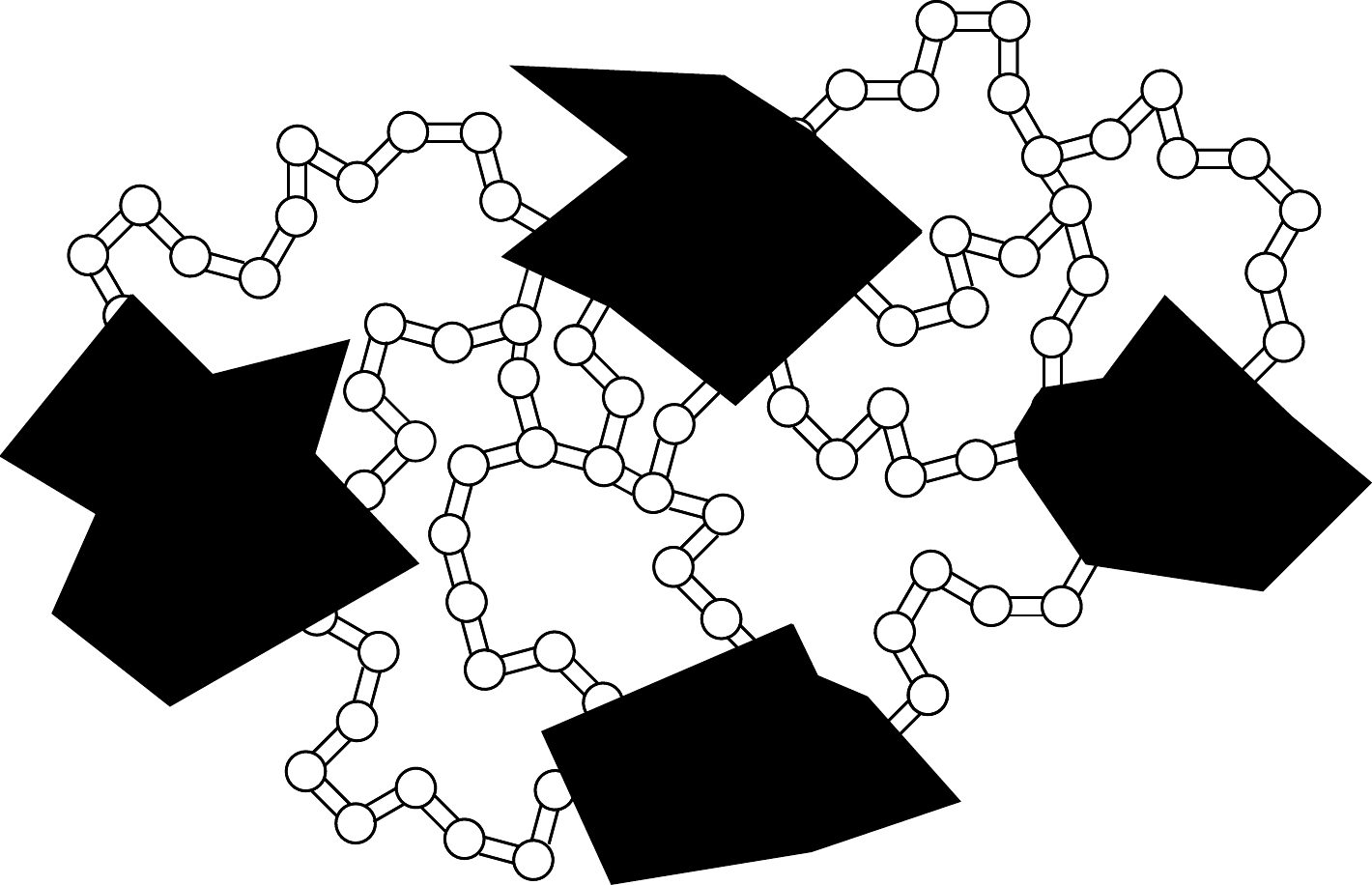}
  \caption{Schematic representation of the microscopic structure of filled rubber. Large black particles are filler, white particles represent the rubber matrix. \label{fig:microstructure}}
\end{figure}
  
If the material undergoes slow loading, so that inertial effects may be neglected, it follows that at all times the net force on each material point vanishes. Hence, in an appropriately weak sense, we suppose that the internal stresses satisfy the force balance
\begin{equation}
  -\div\b(\sigma_\text{e}(t,x) +\sigma_\text{w}(t,x)+\sigma_\text{d}(t,x)\b)=0\quad\text{for any $t\in[0,T]$}.
  \label{eq:1DForceBalance}
\end{equation}
In one dimension, any divergence--free field is constant, so~\eqref{eq:1DForceBalance} entails that there exists a (random) function $\overline{\sigma}(t)$ such that
\begin{equation}
\forall x \in \Dom, \quad \sigma_\text{e}(t,x)+\sigma_\text{w}(t,x) + \sigma_\text{d}(t,x)=\overline{\sigma}(t).
\label{eq:1DForceBalance_bis}
\end{equation}
We may then write
\begin{align}
  A\left(\omega,\frac{x}{\eps}\right) D_xy +\nu\left(\omega,\frac{x}{\eps}\right) D_x\dot{y}+\mu\left(\omega,\frac{x}{\eps}\right) &=\overline{\sigma}(t) \quad\text{when $D_x\dot{y}>0$},
  \nonumber
  \\
  A\left(\omega,\frac{x}{\eps}\right) D_xy+\sigma_\text{d}(t,x) &=\overline{\sigma}(t) \quad\text{when $D_x\dot{y}=0$},
  \label{eq:StressBalance}
  \\
  A\left(\omega,\frac{x}{\eps}\right) D_xy +\nu\left(\omega,\frac{x}{\eps}\right) D_x\dot{y}-\mu\left(\omega,\frac{x}{\eps}\right) &=\overline{\sigma}(t) \quad\text{when $D_x\dot{y}<0$}.
  \nonumber
\end{align}
  
We now recast~\eqref{eq:StressBalance} into two equivalent, but more mathematically convenient forms, which turn out to be general enough to allow us to prove convergence results. Let us define an elastic energy density $W(\omega,x,\xi)$ and a dissipation potential density $\psi(\omega,x,\xi)$ via
\begin{equation*}
  W(\omega,x,\xi) := \frac12 \, A(\omega,x) \, \xi^2 \quad \text{and} \quad \psi(\omega,x,\xi) := \frac12 \, \nu(\omega,x) \, \xi^2 + \mu(\omega,x) \, |\xi|.
\end{equation*}
We note that $\sigma_\text{e}(t,x) = D_\xi W\b(\omega,x/\eps,D_xy(t,x)\b)$. Furthermore, $\psi$ is convex in $\xi$ and hence has a subdifferential, denoted $\partial_\xi\psi(\omega,x/\eps,\xi)$, which is
\begin{equation*}
  \partial_\xi\psi(\omega,x/\eps,\xi) = \begin{cases}
    \dps \left\{\nu\left(\omega,\frac{x}{\eps}\right)\xi + \mu\left(\omega,\frac{x}{\eps}\right)\right\} & \text{when $\xi>0$},\\[2mm]
    \dps \phantom{\{}\left[-\mu\left(\omega,\frac{x}{\eps}\right),\mu\left(\omega,\frac{x}{\eps}\right)\right] & \text{when $\xi=0$},\\[2mm]
    \dps \left\{\nu\left(\omega,\frac{x}{\eps}\right)\xi - \mu\left(\omega,\frac{x}{\eps}\right)\right\} & \text{when $\xi<0$}.
  \end{cases}
\end{equation*}
It follows that the frictional stresses satisfy the inclusion $\sigma_\text{w}(t,x)+\sigma_\text{d}(t,x) \in \partial_\xi\psi(\omega,x/\eps,D_x\dot{y}(t,x))$, and we may express~\eqref{eq:1DForceBalance_bis} in the compact form
\begin{equation}
\forall x \in \Dom, \qquad D_\xi W\left(\omega,\frac{x}{\eps},D_xy\right) + \partial_\xi \psi\left(\omega,\frac{x}{\eps},D_x\dot{y}\right) \ni \overline{\sigma}(t).
\label{eq:ForceInclusion}
\end{equation}
  
Next, we derive a formulation which is ``conjugate'' to~\eqref{eq:ForceInclusion}, in the sense that it provides an equivalent relation between strain and strain rate, rather than relating dissipative and elastic stresses as~\eqref{eq:ForceInclusion} does (for further motivation of this construction, see Section~1 of~\cite{MielNotes}). This reformulation is particularly convenient for the purposes of the numerical results we describe in Section~\ref{sec:numerics}. Recall that, for any vector space $X$, the Legendre--Fenchel transform of a \red{proper convex function $f:X\to\R\cup\{+\infty\}$ (i.e. a convex function which takes a finite value for at least one point in $X$)} is the function $f^*:X'\to\R\cup\{+\infty\}$ defined by
\begin{equation*}
  f^*(\sigma) := \sup_{\xi\in X} \b\{\<\sigma,\xi\>_X-f(\xi)\b\},
\end{equation*}
where $\<\cdot,\cdot\>_X:X'\times X\to\R$ is the duality bracket between $X$ and its topological dual. We note that, by definition, $f(\xi)+f^*(\sigma)\geq \<\sigma,\xi\>_X$ for any $\xi\in X$ and $\sigma\in X'$. When $f$ is convex and lower semicontinuous, the following statements are all equivalent:
\begin{equation}
  (1)\quad\sigma\in\partial_\xi f(\xi), \qquad (2)\quad \xi\in\partial_\sigma f^*(\sigma)\quad\text{and}\qquad(3)\quad f(\xi)+f^*(\sigma) = \<\sigma,\xi\>_X.
  \label{eq:LFEquivalence}
\end{equation}
A proof of this fact is given in Theorem~23.5 of~\cite{Rockafellar}. \blue{We also recall that $f^*$ is always convex, being the supremum of convex functions.}

A straightforward computation demonstrates that
\begin{equation*}
\psi^*(\omega,x,\sigma)=
\begin{cases}
  0 & \text{if $|\sigma|\leq \mu(\omega,x)$},
  \\
  \displaystyle \frac{\b(|\sigma|-\mu(\omega,x)\b)^2}{2\nu(\omega,x)} & \text{if $|\sigma|>\mu(\omega,x)$}.
\end{cases}
\end{equation*}
Using the equivalence of the statements given in~\eqref{eq:LFEquivalence} and the fact that $\psi^*$ is continuously differentiable with respect to its third variable, so that its subdifferential is simply its derivative, \eqref{eq:ForceInclusion} is equivalent to the ``rate equation''
\begin{equation}
  D_x\dot{y}= D_\sigma\psi^*\left(\omega,\frac{x}{\eps},\overline{\sigma}(t)-D_\xi W\left(\omega,\frac{x}{\eps},D_xy\right)\right).
  \label{eq:RateEquation}
\end{equation}
This rate equation and the $\psi$--$\psi^*$ framework are convenient formulations of the problem mathematically, and also for implementing the numerical experiments we describe in the following section.
  
\subsection{Numerical simulation}
\label{sec:numerics}

We now present a numerical study of the model described in the previous section in order to motivate our subsequent work. We define the random constitutive relations as follows: divide $\R$ into intervals $I_i:=[p+i,p+i+1)$ for $i\in\Z$, where $p$ is a random variable uniformly distributed in $(-1,0]$, and assume that $A$, $\mu$ and $\nu$ are identically independently distributed constants on each interval $I_i$. More precisely, suppose that $A$, $\mu$ and $\nu$ are constant on each $I_i$, with values chosen uniformly at random from the following sets:
\begin{equation*}
  A|_{I_i} \in \{1,3\},\qquad \mu|_{I_i}\in\{0,0.4,0.7\} \quad\text{and}\quad \nu|_{I_i} \in\{0.05,0.1\}.
\end{equation*}
Define a reference displacement $\overline{y}(t,x) := \ell(t)x$, and suppose that initially $y(0,x) = \ell(0)x$. It is convenient to introduce the displacement away from this reference, $u(t,x):= y(t,x) - \overline{y}(t,x)$, satisfying the initial condition $u(0,x) = 0$ and the boundary condition $u(0,t) = u(1,t) = 0$. We infer from~\eqref{eq:RateEquation} that
\begin{equation*}
  %\label{eq:revival}
  \dot{\ell} + D_x\dot{u}= D_\sigma\psi^*\left(\omega,\frac{x}{\eps},\overline{\sigma}(t)-D_\xi W\left(\omega,\frac{x}{\eps},\ell+D_xu\right)\right).
\end{equation*}
It is straightforward to check that $D_xu$ is constant in space on each interval $I_i$ but time--dependent. It is therefore natural to introduce a vector of strains $S = \left\{ S_i \right\}_{0 \leq i \leq n}$ with $n=\lceil\eps^{-1}\rceil$, where $S_i:= D_xu|_{I_i}$ for each $i\in\b\{0,\ldots,n\b\}$. We remark that $S_i$ is simply the difference between the true strain for points in $I_i$ and the purely linear strain response to the boundary conditions, which would be $\ell(t)$. We also introduce the vector of constant elastic stresses, $\Sigma_i:=D_\xi W(\omega,x/\eps,\ell+D_xu)\big|_{I_i}$. To generate a random constitutive relation, define vectors of random parameters $A$, $\mu$ and $\nu$, where each element of these vectors is the corresponding constant value on the interval $I_i$, i.e. $A_i = A|_{I_i}$. With these definitions, we find that, on each interval $I_i$, we have
\begin{equation}\label{eq:1D_ODE_system_a}
  \Sigma_i(t) = A_i\cdot [\ell(t)+S_i(t)]
\end{equation}
and
\begin{equation}\label{eq:1D_ODE_system_b}
  \dot{\ell}(t)+\dot{S}_i(t) 
  = \begin{cases}
    0 & \text{if \ \ $|\overline{\sigma}(t)-\Sigma_i(t)|\leq \mu_i$},
    \\
    \displaystyle\frac{|\overline{\sigma}(t)-\Sigma_i(t)|-\mu_i}{\nu_i} \, \sgn\big(\overline{\sigma}(t)-\Sigma_i(t)\big) & \text{if \ \ $|\overline{\sigma}(t)-\Sigma_i(t)|\geq \mu_i$}.
  \end{cases}
\end{equation}

\subsubsection{Numerical method}

We now describe the numerical scheme we use to solve~\eqref{eq:1D_ODE_system_a}--\eqref{eq:1D_ODE_system_b}. Let $\Delta t$ denote a timestep, $S^j$ and $\Sigma^j$ be the values of $S$ and $\Sigma$, the vectors of strains and elastic stresses, computed at the $j$th timestep. Define the forward finite difference $\dps \Delta S^j:=\frac{S^{j+1}-S^j}{\Delta t}\in\R^{n+1}$. Let $\ell^j=\ell(j\Delta t)$, and set $\overline{\sigma}^j$ be the total stress at the $j$th timestep.

We discretize~\eqref{eq:1D_ODE_system_a}--\eqref{eq:1D_ODE_system_b} in the following way:
$$
S^{j+1} := S^j + \Delta t \, \Delta S^j \qquad \text{and} \qquad \Sigma^{j+1}_i := A_i(\ell^{j+1}+S^{j+1}_i)
$$
where $\Delta S^j$ and $\overline{\sigma}^{j+1}$ are chosen to solve
\begin{equation}\label{eq:1Drate}
  \Delta S^j_i
  = \begin{cases}
    -\Delta\ell^j & \text{if \ \ $|\overline{\sigma}^{j+1}-\Sigma^{j+1}_i|\leq \mu_i$},
    \\
    \displaystyle\frac{|\overline{\sigma}^{j+1}-\Sigma^{j+1}_i|-\mu_i}{\nu_i} \, \sgn(\overline{\sigma}^{j+1}-\Sigma^{j+1})-\Delta\ell^j & \text{if \ \ $|\overline{\sigma}^{j+1}-\Sigma^{j+1}_i|\geq \mu_i$},
  \end{cases}
\end{equation}
subject to
\begin{equation} \label{eq:1Dconstraint}
0 =\sum_{i=0}^n \Big|(\eps I_i)\cap[0,1] \Big| \, \Delta S^j_i.
\end{equation}
The equation~\eqref{eq:1Drate} describes the rate given a stress $\overline{\sigma}^{j+1}$, and~\eqref{eq:1Dconstraint} is a constraint which ensures that the boundary conditions are satisfied. Using~\eqref{eq:LFEquivalence}, the equation~\eqref{eq:1Drate} is equivalent to solving the inclusion
\begin{equation*}
\overline{\sigma}^{j+1} \in A_i \b(\ell^{j+1}+S^j_i + \Delta t \, \Delta S^j_i \b)+\begin{cases}
    \b\{\nu_i(\Delta S^j_i+\Delta\ell^j) + \mu_i\b\} & \text{if \ \ $\Delta S^j_i>-\Delta\ell^j$},\\[2mm]
    \b[-\mu_i,\mu_i\b] & \text{if \ \ $\Delta S^j_i=-\Delta\ell^j$},\\[2mm]
    \b\{\nu_i(\Delta S^j_i+\Delta\ell^j) - \mu_i\b\} & \text{if \ \ $\Delta S^j_i<-\Delta\ell^j$}.
  \end{cases}
\end{equation*}
Since the right--hand side of this inclusion is monotone in $\Delta S^j_i$, it follows that there is a unique solution for any $\overline{\sigma}^{j+1}$, which moreover increases as $\overline{\sigma}^{j+1}$ increases.

Given $\overline{\sigma}^{j+1}$, we can solve for $\Delta S^j \in \R^{n+1}$, and progressively optimize $\overline{\sigma}^{j+1}$ to find a value such that the constraint~\eqref{eq:1Dconstraint} is approximately satisfied. Since the inverse function for the right--hand side is only Lipschitz and not differentiable, we use the secant method to perform this optimization.

\begin{remark}
We have explained here how to solve the problem \emph{after} time--discretization. The well-posedness of the problem (for $\eps > 0$ fixed) \emph{before} time--discretization is established in Section~\ref{sec:wellposedness}.
\end{remark}
  
\subsubsection{Numerical results}

Our calculations are carried out in Julia 0.6.2~\cite{Julia}. Random samples are generated via the \texttt{sample} command from the \texttt{Distributions} package, which by default uses a Mersenne--Twister algorithm to generate pseudorandom numbers. Plots are created using \texttt{PyPlot}, which provides an interface with the Python plotting library \texttt{matplotlib}.
  
Figure~\ref{fig:rates} shows stress--strain curves for a fixed sample generated with $\eps=1/200$. Each curve corresponds to the same loading $\ell_\delta(t) := \sin^2(2\pi \, \delta  \, t)$ with different rates $\delta\in\{1,2^{-2},2^{-4},2^{-6},2^{-8}\}$, over the time range $[0,1/\delta]$, which corresponds to 2 cyclic loading periods. The system exhibits persistent hysteresis as the loading rate $\delta$ decreases, and the stress--strain curve appears to converge to a fixed limit cycle, indicating a rate--independent component in the model.

\begin{figure}[htbp]
  \includegraphics[width=0.8\textwidth]{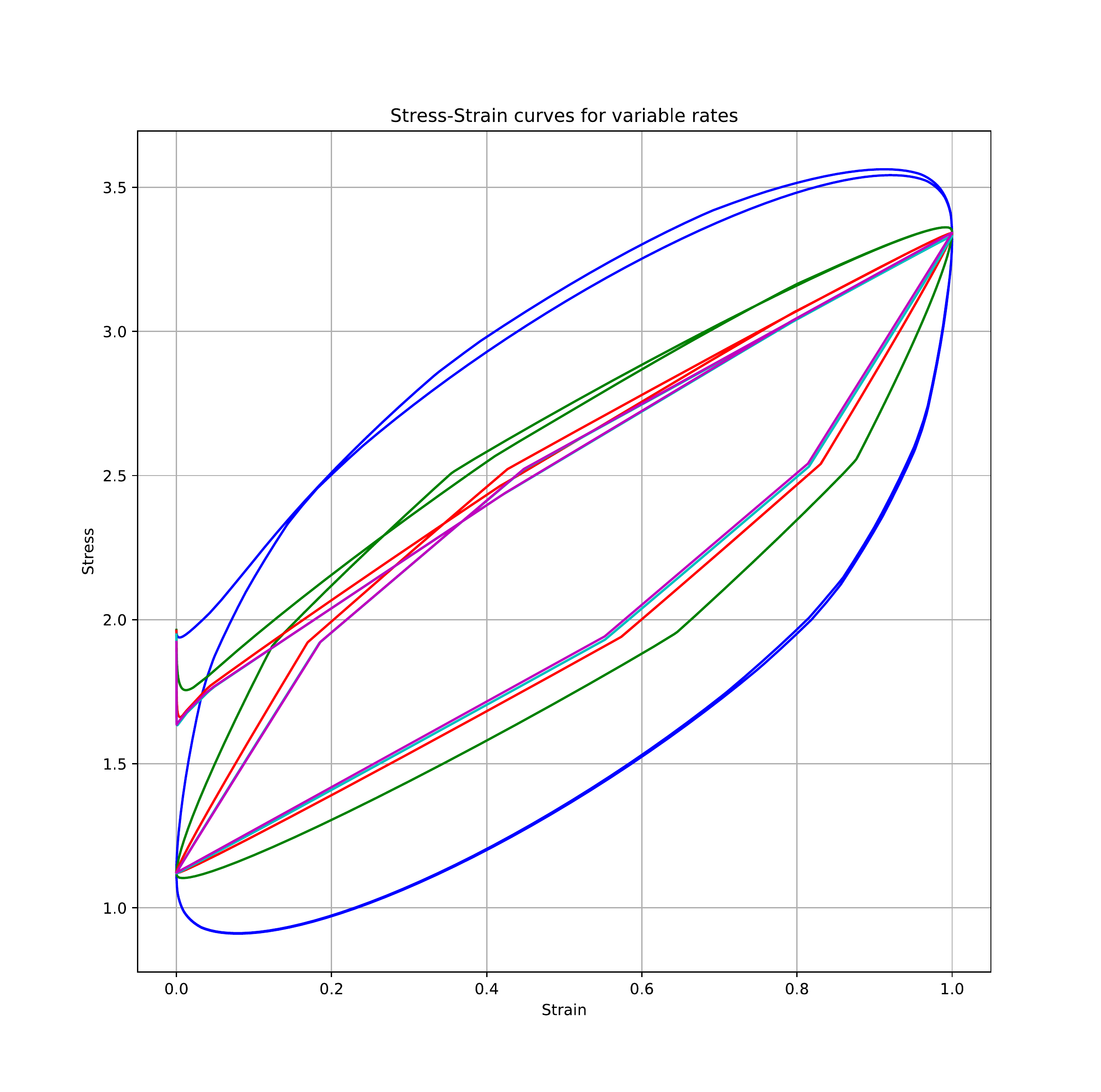}
  \caption{Loading and unloading at different rates for a given sample material with $\eps=1/200$ (blue: $\delta = 1$; green: $\delta = 2^{-2}$; red: $\delta = 2^{-4}$; cyan: $\delta = 2^{-6}$, magenta: $\delta=2^{-8}$). We represent $\overline{\sigma}(t)$ (the stress in the system, which is independent of $x$) as a function of $\ell_\delta(t)$ (the total strain of the system). The numerical evidence of convergence when $\delta \to 0$ is clear. \label{fig:rates}}
\end{figure} 
  
Next, for each $\eps\in\{1/100, 1/200, 1/400, 1/800, 1/1600\}$, $50$ random environments are generated. For each realization, the dynamics are simulated with $\delta=0.1$, again over 2 periods. The results of these calculations are shown on Figures~\ref{fig:sample} and~\ref{fig:variance}. Figure~\ref{fig:sample} shows the mean (over the random environments) stress--strain curve for $\eps=1/1600$, along with an envelope indicating an error bar of one standard deviation in the calculated stress at each timestep. Figure~\ref{fig:variance} shows the decrease in the variance of the stress calculated at $t=0.25$ as $\eps$ decreases. The latter figure shows a typical linear relationship between $\eps$ and the variance, indicating convergence. A similar decrease is observed at every time.

Finally, we note that, for fixed rate $\delta\ll1$, numerical experiments show that taking the magnitude of the possible values of $\mu$ to zero results in a collapse of the hysteresis loops, indicating that the internal dry friction included in the model is indeed the mechanism which results in hysteresis which persists at very low strain rates.
 
\begin{figure}[tp]
  \includegraphics[width=0.8\textwidth]{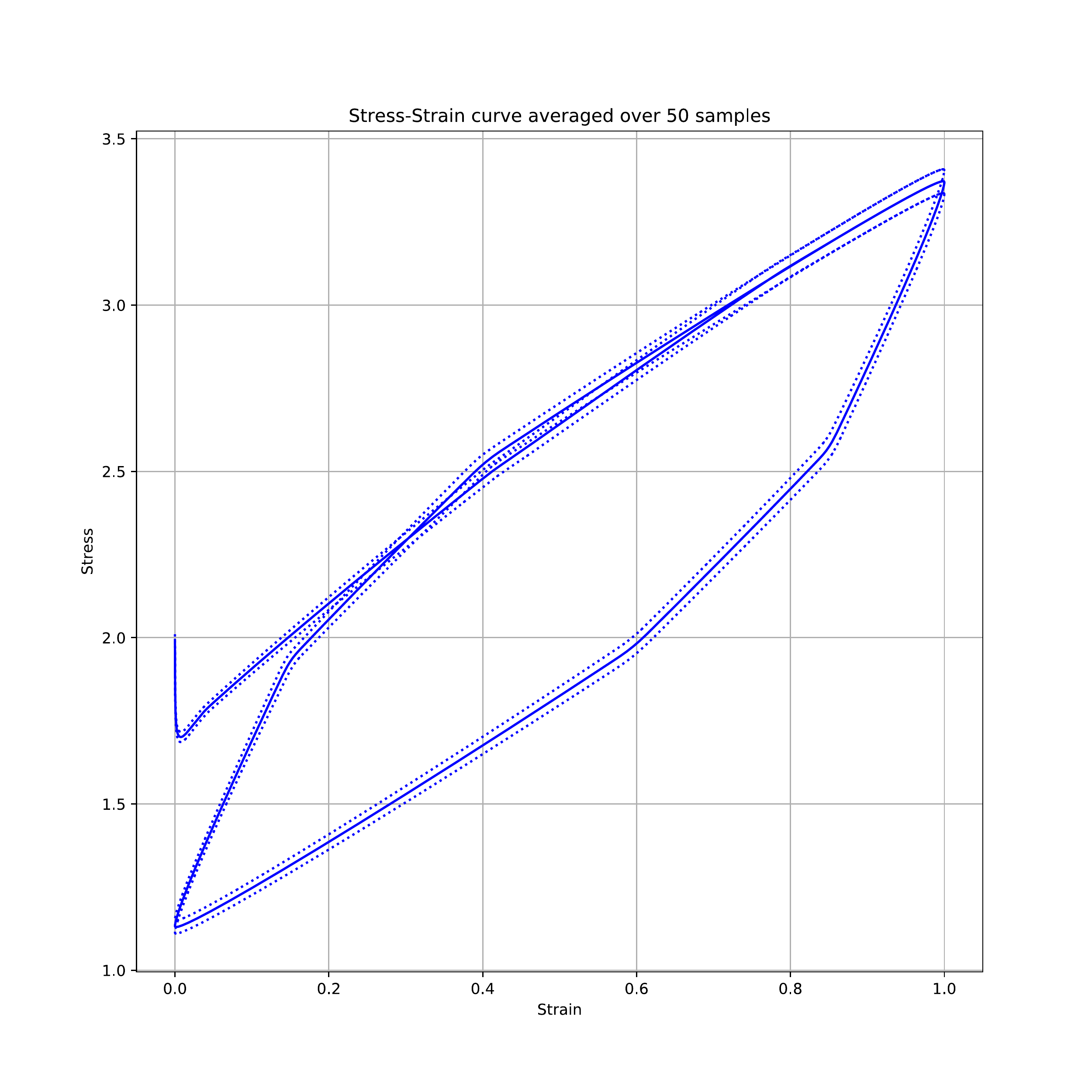}
  \caption{Stress--strain curve (mean taken over $50$ realizations) for $\eps=1/1600$. Dotted lines indicate the standard deviation in the calculated stress at each timestep. For all simulations, the rate parameter is taken to be $\delta=0.1$. \label{fig:sample}}
\end{figure}
  
\begin{figure}[tp]
  \includegraphics[width=0.6\textwidth]{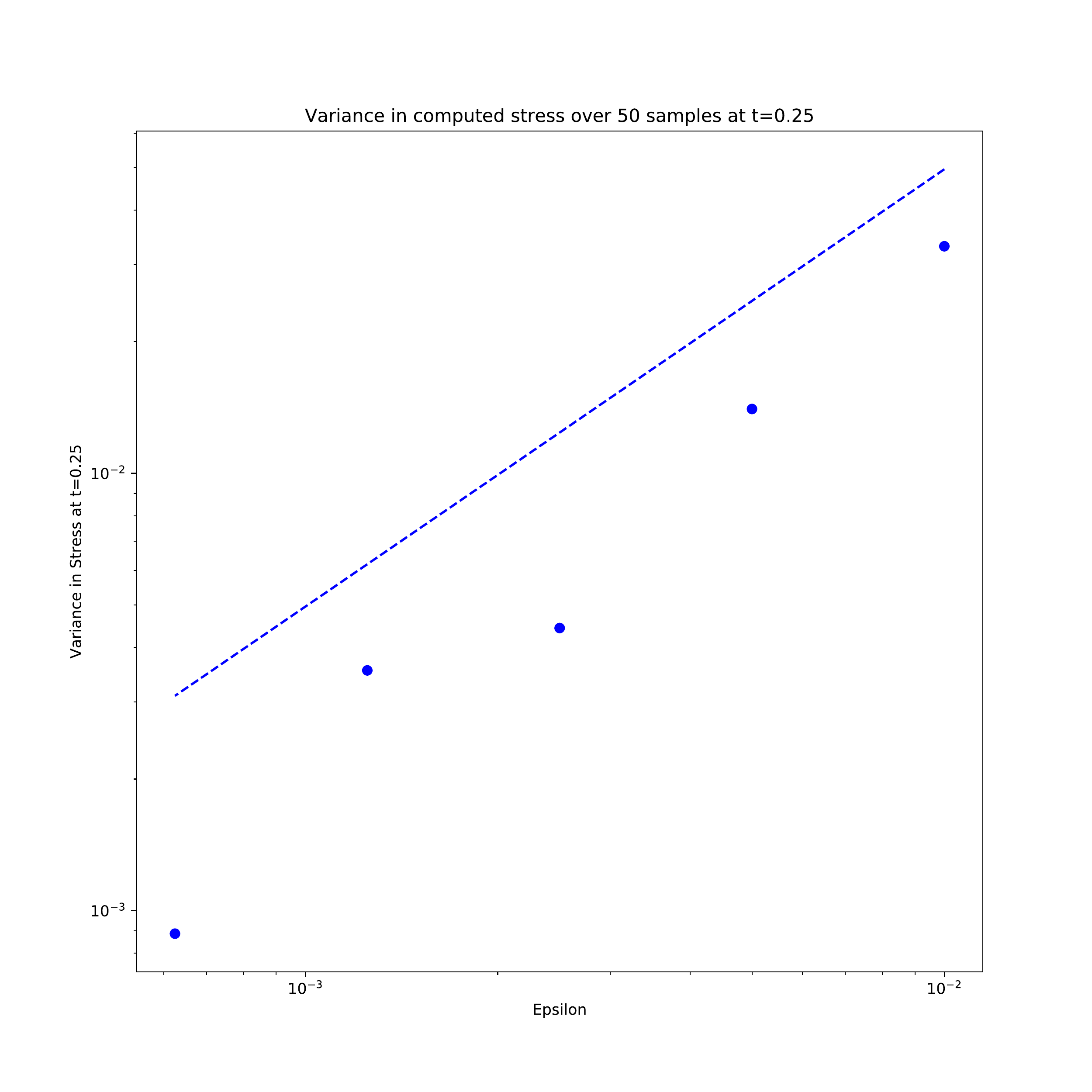}
  \caption{Log--log plot of the variance (calculated using 50 realizations) of the stress at $t=0.25$ as a function of $\eps\in\{1/100,1/200,1/400,1/800,1/1600\}$. The dashed line corresponds to a linear rate in $\eps$. \label{fig:variance}}
\end{figure}  
  
\subsection{General model}
\label{sec:general}

The numerical results of the previous section suggest that, as $\eps\to0$, the random fluctuations tend to be ``averaged'', so that we may hope that there is ultimately convergence to an underlying deterministic model which describes the limiting asymptotic behaviour. As we prove in this work, this is indeed the case. In this section, we therefore detail the precise mathematical assumptions made in order to prove our subsequent results.
% In Section~\ref{sec:examples}, we provide an example of a higher--dimensional model satisfying these assumptions.
We comment on both the applicability of these assumptions and the possibility of extending our study to other cases in Section~\ref{sec:Discussion}.
  
\subsubsection{Assumptions on randomness}
\label{sec:Randomness}

We assume that the random constitutive laws may be described in terms of random variables defined on a probability space $(\Omega,\Sigma,\Pr)$. This probability space is assumed to satisfy the key assumption that the usual Hilbert space of square--integrable random variables, $\LL^2(\Omega;\Pr)$, is \emph{separable}, i.e. contains a countable dense subset.
  
We denote $d$ the ambient physical dimension (which was taken to be $d=1$ in Sections~\ref{sec:InitialExample} and~\ref{sec:numerics}). We suppose that the space $\Omega$ is endowed with a \emph{$d$--dimensional ergodic dynamical system}, i.e. there exists a family of $\Pr$--measurable invertible maps $\{T(x):\Omega\to\Omega \, \sep \, x\in\R^d\}$ such that
\begin{enumerate}
\item $T$ is a \emph{group action} on $\Omega$ for the addition in $\R^d$, i.e. $T(x)\circ T(y)=T(x+y)$ for any $x$ and $y\in\R^d$, and $T(0)$ is the identity map;
\item $\Pr$ is an \emph{invariant measure} with respect to $T$, i.e. $\Pr[T(x)^{-1}E] = \Pr[E]$ for any $x\in\R^d$ and any $E\in\Sigma$;
\item for any $E\in\Sigma$, the set $\{(\omega,x)\in\Omega\times\R^d \ \sep \ T(x)\omega\in E\}$ is an element of the sigma--algebra generated by $\Sigma\times\mathcal{M}^d$, where $\mathcal{M}^d$ is the $d$--dimensional Lebesgue sigma--algebra;
\item $T$ is \emph{ergodic}, i.e. any set $E \in \Sigma$ such that
\begin{equation*}
  \Pr\b[\b(T(x)E\cup E\b)\setminus \b(T(x)E\cap E\b)\b]=0 \quad\text{for any $x\in\R^d$}
\end{equation*}
satisfies either $\Pr[E]=0$ or $\Pr[E]=1$.
\end{enumerate}
A function $F$ defined on $\Omega\times\R^d$ is called \emph{stationary} if there exists another function $F_0$ defined on $\Omega$ such that
\begin{equation}
  F(\omega,x) = F_0\b(T(x)\omega\b)\qquad\text{for any $x\in\R^d$ and $\Pr$--a.e. in $\Omega$}. 
  \label{eq:Stationarity}
\end{equation}
% \fl{a bit strange to write for ALL $x$ and ALMOST ALL $\omega$; it would probably be better to write ``for almost all $x$ and $\omega$''; many places to change ...}
Informally, $F$ being stationary means that the distribution of $F(\cdot,x)$ does not depend upon $x$. As an example, when such quantities are defined, we have
\begin{equation*}
  \Ex[F(\cdot,x)] = \Ex[F_0] \qquad \text{and} \qquad \mathrm{Var}[F(\cdot,x)] = \mathrm{Var}\big[F_0\big] \qquad\text{for any $x\in\R^d$}.
\end{equation*}
Stationarity and ergodicity are the main constitutive assumptions made on the ``randomness'' of the material parameters.
  
\subsubsection{Elastic constitutive law}
\label{sec:elasticity_assumptions}

We suppose that the material under consideration obeys a linear elastic constitutive law with coefficients that are random and vary on a small length scale (denoted henceforth $\eps$) relative to the size of the body $\Dom$. In particular we assume that the elastic stored energy $W:\Omega\times\R^d\times\R^d\to\R$ takes the form
\begin{equation*}
  W(\omega,x,\xi) = \frac12 \xi\cdot A(\omega,x)\xi,
\end{equation*}
where $A$ is assumed to satisfy the following properties.

\medskip

\begin{center}
  \framed{0.9\textwidth}{
    {\bf Assumptions on $A$.}
    \begin{enumerate}
    \item[$(A1)$] $A:\Omega\times\R^d\to\R^{d\times d}$ is measurable with respect to the sigma--algebra generated by $\Sigma\times\mathcal{M}^d$, and is stationary in the sense of~\eqref{eq:Stationarity}, i.e. there exists \blue{a measurable function} $A_0:\Omega\to\R^{d\times d}$ such that
      \begin{equation*}
        A(x,\omega) = A_0\big(T(x)\omega\big)\quad\text{for any $x\in\R^d$ and almost any $\omega\in\Omega$}.
      \end{equation*}
    \item[$(A2)$] $A$ is symmetric, i.e. $A(\omega,x)=A(\omega,x)^T$ almost everywhere in $\Omega \times \R^d$.
    \item[$(A3)$] There exist constants $0<\underline{A} \leq \overline{A}<+\infty$ such that, for any $\xi\in\R^d$,
      \begin{equation}
        \underline{A} \, |\xi|^2 \leq \xi\cdot A(\omega,x)\xi \leq \overline{A} \, |\xi|^2\quad\text{a.e. in $\Omega\times\R^d$}.
        \label{eq:A_bounds}
      \end{equation}
  \end{enumerate}}
\end{center}

\medskip

We note that assumptions~$(A2)$ and~$(A3)$ are equivalent to similar hypotheses on the function $A_0$ introduced in assumption~$(A1)$. To complement assumption~$(A1)$, for convenience we define $W_0(\omega,\xi) := \frac12 \xi\cdot A_0(\omega)\xi$.

The total elastic potential energy of the body associated with a displacement $y$ is assumed to be
\begin{equation*}
  \Phi^\eps_\omega[y]
  :=
  \int_\Dom W\left(\omega,\frac{x}{\eps},D_xy\right)\dx
  =
  \frac{1}{2} \int_\Dom D_xy\cdot A\left(\omega,\frac{x}{\eps}\right)D_xy\dx
  \quad\text{for any $y\in\HH^1(\Dom)$}.
\end{equation*}
We note that the bounds~\eqref{eq:A_bounds} ensure that $\Phi^\eps_\omega[y]$ is well--defined. We also note that $\Phi^\eps_\omega$ is G\^ateaux--differentiable on $\HH^1(\Dom)$, with derivative $\nabla\Phi^\eps_\omega: \HH^1(\Dom)\to \left( \HH^1(\Dom) \right)'$ given by
\begin{equation} \label{eq:nabla_Phi}
\forall u \in \HH^1(\Dom), \quad \<\nabla \Phi^\eps_\omega[y],u\>_{\HH^1(\Dom)} := \int_\Dom A\left(\omega,\frac{x}{\eps}\right)D_xy\cdot D_xu\dx.
\end{equation}
We note that $\nabla\Phi^\eps_\omega[y]$ may be thought of as the elastic force arising due to the displacement $y$, and $\dps A\left(\omega,\frac{\cdot}{\eps}\right)D_xy\in\LL^2(\Dom)^d$ as the stress field due to the elastic deformation.

\begin{remark}
Throughout this article, we adopt the terminology of elasticity (refering to displacement, strain, stress, \dots), even though our unknown function is scalar-valued. This terminology indeed provides a clearer intuition about our approach. In Section~\ref{sec:Discussion} we discuss the extension of our work to a true elastic problem.
\end{remark}
  
\subsubsection{Dissipative constitutive law}
\label{sec:dissipation_assumptions}

We suppose that energy is locally dissipated via a dissipation potential which induces forces which act to oppose local changes in strain only, and not the absolute position of the body: this is expressed as a function
\begin{equation*}
  \psi:\Omega\times\R^d\times\R^d\to\R,
\end{equation*}
where $\psi$ is assumed to satisfy the following assumptions.

\medskip

\begin{center}
  \framed{0.9\textwidth}{
  {\bf Assumptions on $\psi$.}
  \begin{enumerate}
  \item[$(\psi1)$] For any $\xi\in\R^d$, the function $(\omega,x)\in\Omega\times\R^d\mapsto\psi(\omega,x,\xi)$ is measurable with respect to the sigma--algebra generated by $\Sigma\times\Leb^d$, and is stationary in the sense of~\eqref{eq:Stationarity}, i.e. there exists $\psi_0:\Omega\times\R^d\to\R$ such that
    \begin{equation*}
      \text{for any $\xi\in\R^d$}, \quad \psi(\omega,x,\xi) = \psi_0\big(T(x)\omega,\xi\big) \quad\text{for any $x\in\R^d$ and almost every $\omega\in\Omega$}.
    \end{equation*}
  \item[$(\psi2)$] $\psi(\omega,x,\xi)\geq 0$ for any $(\omega,x,\xi) \in \Omega\times\R^d\times\R^d$, and $\psi(\omega,x,0)=0$ for any $(\omega,x)\in \Omega\times\R^d$.
  \item[$(\psi3)$] $\psi$ is uniformly strongly convex in its final variable, i.e. there exists $c>0$ such that
    \begin{equation*}
      \xi \in \R^d \mapsto \psi(\omega,x,\xi)-c|\xi|^2\quad\text{is convex for almost every $(\omega,x)\in\Omega\times\R^d$}.
    \end{equation*}
%    Without loss of generality, we can assume that $c \leq \underline{A}/2$, where $\underline{A}$ is defined in Assumption~$(A3)$. 
  \item[$(\psi4)$] There exists $C>0$ such that, for any $\xi\in\R^d$,
    \begin{equation*}
      \psi(\omega,x,\xi)\leq C(1+|\xi|^2)\quad\text{for almost every $(\omega,x)\in\Omega\times\R^d$}.
    \end{equation*}
  \end{enumerate}}
\end{center}

\medskip

As in the case of the elastic potential energy, we note that assumptions~$(\psi2)-(\psi4)$ correspond to similar hypotheses on the function $\psi_0$ introduced in assumption~$(\psi1)$.

We note that, as an immediate consequence of the strong convexity assumption~$(\psi3)$,
$$
%\begin{equation} \label{eq:ralf}
  (\xi-\xi')\cdot (v-v')\geq 2c|v-v'|^2 \quad \text{where } \xi\in\partial_\xi\psi(\omega,x,v) \text{ and } \xi'\in\partial_\xi\psi(\omega,x,v'),
%\end{equation}
$$
for almost every $(\omega,x)\in\Omega\times\R^d$ and any $v,v'\in\R^d$. The positivity assumption~$(\psi2)$ entails that
$$
%\begin{equation} \label{eq:0vel0force}
  0\in\partial_\xi\psi(\omega,x,0)\quad\text{for any $(\omega,x)\in \Omega \times \R^d$},
%\end{equation}
$$
and hence, using~$(\psi3)$, we get that, for any $(\omega,x,v) \in\Omega\times\R^d\times\R^d$,
\begin{equation}
  \psi(\omega,x,v)-c|v|^2\geq \psi(\omega,x,0)-c|0|^2+0 \cdot v=0,\quad\text{thus}\quad\psi(\omega,x,v)\geq c|v|^2.
  \label{eq:psi_lower_bound}
\end{equation}

We also note that~$(\psi4)$ entails an important bound on elements of the subdifferential of $\psi$. Suppose that $\sigma\in\partial_\xi\psi(\omega,x,\xi)$. Then, for any $\eta \in \R^d$, we have
\begin{equation*}
  \psi(\omega,x,\xi)+\sigma\cdot\eta\leq \psi(\omega,x,\xi+\eta).
\end{equation*}
Setting $\eta = t\sigma$, and using the fact that $\psi(\omega,x,\xi)\geq0$ (see~$(\psi2)$) and the upper bound assumed in~$(\psi4)$, we have
\begin{equation*}
  t|\sigma|^2\leq \psi(\omega,x,\xi+t\sigma)\leq C(1+|\xi+t\sigma|^2)\leq C(1+2|\xi|^2+2t^2|\sigma|^2).
\end{equation*}
Rearranging, and adding $C$ to the right--hand side, we obtain
\begin{equation*}
  \left(t-2Ct^2\right)|\sigma|^2 \leq 2C(1+|\xi|^2).
\end{equation*}
Choosing $t=1/4C$ and rearranging, we have that
\begin{equation}\label{eq:subdiff_bnd}
  \forall \sigma \in \partial_\xi\psi(\omega,x,\xi), \qquad |\sigma|^2\leq 16C^2\left(1+|\xi|^2\right)
\end{equation}
where $C$ is independent of $\omega$, $x$ and $\xi$ ($C$ is actually the constant appearing in~$(\psi4)$).

\medskip

The dissipation potential evaluated at a velocity field $v \in \HH^1(\Dom)$ is assumed to be
\begin{equation*}
  \Psi^\eps_\omega[v] = \int_\Dom \psi\left(\omega,\frac{x}{\eps},D_xv\right)\dx.
\end{equation*}
Using~\eqref{eq:psi_lower_bound} and assumptions~$(\psi2)$ and~$(\psi3)$, it follows that $\Psi^\eps_\omega$ is a positive strictly convex functional on $\HH^1(\Dom)$ for almost every $\omega\in\Omega$. Moreover, the subdifferential of $\Psi^\eps_\omega$ on $\HH^1(\Dom)$ may be identified as being
\begin{multline*}
  \partial\Psi^\eps_\omega[v]:=\B\{f\in\left(\HH^1(\Dom)\right)' \Bsep \<f,u\>_{\HH^1(\Dom)} = \int_\Dom \sigma\cdot D_xu\dx \ \ \text{for any $u \in \HH^1(\Dom)$},
  \\
  \text{ with } \sigma\in \LL^2(\Dom)^d \text{ and }\sigma(x)\in\partial_\xi\psi\left(\omega,\frac{x}{\eps},D_xv(x)\right)\text{ for almost every $x\in\Dom$} \B\}.
\end{multline*}
In the above definition, we may think of $\sigma$ as being a candidate for the dissipative stress which acts to oppose motion when the strain rate is $D_xv$.
  
\subsubsection{Evolution problem}
\label{sec:EvolutionProblem}

We are now in a position to formulate the evolution problem we study. We suppose that the material we consider is driven by displacement boundary conditions on $\GDir\subseteq\partial\Dom$ and undergoes a loading which is slow enough such that inertial effects may be neglected. Elastic and dissipative body forces are hence equilibrated at all times. This is equivalent to requiring that the displacement $y$ satisfies the inclusion
\begin{equation}
  0\in\partial\Psi^\eps_\omega[\dot{y}(t)]+\nabla \Phi^\eps_\omega[y(t)]\qquad \text{for almost every $t\in[0,T]$}.
  \label{eq:ForceBalance}
\end{equation}
We suppose that initially
\begin{equation}
  y(0) = \overline{y}(0),
  \label{eq:IC}
\end{equation}
and that the displacement boundary condition takes the form
\begin{equation}
  y(t)\b|_{\GDir} = \overline{y}(t)\b|_{\GDir}
  \label{eq:BC}
\end{equation}
for some function $\overline{y}(t)$ defined on $\Dom$. On the remainder $\GNeu$ of the boundary, the material is free to relax. To enforce these conditions, we decompose $y(t) = \overline{y}(t) + u(t)$ where $u$ vanishes on $\GDir$, and consider the ``lifted'' functionals $\Phi^\eps_{t,\omega}$ and $\Psi^\eps_{t,\omega}$, defined by
\begin{equation*}
  \Phi^\eps_{t,\omega}[u] := \Phi^\eps_\omega[\overline{y}(t)+u] \quad\text{and}\quad \Psi^\eps_{t,\omega}[v] := \Psi^\eps_\omega[\dot{\overline{y}}(t)+v].
\end{equation*}
The requirement that the force balance~\eqref{eq:ForceBalance} along with the initial condition~\eqref{eq:IC} and the boundary conditions~\eqref{eq:BC} are satisfied is then equivalent to seeking a function $u^\eps_\omega : [0,T]\times\Dom\to\R$ which, for almost every $\omega\in\Omega$, satisfies
\begin{equation}
  \partial \Psi^\eps_{t,\omega}[\dot{u}^\eps_\omega(t)]+\nabla \Phi^\eps_{t,\omega}[u^\eps_\omega(t)]\ni0\quad\text{for a.e. $t\in[0,T]$}, \qquad u^\eps_\omega(0) = 0, \qquad \text{$u^\eps_\omega(t)=0$ on $\GDir$}.
  \label{eq:EvolutionProblem}
\end{equation}
Our subsequent goals are to show that this evolution problem is well--posed for any fixed $\eps>0$, and then to identify a limiting problem when $\eps\to0$.

\subsubsection{A ``checkerboard'' example}
\label{sec:examples}

We now describe a specific example falling within our assumptions, which may be thought of as a randomly--coloured ``checkerboard'' of constitutive laws. We consider regularly--spaced sites, in which one of finitely--many sets of constitutive relations are satisfied, selected independently according to a probability measure $\mu$ defined on the finite state space $\mathcal{S}$. Sites are independent from each other, and in a given coordinate frame, these sites are shifted by a random vector relative to the coordinate axes. Following the construction described in Section~7.3 of~\cite{JKO94}, let $\dps Q=\left[-\frac12,\frac12\right)^d$ be the unit cube, and set
\begin{equation*}
  \Omega := \b\{\omega:\R^d\to\mathcal{S} \ \bsep \ \text{$\exists q \in Q$ such that, for any $k\in\Z^d$, $\omega(x)$ is constant on $q+k+Q$}\b\}.
\end{equation*}
Equivalently, we may identify $\omega\in \Omega$ with a vector $(q,Z)\in Q\times \mathcal{S}^{\Z^d}$, and define a sigma--algebra $\dps \Sigma=\mathcal{M}^d \otimes \bigotimes_{k\in\Z^d}\mathcal{P}(\mathcal{S})$, i.e. the sigma--algebra generated by the product of the Lebesgue sigma--algebra $\mathcal{M}^d$ on $Q$ with countably many copies of the power set $\mathcal{P}(\mathcal{S})$. We define the probability measure $\Pr$ as follows: given any $V \subset Q$ and any finite subset $\mathcal{I}\subset\Z^d$, let $U_k \subset \mathcal{S}$ for any $k \in \mathcal{I}$. Set $U_k=\mathcal{S}$ for $k\notin\mathcal{I}$. Then $\dps V \times \prod_{k\in\Z^d} U_k \subset Q \times \mathcal{S}^{\Z^d} = \Omega$, and we define
\begin{equation*}
  \Pr\bg(V \times \prod_{k\in\Z^d} U_k \bg)
  = \Leb^d(V)\prod_{k\in\mathcal{I}} \mu(U_k).
\end{equation*}
Now, using the representation of $\Omega$ as piecewise constant functions, we define the action $T(x)$ as follows:
\begin{equation*}
\forall x,s \in \R^d, \quad \b[T(x)\omega\b](s) = \omega(s+x).
\end{equation*}
It is straightforward to check that $T$ is well--defined as a bijection on $\Omega$, and that $(\Omega,\Sigma,\Pr)$ and $T$ satisfy all the properties assumed in Section~\ref{sec:Randomness}.
  
\medskip

Now that we have defined the probabilistic setting, we build stationary functions. Given $\omega\in\Omega$, we define
\begin{equation*}
  A\big((q,Z),x\big) = A_{0,Z(k)}(x+q) \quad \text{for any $x\in k+Q$, \ \ $k\in\Z^d$},
\end{equation*}
where $A_{0,i}\in\LL^\infty_\per(Q)^{d\times d}$ for each $i\in\mathcal{S}$, $\overline{A} \, |\xi|^2 \geq \xi\cdot A_{0,i}(x)\xi \geq \underline{A} \, |\xi|^2$ for any $\xi\in\R^d$, $x\in Q$ and $i\in\mathcal{S}$, and $A_{0,i}(x)$ is a symmetric matrix for any $x\in Q$ and $i\in\mathcal{S}$. Similarly, we define $\psi$ by
\begin{equation*}
  \psi\big((q,Z),x,\xi\big) = \frac12 \, \nu_{0,Z(k)}(x+q) \, |\xi|^2 + \mu_{0,Z(k)}(x+q) \, |\xi| \quad\text{for any $x\in k+Q$, \ \ $k\in\Z^d$},
\end{equation*}
where $\nu_{0,i}$ and $\mu_{0,i}$ belong to $\LL^\infty_{\rm per}(Q)$ for each $i\in\mathcal{S}$, with $\nu_{0,i}(x) \geq \underline{\nu} > 0$ and $\mu_{0,i}(x)\geq0$ for any $x \in Q$ and any $i\in\mathcal{S}$.

Both $A$ and $\psi$ are measurable and stationary in the sense of~\eqref{eq:Stationarity}, so the above construction provides an example of a model satisfying our assumptions. We note that the model described in Sections~\ref{sec:InitialExample} and~\ref{sec:numerics} is a particular case of this construction.
  
We also remark that, if we take $|\mathcal{S}|=1$, then we have, for any $x\in \R^d$,
\begin{equation*}
  A\big((q,Z),x\big) = A_0(x+q) \quad \text{and} \quad \psi\big((q,Z),x,\xi\big) = \frac12 \, \nu_0(x+q) \, |\xi|^2 + \mu_0(x+q) \, |\xi|,
\end{equation*}
for fixed $Q$--periodic functions $A_0$, $\nu_0$ and $\mu_0$. In this case, our results correspond to the case of periodic homogenization, where we average uniformly over shifts of the underlying periodic lattice.
  
\subsubsection{Discussion of assumptions}
\label{sec:Discussion}

In this section, we briefly discuss the various assumptions made above.

\medskip

\noindent
\emph{Linear elastic stress--strain relation.} The model described above assumes that $W(\omega,x,\xi)$ is quadratic with respect to $\xi$. With some adaptations, we believe that the proofs could be extended to the case where $W$ remains $C^2$ and strictly convex with respect to $\xi$, with a standard $p$--growth condition that reads as follows: there exist $\underline{A} >0$, $\overline{A}>0$ and $p\geq 2$ such that
\begin{equation*}
  \underline{A} \, \b(|\xi|^p-1\b)\leq W(\omega,x,\xi)\leq \overline{A} \, \b(|\xi|^p+1\b)\qquad\text{for any $x$ and $\xi\in\R^d$ and $\Pr$--a.e. in $\Omega$}.
\end{equation*}
More generally, we would expect the results to hold when the dissipation potential $\psi$ is strictly convex and satisfies $q$--growth conditions with $q>1$, and $W$ satisfies $p$--growth conditions with $p\geq q$.

\medskip

\noindent
\emph{Scalar displacement variable.} We assume throughout this article that $u$ is scalar--valued. Physically, the multidimensional formulation of the model therefore corresponds to a membrane or anti--plane model where only the out--of--plane displacement is taken into account, rather than to a true bulk viscoelastic problem.

Similar techniques to those used below should cover the case where $u$ is a vector--valued function, $W$ remains quadratic (i.e. $W(\omega,x,\xi) = \tfrac12 \xi:\mathbb{C}(\omega,x):\xi$ for any $\xi\in\R^{d\times d}$) as in linear elasticity, and where the dissipation potential $\psi(\omega,x,\xi)$ is again convex. The main additional steps required in such a case would be applications of Korn's inequality. Some results regarding two--scale convergence and the characterisation of the dual dissipation potential would also need to be reproved.

To correctly treat a more general nonlinear elastic problem where the stored energy density $W$ is a function of the deformation gradient and satisfies physical frame--indifference conditions, would require the assumption that $W(\omega,x,\xi)$ is polyconvex, i.e. is a convex function of the minors of $\xi\in\R^{d\times d}$. In that case, our analysis would become significantly more complex, and it is not completely obvious what sort of convexity condition on a corresponding dissipation potential $\psi$ would be sufficient to guarantee existence.

Moreover, there are examples in the literature (see e.g.~\cite{B94}) in which polyconvexity does not persist after a homogenization procedure in a static setting, so existence results for any limiting evolution are not clear. For these reasons, we have chosen to avoid these significant technical complications, and to restrict ourselves to the case of a scalar--valued function $u$.
  
\subsection{Main results}
\label{sec:main}

We are now in a position to informally state and discuss our two main results (rigorous statements of these theorems are provided at the beginning of Section~\ref{sec:wellposedness} and Section~\ref{sec:limit} respectively).

\subsubsection{General statement}

Our first result demonstrates the existence of a solution for fixed $\eps$.
  
\begin{theorem}
\label{th:eps_existence}
Under appropriate regularity hypotheses on $\overline{y}$, there exists a unique solution $u^\eps_\omega$ to~\eqref{eq:EvolutionProblem} for almost every $\omega\in\Omega$. % \fl{or for almost any $\omega$?}
\end{theorem}

\medskip
  
\noindent
It follows that the model described above is well--posed for any $\eps>0$. Our main theorem is then the following homogenization result.
  
\begin{theorem}
\label{th:main}
As $\eps\to0$, the sequence of functions $u^\eps_\omega$ solution to~\eqref{eq:EvolutionProblem} converges in an appropriate sense (which is made precise in Section~\ref{sec:stoc2scale}) to $u^\star$, and $D_xu^\eps_\omega$ converges to $D_xu^\star+\theta$, where $u^\star$ is independent of $\omega\in\Omega$, and $\theta$ is a time--dependent vector field which depends on both $x$ and $\omega$, and additionally satisfies $\Ex[\theta]=0$. Moreover, for almost every $t\in[0,T]$, the pair $(u^\star,\theta)$ is the unique solution to the system of inclusions
\begin{gather}
  -\div_x\bg[\int_\Omega \partial_\xi\psi_0\left(\omega,D_x\dot{\overline{y}}+D_x\dot{u}^\star+\dot{\theta}\right) + D_\xi W_0\b(\omega,D_x\overline{y}+D_xu^\star+\theta\b)\dPr(\omega)\bg]\ni0,
  \label{eq:homogenised_eqn}
\\
  -\div_\omega \left[\partial_\xi\psi_0\left(\omega,D_x\dot{\overline{y}}+D_x\dot{u}^\star +\dot\theta\right)+D_\xi W_0(\omega,D_x\overline{y}+D_xu^\star+\theta)\right]\ni 0,
  \label{eq:corrector_eqn}
\end{gather}
with boundary and initial conditions
\begin{equation*}
u^\star|_{\GDir}=0, \qquad u^\star(0)=0 \qquad\text{and} \qquad \theta(0)=0,
\end{equation*}
in an appropriate functional space, detailed in Section~\ref{sec:limit}.
\end{theorem}

\medskip

\blue{Equations~\eqref{eq:homogenised_eqn} and~\eqref{eq:corrector_eqn} should be understood as follows: there exists some vector valued function $G^\star$, which depends on $(t,\omega,x)$ and satisfies $G^\star(t,\omega,x) \in \partial_\xi\psi_0\left(\omega,D_x\dot{\overline{y}}+D_x\dot{u}^\star +\dot\theta\right)$ for almost any $(t,\omega,x)$, such that
\begin{gather}
  -\div_x\bg[\int_\Omega G^\star(t,\omega,x) + D_\xi W_0\b(\omega,D_x\overline{y}+D_xu^\star+\theta\b)\dPr(\omega)\bg] = 0,
  \label{eq:homogenised_eqn_bis}
\\
-\div_\omega \left[G^\star(t,\omega,x) + D_\xi W_0(\omega,D_x\overline{y}+D_xu^\star+\theta)\right] = 0.
\label{eq:corrector_eqn_bis}
\end{gather}
}
% \fl{Tom, I added this because I wanted to make clear that the element of $\partial_\xi\psi_0\left(\omega,D_x\dot{\overline{y}}+D_x\dot{u}^\star +\dot\theta\right)$ which appears in~\eqref{eq:homogenised_eqn} and~\eqref{eq:corrector_eqn} is the same; I also added something similar after the rigorous statement in Section~\ref{sec:limit}.}
Equation~\eqref{eq:homogenised_eqn} may be interpreted as a macroscopic force balance, where the total stress is computed as an expectation over random variations. Equation~\eqref{eq:corrector_eqn} may be interpreted as a corrector problem, corresponding to a microscopic force balance. In general, \eqref{eq:homogenised_eqn} and~\eqref{eq:corrector_eqn} are coupled since $D_x\overline{y}$, $D_xu^\star$ and $\theta$ all depend upon $x\in\Dom$ and the time $t$.

\medskip

At this stage, we have yet to define the operator $-\div_\omega$ appearing in~\eqref{eq:corrector_eqn}. Since we have not assumed any topology on the probability space $\Omega$, the definition is not obvious. On the other hand, \red{this operator (along with other stochastic differentiation operators) has a familiar interpretation as a derivative} with respect to a periodic variable when considering the form of periodic homogenization covered by our result. To help elucidate our main result, in Section~\ref{sec:periodic} we provide an example of the result in the periodic setting previously described in Section~\ref{sec:examples}, and we consider the one-dimensional case in Section~\ref{sec:oneD_fred}. We postpone a precise definition of $-\div_\omega$ (and of other differentiation operators) in the fully stochastic case until Section~\ref{sec:stoc_setup}.

\subsubsection{A periodic ``checkerboard'' example}
\label{sec:periodic}

To provide some intuition about the result stated in Theorem~\ref{th:main}, we return to the checkerboard example considered in Section~\ref{sec:examples}. As noted at the end of that section, if $\#\mathcal{S}=1$, then we claim that our result in this case corresponds to periodic homogenization with a random shift of the coordinate frame. To see this, we note that the mapping $\omega=(q,Z)\mapsto q$ is a bijective measure--preserving map from $(\Omega,\Sigma,\Pr)$ to $(Q,\mathcal{M}^d,\Leb^d)$, so we may identify $\omega\in\Omega$ with $q\in Q$. In this case, we may write, for any $x\in \R^d$, that
\begin{equation*}
  \psi(q,x,\xi) = \frac12 \, \nu_0(x+q) \, |\xi|^2 + \mu_0(x+q) \, |\xi| \qquad\text{and}\qquad W(q,x,\xi) = \frac12 \, \xi\cdot A_0(x+q)\xi,
\end{equation*}
where we recall that in this case $\nu_0$, $\mu_0$ and $A_0$ are all $Q$-periodic functions. Note in passing that, in this case, the function $D_\sigma\psi^*$, which is useful below, has a simple expression:
\begin{equation*}
  D_\sigma\psi^*(q,x,\sigma)=
  \begin{cases}
    0 & \text{when $|\sigma|\leq\mu_0(x+q)$},
    \\
    \displaystyle \frac{|\sigma|-\mu_0(x+q)}{\nu_0(x+q)} \, \red{\frac{\sigma}{|\sigma|}} & \text{when $|\sigma|\geq\mu_0(x+q)$}.
  \end{cases}
\end{equation*}
We now write Theorem~\ref{th:main} in this case. We set $y^\star=\overline{y}+u^\star$, where $u^\star$ is the homogenized solution introduced in Theorem~\ref{th:main}. We note that, in the periodic setting, the corrector vector field takes the form of the gradient of a periodic function, $\theta=D_qu^1$ (for an explanation, refer to the discussion at the end of Section~\ref{sec:stress_and_strain}). We fix $u^1$ by further imposing that $\dps \int_Q u^1(t,q,x) \, \dd q=0$ for any $t$ and $x$. We may then write the corrector equation~\eqref{eq:corrector_eqn} as
\begin{equation*}
  -\div_q \B[\partial_\xi\psi_0(q,D_x\dot{y}^\star+D_q\dot{u}^1)+A_0(q)\b[D_xy^\star+D_q u^1\b]\B]\ni 0,
\end{equation*}
which should be interpreted as asserting that there exists a stress field $\Sigma(t,q,x)$ which is weakly divergence--free in $q$ such that
\begin{equation}
  \label{eq:micro_stress_bal}
  \partial_\xi\psi_0\b(q,D_x\dot{y}^\star(t,x) +D_q\dot{u}^1(t,q,x)\b)\ni\Sigma(t,q,x) - A_0(q)\b[D_xy^\star(t,x)+D_q u^1(t,q,x)\b].
\end{equation}
This equation should be viewed as a balance of stresses, and $\Sigma$ is an unknown of the problem.
  
Applying the properties~\eqref{eq:LFEquivalence} of the Legendre--Fenchel transform, we recast the above equation as the equivalent rate equation
\begin{equation*}
  D_q\dot{u}^1(t,q,x)=D_\sigma\psi^*_0\B(q,\Sigma(t,q,x) -A_0(q)\b[D_xy^\star(t,x)+D_q u^1(t,q,x)\b]\B)-D_x\dot{y}^\star(t,x)
\end{equation*}
subject to the conditions
$$
\int_Q u^1(t,q,x) \, \dd q=0 \qquad \text{and} \qquad u^1(0,q,x)=0,
$$
the second condition stemming from the fact that $D_q u^1(0,q,x) = \theta(0,q,x) = 0$. We may also integrate~\eqref{eq:micro_stress_bal} over $q \in Q$ to obtain
\begin{multline}
  \left\{ \int_Q \xi(t,q,x) \dq \ \middle| \ \xi(t,q,x) \in \partial_\xi\psi_0\b(q,D_x\dot{y}^\star(t,x) + D_q\dot{u}^1(t,q,x)\b) \ \ \text{ a.e. $(t,q,x)\in[0,T]\times Q\times\Dom$} \right\}
  \\
  \ni \int_Q \Sigma(t,q,x) \dq - \int_Q A_0(q)\b[D_xy^\star(t,x)+D_q u^1(t,q,x)\b]\dq,
  \label{eq:mean_stress_bal}
\end{multline}
an equation which is useful below.

We now turn to the macroscopic equation~\eqref{eq:homogenised_eqn}. To interpret it in this case, we define the homogenized potential
\begin{equation}\label{eq:psibar}
  \overline{\psi}(t,x,\xi):=\int_Q\psi_0\b(q,\xi+D_q\dot{u}^1(t,q,x)\b)\,\dd q.
\end{equation}
Setting $\dps A_{\rm mean}=\int_Q A_0(q)\dq$, the homogenized force balance~\eqref{eq:homogenised_eqn} should then be read as requiring that there exists a weakly divergence--free stress field $\sigma^\star(t,x)$ such that 
\begin{equation} \label{eq:enpc}
  \partial_\xi\overline{\psi}\b(t,x,D_x\dot{y}^\star(t,x)\b) \ni \sigma^\star(t,x) - A_{\rm mean} \, D_xy^\star(t,x) - \int_Q A_0(q) D_q u^1(t,q,x) \, \dd q.
\end{equation}
Noting the definition~\eqref{eq:psibar} of $\overline{\psi}$ and the relation~\eqref{eq:mean_stress_bal}, it may be inferred that the \red{field $\sigma^\star$ can be chosen to} satisfy
\begin{equation*}%\label{eq:period_stress}
  \sigma^\star(t,x) = \int_Q \Sigma(t,q,x) \dq,
\end{equation*}
so that \red{in this case} $\sigma^\star$ can be interpreted as the average stress at point $x$ and at time $t$.
  
Invoking~\eqref{eq:LFEquivalence}, we recast~\eqref{eq:enpc} as the rate equation
\begin{gather*}
  %\label{eq:period_homogenized}
  D_x\dot{y}^\star(t,x) = D_\sigma\big(\overline{\psi}\big)^*\bigg(t,x,\sigma^\star(t,x) - A_{\rm mean} \, D_xy^\star(t,x) - \int_Q A_0(q) D_q u^1(t,q,x) \, \dd q \bigg), \\
  \text{subject to} \quad y^\star(t)\b|_{\GDir} = \overline{y}(t) \qquad \text{and} \qquad y^\star(0) = \overline{y}(0).
\end{gather*}
\red{We note that, in the first equality above, we know that the subdifferential of $\big(\overline{\psi}\big)^*$ is single--valued thanks to Theorem~26.3 in~\cite{Rockafellar}.} In summary, we have obtained the following corollary of Theorem~\ref{th:main} in this periodic setting.

\begin{corollary}
In the case of the model described in Section~\ref{sec:examples} with $\#\mathcal{S}=1$, the sequence of solutions $u^\eps$ to~\eqref{eq:EvolutionProblem} converges in an appropriate sense to $u^\star$, and $D_xu^\eps$ converges to $D_xu^\star+D_qu^1$, where $u^\star(t,x)$ and $u^1(t,q,x)$ solve the equations
\begin{align*}
  D_q\dot{u}^1 &= D_\sigma\psi^*_0\left(q,\Sigma - A_0(q) \b[D_x\overline{y}+D_xu^\star+D_q u^1\b]\right)-D_x\dot{\overline{y}}-D_x\dot{u}^\star,
  \\
  D_x\dot{u}^\star &= D_\sigma\big(\overline{\psi}\big)^*\left(t,x,\sigma^\star-A_{\rm mean} (D_x\overline{y}+D_xu^\star)-\int_Q A_0(q) D_q u^1 \dq \right)-D_x\dot{\overline{y}},
\end{align*}
with
$$
\int_Q u^1(t,q,x) \, \dd q=0, \qquad u^1(0,q,x)=0, \qquad u^\star(t)\b|_{\GDir} = 0 \qquad \text{and} \qquad u^\star(0) = 0,
$$
where the stress fields $\sigma^\star$ and $\Sigma$ are divergence-free (respectively in $x$ and $q$), satisfy $\dps \sigma^\star(t,x)=\int_Q \Sigma(t,q,x)\dq$, and are chosen in order to ensure that the boundary and initial conditions for $u^\star$ and $u^1$ are satisfied.
\end{corollary}

\medskip

One of the key observations which arises from this example is that, even in this simple case (which can even be simplified further by choosing $\psi(q,x,\xi)= |\xi|^2/2$), the dissipation potential in the limit typically depends on the entire history of loading: $\overline{\psi}$ is indeed defined using a time--dependent corrector. This demonstrates that, while the solutions converge to a well--defined limit when $\eps \to 0$, computing approximations to the limit problem is very challenging.

\subsubsection{The one--dimensional case}
\label{sec:oneD_fred}

We consider here the one--dimensional case, and further assume, for the sake of simplicity, that $\dps \psi\big(\omega,x,\xi\big) = \frac12 \, \nu(\omega,x) \, |\xi|^2$ for some random stationary coefficient $\nu$ satisfying $0<\underline{\nu} \leq \nu(\omega,x) \leq \overline{\nu}$ for any $(\omega,x)\in\Omega\times\R$. We recall from the definition of stationarity given in~\eqref{eq:Stationarity} that there exists $\nu_0:\Omega\to\R$ such that $\nu(\omega,x) = \nu_0\b(T(x)\omega\b)$. We are going to explicitly compute the solution $u^\eps$ to the evolution equation~\eqref{eq:EvolutionProblem}, pass to the limit $\eps \to 0$ on that explicit formula, and check that the so-obtained function $u^\star$ is the unique solution to the homogenized problem~\eqref{eq:homogenised_eqn}--\eqref{eq:corrector_eqn}.

\medskip

The evolution equation~\eqref{eq:EvolutionProblem} reads
$$
\forall v \in \HBC, \quad \int_\Dom \nu\Big(\omega,\frac{x}{\eps}\Big) \, D_x \dot{y}^\eps \, D_x v + \int_\Dom A\Big(\omega,\frac{x}{\eps}\Big) \, D_x y^\eps \, D_x v = 0,
$$
where $y^\eps = \overline{y}+u^\eps$. This implies that there exists some function $\sigma^\eps$, depending on $t$ and $\omega$ but independent of $x$, such that
\begin{equation} \label{eq:seattle0}
\nu\Big(\omega,\frac{x}{\eps}\Big) \, D_x \dot{y}^\eps + A\Big(\omega,\frac{x}\eps\Big) \, D_x y^\eps = \sigma^\eps,
\end{equation}
which is exactly~\eqref{eq:ForceInclusion}. Dividing by $\nu(\omega,x/\eps)$, we obtain, for any $x \in \Dom$, an ordinary differential equation that we can integrate. Using that the initial condition is independent from $\eps$ and $\omega$, we find that
\begin{equation} \label{eq:seattle}
D_x y^\eps(t,\omega,x) = D_x y(0,x) \exp\left(-\frac{A(\omega,x/\eps)}{\nu(\omega,x/\eps)} \, t\right) + \int_0^t \frac{\sigma^\eps(s,\omega)}{\nu(\omega,x/\eps)} \, \exp\left(\frac{A(\omega,x/\eps)}{\nu(\omega,x/\eps)} \, \red{(s-t)} \right) \, ds.
\end{equation}
Using the boundary condition $\dps \int_\Dom D_x y^\eps(t,\omega,x) \, dx = \ell(t)$ for some $\ell$ independent of $\eps$ and $\omega$, we get, by integrating~\eqref{eq:seattle} on $\Dom$, that
\begin{equation} \label{eq:seattle2}
\ell(t) = f^\eps(t,\omega) + \int_0^t \sigma^\eps(s,\omega) \, g^\eps(\red{s-t},\omega) \, ds
\end{equation}
with
$$
  f^\eps(t,\omega) = \int_\Dom D_x y(0,x) \exp\left(-\frac{A(\omega,x/\eps)}{\nu(\omega,x/\eps)} \, t\right) dx,
  \qquad
  g^\eps(s,\omega) = \int_\Dom \frac{1}{\nu(\omega,x/\eps)} \, \exp\left(\frac{A(\omega,x/\eps)}{\nu(\omega,x/\eps)} \, s \right) \, dx.
$$
We prove in Section~\ref{sec:boundedness} below (in the general case) that $y^\eps(t,\omega,\cdot)$ and $\dot{y}^\eps(t,\omega,\cdot)$ are bounded in $H^1(\Dom)$ by constants independent of $t \in [0,T]$ and $\omega \in \Omega$. They thus weakly converge (up to a subsequence extraction) to $y^\star(t,\omega,\cdot)$ and $\dot{y}^\star(t,\omega,\cdot)$ respectively, \blue{for almost every $(t,\omega)\in[0,T]\times\Omega$}. The relation~\eqref{eq:seattle0} and the bounds on $y^\eps$ and $\dot{y}^\eps$ mentioned above show that, \blue{again for almost every $\omega$, the sequence $\sigma^\eps(\cdot,\omega)$ is bounded in $\LL^2(0,T)$: it thus weakly converges (up to a subsequence extraction) to $\sigma^\star(\cdot,\omega)$.}

We now pass to the limit $\eps \to 0$ in~\eqref{eq:seattle}. Using the ergodic theorem, we obtain that
\begin{equation} \label{eq:seattle_hom}
D_x y^\star(t,\omega,x) = D_x y(0,x) \, h^\star(t) + \int_0^t \sigma^\star(s,\omega) \, g^\star(\red{s-t}) \, ds
\end{equation}
\blue{for almost every $(t,\omega, x)$,} with
$$
  h^\star(t) = \Ex \left[ \exp\left(-\frac{A_0(\omega)}{\nu_0(\omega)} \, t\right) \right],
  \qquad
  g^\star(s) = \Ex \left[ \frac{1}{\nu_0(\omega)} \, \exp\left(\frac{A_0(\omega)}{\nu_0(\omega)} \, s \right) \right].
$$
Passing to the limit $\eps \to 0$ in~\eqref{eq:seattle2}, we also obtain
\begin{equation} \label{eq:seattle2_hom}
\ell(t) = \int_\Dom D_x y(0,x) \, h^\star(t) \, dx + |\Dom| \int_0^t \sigma^\star(s,\omega) \, g^\star(\red{s-t}) \, ds.
\end{equation}
Collecting~\eqref{eq:seattle_hom} and~\eqref{eq:seattle2_hom}, we deduce that
\begin{equation} \label{eq:seattle_hom_fin}
D_x y^\star(t,\omega,x) = \Big( D_x y(0,x) - |\Dom|^{-1} \ell(0) \Big) h^\star(t) + |\Dom|^{-1} \ell(t),
\end{equation}
which shows that $D_x y^\star$ is indeed deterministic.

\medskip

We now consider the homogenized problem~\eqref{eq:homogenised_eqn}--\eqref{eq:corrector_eqn}. \red{It turns out that, in the one--dimensional case, the vanishing of the stochastic divergence in~\eqref{eq:corrector_eqn} is analogous with the vanishing of a usual derivative, allowing} us to deduce that there exists a function $\Sigma$, independent of $\omega$, such that
$$
\nu_0(\omega) \left( D_x\dot{y}^\star + \dot\theta \right) + A_0(\omega) \left( D_xy^\star + \theta \right) = \Sigma(t,x).
$$
Equation~\eqref{eq:homogenised_eqn} reads $\dps 0 = \div_x \Ex[ \Sigma(t,x) ] = \div_x \Sigma(t,x)$, hence $\Sigma$ only depends on time. Integrating the above ordinary differential equation and using that $\theta(t=0,\omega,x)=0$, we obtain that
$$
D_x y^\star(t,x) + \theta(t,\omega,x) = D_x y^\star(0,x) \exp\left(-\frac{A_0(\omega)}{\nu_0(\omega)} \, t\right) + \int_0^t \frac{\Sigma(s)}{\nu_0(\omega)} \, \exp\left(\frac{A_0(\omega)}{\nu_0(\omega)} \, (\red{s-t}) \right) \, ds.
$$
Taking the expectation and using that $\Ex[\theta] = 0$, we deduce that
\begin{equation} \label{eq:seattle_5}
D_x y^\star(t,x) = D_x y^\star(0,x) \, h^\star(t) + \int_0^t \Sigma(s) \, g^\star(\red{s-t}) \, ds.
\end{equation}
Integrating~\eqref{eq:seattle_5} over $\Dom$ and using the boundary conditions, we deduce that
\begin{equation} \label{eq:seattle_6}
\ell(t) = \ell(0) \, h^\star(t) + |\Dom| \int_0^t \Sigma(s) \, g^\star(\red{s-t}) \, ds.
\end{equation}
Collecting~\eqref{eq:seattle_5} and~\eqref{eq:seattle_6}, we recover~\eqref{eq:seattle_hom_fin}.

%\fl{dans cette preuve, on confond la limite a deux echelles et la limite faible; le theo general parle de limite a deux echelles, et ce qu'on calcule dans cet exemple 1D, c'est la limite faible. On decide de garder cette confusion sous silence.}

%\fl{les deux limites peuvent etre differentes; par ex, une sequence aleatoire independante de epsilon cv faible vers elle meme, et cv a deux echelles vers son esperance.}

%\fl{Dans le cas precis qui nous interesse, la limite faible est deterministe -- c'est ce que le calcul explicite dit ici -- et la limite a deux echelles est aussi deterministe -- c'est ce que dit le theoreme general. Et quand ces deux limites sont deterministes, alors c'est facile de montrer que ce sont forcement les memes; prendre psi deterministe dans la def de la cv a deux echelles comme premiere etape pour montrer ceci.}

\subsection{Hysteretic behaviour}
\label{sec:hysteresis}

We conclude this introductory section by showing that, in a particular instance of our model, the system exhibits hysteretic behaviour which persists at asymptotically low strain rates, as observed numerically on Figure~\ref{fig:rates}.
  
As in Section~\ref{sec:examples}, we suppose here that
\begin{equation*}
  \psi\left(\omega,\frac{x}{\eps},\xi\right) = \frac12 \, \nu\left(\omega,\frac{x}{\eps}\right)|\xi|^2 + \mu\left(\omega,\frac{x}{\eps}\right) |\xi|,
\end{equation*}
where $\mu,\nu:\Omega\times\R^d\to\R$ are random stationary coefficients which additionally satisfy
\begin{equation*}
  0<\underline{\nu} \leq \nu(\omega,x)\leq \overline{\nu}\quad\text{and}\quad 0< \underline{\mu}\leq \mu(\omega,x)\leq \overline{\mu}\quad \text{for any $(\omega,x)\in\Omega\times\R^d$}.
\end{equation*}
As a reference domain, we choose $\Dom = \prod_{i=1}^d(a_i,b_i)$. We prescribe the boundary data by defining $\overline{y}_\delta(t,x) := \ell(\delta \, t) \, x_1$, where $\ell$ is a 1--periodic $\CC^1$ function and $\delta$ is a parameter that scales the time (the smaller $\delta$ is, the slower the loading is). We impose a Dirichlet boundary condition on $\GDir = \left( \{a_1\} \times \prod_{i=2}^d(a_i,b_i) \right) \cup \left( \{b_1\} \times \prod_{i=2}^d(a_i,b_i) \right)$, and let the material free to relax on the remainder of the boundary.
  
For fixed $\eps$ and $\omega$, the energy dissipated over one time period by the system driven by the boundary condition $y_\delta|_\GDir=\overline{y}_\delta|_\GDir$ may be expressed as
\begin{align*}
  {\cal E}^\eps_\omega = \text{Energy Dissipated}
  &=\int_0^{1/\delta}\int_\Dom\partial_\xi\psi\left(\omega,\frac{x}{\eps},D_x\dot{y}_\delta\right)\cdot D_x\dot{y}_\delta\dx\dt
  \\
  &=\int_0^{1/\delta}\int_\Dom \nu\left(\omega,\frac{x}{\eps}\right) |D_x\dot{y}_\delta|^2+\mu\left(\omega,\frac{x}{\eps}\right) |D_x\dot{y}_\delta|\dx\dt
  \\
  &\geq\int_0^{1/\delta}\int_\Dom \nu\left(\omega,\frac{x}{\eps}\right)|D_x\dot{y}_\delta\cdot e_1|^2 +\mu\left(\omega,\frac{x}{\eps}\right) |D_x\dot{y}_\delta\cdot e_1|\dx\dt.
\end{align*}
Using the lower bounds on $\mu$ and $\nu$ and applying Jensen's inequality, the boundary conditions and a change of time variable, we find that
\begin{align*}
  \int_0^{1/\delta} \int_\Dom \nu \, |D_x\dot{y}_\delta \cdot e_1|^2 \dx\dt
  &\geq
  \int_0^{1/\delta}\frac{\underline{\nu}}{\Leb^d(\Dom)} \bg| \int_\Dom D_x\dot{y}_\delta \cdot e_1 \, \dx\bg|^2\dt\\
  &=
  \int_0^{1/\delta}\frac{\underline{\nu}}{\Leb^d(\Dom)} \bg| \int_\Dom \left[ \delta \, \dot{\ell}(\delta \, t) + D_x\dot{u}_\delta \cdot e_1 \right] \dx\bg|^2\dt\\
  &=
  \Leb^d(\Dom) \, \underline{\nu} \, \delta \int_0^1 \left| \dot{\ell}(t) \right|^2\dt,
\end{align*}
where $\Leb^d(\Dom)$ denotes the volume of the domain $\Dom$, and likewise
\begin{equation*}
  \int_0^{1/\delta} \int_\Dom \mu \, |D_x\dot{y}_\delta \cdot e_1| \dx\dt
  \geq
  \int_0^{1/\delta} \underline{\mu}\,\bg| \int_\Dom \left[ \delta \, \dot{\ell}(\delta \, t) + D_x \dot{u}_\delta \cdot e_1 \right] \dx\bg|\dt
  =
  \Leb^d(\Dom) \, \underline{\mu} \int_0^1 \left| \dot{\ell}(t) \right| \dt.
\end{equation*}
Noting that the second lower bound is independent of the rate $\delta$, we obtain that
\begin{equation}
  \label{eq:bound_hyst}
{\cal E}^\eps_\omega = \text{Energy dissipated in one period} \geq \Leb^d(\Dom) \, \underline{\mu} \int_0^1 \left| \dot{\ell}(t) \right| \dt,
\end{equation}
showing that no matter how slowly the material is loaded and unloaded, the material always dissipates energy (note also that since $\mu$ has units of energy per unit volume, this lower bound has units of energy, as expected). We note that this lower bound is independent of $\eps$ and $\omega$. For the homogenized model we identify, similar results can be indeed obtained showing that the above property is preserved in the limit $\eps\to0$.

\medskip

We now discuss the link between that dissipated energy and the fact that the system shows hysteresis, namely that the area within the loops shown on Figure~\ref{fig:rates} is positive. We consider here the one--dimensional setting of Section~\ref{sec:InitialExample}, and write the mechanical energy at time $t$ as
$$
E^\eps_\omega(t) = \int_\Dom W\left(\omega,\frac{x}{\eps},D_xy(t,x)\right)\dx.
$$
Then, over one period, we compute
$$
E^\eps_\omega(1/\delta) - E^\eps_\omega(0)
=
\int_0^{1/\delta} \frac{\dd E^\eps_\omega}{\dt}\dt
=
\int_0^{1/\delta} \int_\Dom A\left(\omega,\frac{x}{\eps}\right)D_x y\cdot D_x \dot{y}\,\dx\dt.
$$
Using the equilibrium equation~\eqref{eq:ForceInclusion}, we deduce that
$$
E^\eps_\omega(1/\delta) - E^\eps_\omega(0)
\in
\int_0^{1/\delta} \int_\Dom \Big( \overline{\sigma}(t,\omega) - \partial_\xi\psi\left(\omega,\frac{x}{\eps},D_x \dot{y}\right) \Big) D_x \dot{y}\,\dx\dt,
$$
hence, using the specific expression of $\psi$,
$$
E^\eps_\omega(1/\delta) - E^\eps_\omega(0)
=
\int_0^{1/\delta} \overline{\sigma}(t,\omega) \left( \int_\Dom D_x \dot{y} \dx\right) \, \dt - {\cal E}^\eps_\omega.
$$
An application of Stokes Theorem demonstrates that the first term in the above right-hand side is the area within the loop, whereas the last term is the dissipated energy. If the system reaches a limit cycle, we expect the energy difference on the left--hand side to tend to zero over successive periods (indeed, this is what we observe numerically). Under the assumption that we have found such a periodic solution, and that the above left--hand side is therefore $0$,  we find that
$$
\text{Area within the loop} = \text{Energy dissipated in one period}.
$$
Since the energy dissipated in one period satisfies~\eqref{eq:bound_hyst}, we obtain that the area within the loop remains bounded away from 0 for any loading rate $\delta$. This demonstrates that the model exhibits stress--strain hysteresis which persists at arbitrarily slow rates of loading.

\section{Functional analytic setting}
\label{sec:setup}

To mathematically cast the model we study in a correct manner, we first introduce spaces of admissible displacements. Since we consider an evolution problem, we first define the ``spatial'' function spaces and next describe their evolutionary counterparts.
  
\subsection{Space of displacements}
\label{sec:Deformations}

We consider a Lipschitz domain $\Dom\subset\R^d$ which corresponds to the reference configuration of a $d$--dimensional viscoelastic body. The boundary $\partial\Dom$ of the domain is partitioned into relatively open sets
\begin{equation*}
  \partial \Dom = \GDir\cup\GNeu \quad\text{with}\quad \GDir\cap\GNeu=\emptyset \quad\text{and}\quad \GDir \neq\emptyset.
\end{equation*}
The set $\GDir$ is the portion of the boundary subject to Dirichlet conditions and $\GNeu$ is the portion left free, and is therefore stress--free under natural boundary conditions. We consider scalar--valued functions $y:\Dom\to\R$ which correspond to displacements experienced by this body. The weak gradient of a function $y$ is denoted $D_xy:\Dom\to\R^d$. We assume that $\GDir$ has non--zero capacity and set
\begin{equation*}
  \HBC := \overline{\CC^1_0(\Dom\cup \GNeu)}^{\HH^1(\Dom)},
\end{equation*}
where $\CC^1_0(\Dom\cup\GNeu)$ is the space of real--valued continuously differentiable functions which are compactly supported in $\Dom\cup\GNeu$. We write $\HBC$ in place of $\HBC(\Dom)$, since $\Dom$ is fixed throughout.
%% on autorise le support a toucher Gamma_N: c'est correct, sur Gamma_N, on ne veut pas demander que les fonctions soient nulles. Par contre, le support ne peut pas toucher Gamma_D (qui est hors de D et de Gamma_N), et donc les fonctions sont nulles sur Gamma_D. 
  
Our assumption that $\GDir$ has non--zero capacity entails that the Poincar\'e inequality holds: there exists a constant $C>0$ depending only on $\Dom$ and $\GDir$ such that
\begin{equation} \label{eq:poincare}
  \|u\|_{\LL^2(\Dom)}\leq C\|D_xu\|_{\LL^2(\Dom)}\quad\text{for any $u\in\HBC$}.
\end{equation}
It follows that the mapping $u\mapsto\|u\|_{\HBC}:=\|D_xu\|_{\LL^2(\Dom)}$ is a norm on $\HBC$ equivalent to the restriction of the $\HH^1(\Dom)$ norm to $\HBC$. Furthermore
\begin{equation*}
  (u,v)_{\HBC} := \int_\Dom D_xu\cdot D_xv\dx\quad\text{for any $u,v\in\HBC$}
\end{equation*}
defines an inner product which induces the norm $\|\cdot\|_{\HBC}$.
  
We write $\HBCPrime$ to denote the space of bounded linear functionals acting on $\HBC$, and denote by $\<\cdot,\cdot\>_{\HBC}:\HBCPrime\times\HBC\to\R$ the corresponding duality product. Whenever there exists $\sigma\in \LL^2(\Dom)^d$ such that $f\in\HBCPrime$ may be represented as
\begin{equation*}
  \<f,u\>_{\HBC} = \int_\Dom \sigma \cdot D_xu\dx\quad\text{for any $u\in\HBC$},
\end{equation*}
we write $f =-\div_x\sigma$. Upon applying the Riesz Representation Theorem, we note that any $f\in\HBCPrime$ has such a representation, although without further conditions this is non--unique.
  
\subsection{Stochastic displacement space}
\label{sec:stoc_setup}

Our aim is to study the evolution problem~\eqref{eq:EvolutionProblem} and to identify its homogenized limit when $\eps \to 0$. To that end, we use a variant of the theory of stochastic two--scale convergence in the mean introduced in~\cite{BMW94}. We recall here the notion of stochastic weak derivatives, following the exposition in Section~2 of~\cite{BMW94}.
  
We consider the family of unitary operators $\{U(x)\}_{x\in\R^d}$ acting on $f\in\LL^2(\Omega)$ via
\begin{equation*}
  [U(x)f](\omega) = f\b(T(x)\omega\b),
\end{equation*}
where $T(x)$ is the ergodic dynamical system defined in Section~\ref{sec:Randomness}. Since $\LL^2(\Omega)$ is assumed to be separable, it is possible to define \emph{stochastic partial derivatives} $\delta_1,\ldots,\delta_d$ as the infinitesimal generators $\delta_j$ of the strongly continuous unitary group representations, where $x_j$ varies around 0 while the other coordinates remain fixed at $0$ (see Eq.~(2.1) of~\cite{BMW94}). More precisely, for any $f\in\LL^2(\Omega)$ such that the limit makes sense, we set
$$
\delta_j f = \lim_{h\to 0} \frac{U(h\mathrm{e}_j)f-f}{h},
$$
where $\mathrm{e}_j\in\R^d$ is the $j$th Euclidean basis vector, and the limit is taken in $\LL^2(\Omega)$. The operators $i\delta_j$ are self--adjoint and commute on their joint domain of definition, denoted $\mathscr{D}(\Omega)$. For a multi--index $\alpha\in\N^d$, we define the operator
\begin{equation*}
  \delta^\alpha := \delta_1^{\alpha_1} \, \delta_2^{\alpha_2} \dots \delta_d^{\alpha_d}.
\end{equation*}
We define the space $\mathscr{D}^\infty(\Omega)$ of ``test functions'' as
\begin{equation*}
  \mathscr{D}^\infty(\Omega):=\b\{\phi\in\LL^\infty(\Omega) \, \bsep \, \delta^\alpha \phi \in\LL^\infty(\Omega)\cap\mathscr{D}(\Omega) \text{ for all $\alpha\in\N^d$} \b\}.
\end{equation*}
The fact that this space is non--empty is proven in Lemma~2.1 of~\cite{BMW94}, using an explicit construction via ``convolution'' of an $\LL^\infty(\Omega)$ function with the Fourier transform of a function in $\CC^\infty_0(\R^d)$.

For any $f\in\LL^2(\Omega)$, the stochastic weak derivative $\delta^\alpha f$ of $f$ is the distribution in $\mathscr{D}^\infty(\Omega)'$ defined via
\begin{equation*}
  \<\delta^\alpha f,\phi\> := (-1)^{|\alpha|}\int_\Omega f \, \delta^\alpha \phi \, \dPr\quad\text{for any $\phi\in\mathscr{D}^\infty(\Omega)$},
\end{equation*}
where $|\alpha| = \sum_{i=1}^d \alpha_i$. The space $\HH^1(\Omega)$ is defined to be the subspace of $\LL^2(\Omega)$ for which there exists $f_j\in\LL^2(\Omega)$ (for any $1 \leq j \leq d$) such that
\begin{equation*}
  \int_\Omega f_j\, \phi \, \dPr = \<\delta_j f,\phi\>\quad\text{for any $\phi\in\mathscr{D}^\infty(\Omega)$}.
\end{equation*}
As usual, we abuse notation by writing $\delta_j f \in\LL^2(\Omega)$ to denote $f_j$ in the above definition whenever $f\in\HH^1(\Omega)$.

Let $u\in\HH^1(\Omega)$. It is convenient to define the \emph{stochastic gradient} $D_\omega u\in\LL^2(\Omega)^d$ of $u$ by
$$
D_\omega u=(\delta_1u, \ldots,\delta_d u).
$$
We note that $\HH^1(\Omega)$ is a Hilbert space for the usual inner product
\begin{equation*}
  (u,v)_{\HH^1(\Omega)} := (u,v)_{\LL^2(\Omega)}+(D_\omega u,D_\omega v)_{\LL^2(\Omega)} =(u,v)_{\LL^2(\Omega)}+\sum_{j=1}^d (\delta_j u, \delta_jv)_{\LL^2(\Omega)}.
\end{equation*}    
Furthermore, $\mathscr{D}^\infty(\Omega)$ contains a countable subset which is dense in this space (see Lemma~2.1 of~\cite{BMW94}).

\begin{remark} \label{rem:evident}
For any $u \in \HH^1(\Omega)$, we have $\dps \int_\Omega D_\omega u \dPr = 0$. In other words, stochastic gradients are always mean zero fields. This is a direct consequence of the definition of $D_\omega$ and of the fact that $\Pr$ is an invariant measure with respect to $T$, which implies that $\dps \int_\Omega u(T(x) \omega) \dPr$ is independent of $x$.
\end{remark}

%% I (=Fred) remove old stuff commented out by Tom and move it to another file
%% \input{comment_out_tom1.tex}

\subsection{Stress and strain spaces}
  \label{sec:stress_and_strain}
  
In addition to the displacement spaces defined above, we also distinguish appropriate spaces in which to consider the corresponding stresses and strains.

First, we note that $\LL^2(\Dom)^d$ is a Hilbert space when endowed with the inner product
\begin{equation*}
  (\xi,\zeta)_{\LL^2(\Dom)^d}:=\int_\Dom \xi\cdot \zeta\dx.
\end{equation*}
Following the notation used in Chapter~1 of~\cite{JKO94}, we define a space of strains which are compatible with a displacement in $\HBC$,
\begin{equation*}
  \L2Pot(\Dom) = \left\{ \xi \ \middle| \ \xi = D_xu\text{ for some $u\in\HBC$} \right\} \subset \LL^2(\Dom)^d.
\end{equation*}
This is a closed subspace of $\LL^2(\Dom)^d$, since $\HBC$ is a Hilbert space where the Poincar\'e inequality~\eqref{eq:poincare} holds. Moreover, it has orthogonal complement
\begin{equation*}
  \L2Pot(\Dom)^\perp = \left\{\xi\in\LL^2(\Dom)^d\,\middle|\, \int_\Dom\xi\cdot D_xu\dx = 0 \text{ for any $u\in\HBC$} \right\}.
\end{equation*}
Stated differently, this is the space of all $\xi\in\LL^2(\Dom)^d$ which satisfy $-\div_x\,\xi=0$ in $\HBCPrime$, in the notation introduced at the end of Section~\ref{sec:Deformations}. This is an appropriate space for certain stresses we consider in the sequel.

\smallskip

In a similar way (again, see Chapter~7 of~\cite{JKO94}), we define the stochastic version of these spaces, 
\begin{gather*}
  \L2Pot(\Omega) = \overline{\left\{ \theta \ \middle| \ \theta = D_\omega u\text{ for some $u\in\HH^1(\Omega)$} \right\}}^{\ \LL^2(\Omega)^d},
  \\
  \L2Pot(\Omega)^\perp = \left\{ \xi\in\LL^2(\Omega)^d \ \middle| \ \int_\Omega \xi\cdot D_\omega u \dPr = 0 \text{ for any $u\in\HH^1(\Omega)$} \right\}.
\end{gather*}
We note that both of these spaces are separable, a property which they inherit from $\LL^2(\Omega)$. Again, the former space corresponds to ``stochastic strains'', and the latter to certain ``stochastic stresses''. Unlike in the case of $\L2Pot(\Dom)$, we note that we cannot in general assert that $\theta\in\L2Pot(\Omega)$ implies that $\theta=D_\omega u$ for some $u\in\HH^1(\Omega)$, since we cannot guarantee that a Poincar\'e inequality of the form
\begin{equation*}
  \|u\|_{\LL^2(\Omega)}\leq C\,\|D_\omega u\|_{\LL^2(\Omega)}
\end{equation*}
holds for any $u\in\HH^1(\Omega)$. This is the reason why we have defined $\L2Pot(\Omega)$ as the $\LL^2(\Omega)^d$ closure of this set. We remark that, in the example given in Section~\ref{sec:periodic}, a Poincar\'e inequality does hold on $\HH^1_{\mathrm{per}}(Q)$, and therefore in this case we are able to conclude that any $\theta\in\L2Pot(\Omega)$ satisfies $\theta=D_\omega u$ for some $u\in\HH^1(\Omega)$.

\begin{remark} \label{rem:evident2}
  As a direct consequence of Remark~\ref{rem:evident}, we see that any $\theta \in \L2Pot(\Omega)$ satisfies $\dps \int_\Omega \theta \dPr = 0$.
\end{remark}
  
\subsection{Evolution and corrector spaces}

Finally, we introduce Bochner spaces corresponding to trajectories of the evolution problems we study. These spaces are $\LL^2\b([0,T];X\b)$ and $\HH^1\b([0,T];X\b)$ where $X$ is a Banach space, being respectively the space of square Bochner--integrable functions $u:[0,T]\to X$ and the space of square Bochner--integrable functions with square Bochner--integrable weak derivative in time. We write $u(t)$ to mean the value of $u$ in $X$ at $t\in[0,T]$. These spaces are Banach spaces when endowed with the norms
\begin{equation*}
  \|u\|_{\LL^2([0,T];X)}:=\bg(\int_0^T \|u(t)\|^2_X\dt\bg)^{1/2} \qquad\text{and}\qquad\|u\|_{\HH^1([0,T];X)}:=\bg(\int_0^T \left( \|\dot{u}(t)\|^2_X+\|u(t)\|^2_X \right) \dt\bg)^{1/2},
\end{equation*}
where here and throughout the remainder of our analysis, $\dot{u}$ denotes the time derivative of $u$. Moreover, in the case where $X$ is a separable space, we have the isometric isomorphism
\begin{equation*}
  \Big( \LL^2\b([0,T];X\b) \Big)' \simeq \LL^2\b([0,T];X'\b),
\end{equation*}
and we identify these spaces throughout our analysis. We also note that $\LL^2\b([0,T];X\b)$ and $\HH^1\b([0,T];X\b)$ are reflexive whenever $X$ is reflexive, and separable whenever $X$ is separable.
  
We also consider the space $\CC\b([0,T];X\b)$, which is the space of continuous maps from $[0,T]$ into the Banach space $X$. We note that any $u\in\HH^1\b([0,T];X\b)$ has a representative in $\CC\b([0,T];X\b)$, which satisfies
\begin{equation*}
  u(t_1)-u(t_0) = \int^{t_1}_{t_0} \dot{u}(s)\ds\quad\text{for any $t_1,t_0\in [0,T]$}.
\end{equation*}

\blue{
Similarly to functions which are Bochner--integrable and have weak derivatives in time, at various points we consider Bochner--integrable corrector functions with values in $\LL^2(\Dom;\L2Pot(\Omega))$, and stresses with values in $\LL^2(\Dom;\L2Pot(\Omega)^\perp)$. Both of the latter spaces have norm
 \begin{equation*}
   \b\|\theta\b\|_{\LL^2(\Dom;\LL^2(\Omega))}:=\bigg(\int_\Dom\int_\Omega\b|\theta(\omega,x)\b|^2\dPr(\omega)\dx\bigg)^{1/2}.
 \end{equation*}

\subsection{Measurability considerations and representations}
\label{sec:Correctors}

In the sequel, we manipulate functions which depend on time $t\in[0,T]$, random realisation $\omega\in\Omega$ and spatial position $x\in\Dom$. We assume that such functions are measurable on the product space $[0,T]\times\Omega\times\Dom$, and one example of natural function spaces in which to consider these function under our assumptions is $\LL^2([0,T]\times\Omega\times\Dom)^n$, with $n=1$ (for scalar-valued functions) or $n=d$ (for vector-valued functions). The inner product on these spaces is
\begin{equation*}
  (\theta,\xi)_{\LL^2([0,T]\times\Omega\times\Dom)^n}:=\int_0^T\!\!\int_\Omega\int_\Dom \theta(t,\omega,x)\cdot \xi(t,\omega,x)\dx\dPr(\omega)\dt,
\end{equation*}
where we use Fubini's theorem to write the integral over the product space. We note that, as a consequence of the results contained in Chapter~III, Sections~11.16--11.17 of~\cite{DunfordSchwartzI}, we may equivalently require that such functions are represented in $\LL^2\b([0,T]\times\Omega;\LL^2(\Dom)\b)$ or $\LL^2\b([0,T];\LL^2\b(\Omega;\LL^2(\Dom)\b)\b)$ (or indeed any other permutation of this ordering) and vice versa. Throughout this article, we will identify such representations notationally without comment.

As particular examples of this convention, we will freely switch between thinking of:
\begin{itemize}
\item $u\in\HH^1\b([0,T];\LL^2(\Omega;\HBC)\b)$, as being identified with $u\in\LL^2([0,T]\times\Omega\times\Dom)$ and such that $\dot{u}\in\LL^2([0,T]\times\Omega\times\Dom)$ with $D_xu,D_x\dot{u}\in\LL^2([0,T]\times\Omega\times\Dom)^d$;
\item $\theta \in \HH^1\b([0,T];\LL^2(\Dom;\L2Pot(\Omega))\b)$, as being identified with $\theta \in \LL^2([0,T]\times\Omega\times\Dom)^d$ such that $\dot{\theta} \in \LL^2([0,T]\times\Omega\times\Dom)^d$ with $\theta(t,\cdot,x),\dot{\theta}(t,\cdot,x)\in\L2Pot(\Omega)$ for almost every $(t,x)\in[0,T]\times\Dom$; and
\item $\zeta\in\LL^2\b([0,T];\LL^2(\Dom;\L2Pot(\Omega)^\perp)\b)$, as being identified with $\zeta\in\LL^2([0,T]\times\Omega\times\Dom)^d$ such that $\zeta(t,\cdot,x)\in\L2Pot(\Omega)^\perp$ for almost every $(t,x)\in[0,T]\times\Dom$.
\end{itemize}
Since we never require a pointwise identification of the functions we consider (we seek only weak solutions to the problem we consider), we mainly opt to refer to functions as being represented in the most restrictive setting they can be considered for notational brevity.}

\subsection{Stochastic two--scale convergence}
\label{sec:stoc2scale}

In this section, we recall and slightly extend the definition of stochastic two--scale convergence given in~\cite{BMW94}. Throughout, we write $\LL^2\b([0,T]\times\Omega\times\Dom\b)$ to mean the space of measurable square--integrable functions with respect to the product measure $\Leb^1\otimes\Pr\otimes\Leb^d$, with the usual sigma--algebra generated by $\mathcal{M}^1\otimes\Sigma\otimes\mathcal{M}^d$.
  
The compactness results herein are very close to those given in~\cite{BMW94}, which in turn share much in common with those in~\cite{Allaire94}: the generalization we make for the subsequent analysis is that we need to obtain a form of compactness such that, loosely, we may select subsequences of a bounded sequence $u^\eps(t)$ which two--scale converge in space for almost every time $t$ and realization $\omega$.

An important class required for the definition of two--scale convergence is that of \emph{admissible functions}. These are functions which satisfy the correct measurability properties to allow us to define two--scale convergence.

\begin{definition}
  \label{def:admissible}
A function $\psi\in\LL^2([0,T]\times\Omega\times\Dom)$ is \emph{admissible} if $\psi_T:(t,\omega,x)\mapsto \psi(t,T(x)\omega,x)$ satisfies $\psi_T\in\LL^2([0,T]\times\Omega\times\Dom)$.
\end{definition}

\medskip

Note that, for some $\psi\in\LL^2([0,T]\times\Omega\times\Dom)$, the function $\psi_T$ may not be measurable. This motivates the above definition. As particular examples of admissible functions, it is straightforward to use the ideas of~\cite{BMW94} to check that functions in $\CC([0,T];\CC(\overline{\Dom},\mathscr{D}^\infty(\Omega))$ are admissible, as is any function of the form
\begin{equation*}
  (t,\omega,x)\mapsto f(t) \, g(x) \, h(\omega) \qquad \text{where $f\in\CC^\infty_0([0,T])$, $g\in\CC^1_0(\Dom\cup\GNeu)$ and $h\in\mathscr{D}^\infty(\Omega)$}.
\end{equation*}
Series of functions of this form are dense in $\LL^2([0,T]\times\Omega\times\Dom)$. This is an important observation which we use repeatedly in the analysis which follows. With the class of admissible functions prescribed, we define two--scale convergence as follows.
  
\begin{definition}
  \label{def:two-scale}
A sequence $u^\eps\in\LL^2\b([0,T]\times\Omega\times\Dom\b)$ is said to \emph{two--scale converge} to $u^\star\in\LL^2\b([0,T]\times\Omega\times\Dom\b)$ if, for any admissible $\psi$, we have
\begin{equation*}
  \lim_{\eps\to 0} \int_0^T\!\!\blue{\int_\Omega\int_\Dom} u^\eps(t,\omega,x) \, \psi\left(t,T\left(\frac{x}{\eps}\right)\omega,x\right) \dx \dPr(\omega) \dt
  =
  \int_0^T\!\!\blue{\int_\Omega\int_\Dom} u^\star(t,\omega,x) \, \psi(t,\omega,x) \dx \dPr(\omega) \dt.
\end{equation*}
\end{definition}
\blue{This definition is a generalisation of the one given in Definition~3.3 in~\cite{BMW94} to incorporate time dependence. Indeed, if we consider a sequence $u^\eps$ and limit $u^\star$ which are both independent of $t$, then the two definitions coincide.}

\medskip

The following lemma extends the fundamental compactness result concerning two--scale convergence to the case of two--scale convergence in the sense of the above definition. For comparison, similar results are given in Theorem~3.4 of~\cite{BMW94} and in Theorem~1.2 of~\cite{Allaire94}.
  
\begin{lemma}
\label{th:2scale_L2}
Let $u^\eps$ be a bounded sequence in $\LL^2\b([0,T]\times\Omega\times\Dom\b)$. Then there exists a subsequence of $u^\eps$ and some $u^\star\in\LL^2\b([0,T]\times\Omega\times\Dom\b)$ such that, along this subsequence, $u^\eps$ two--scale converges to $u^\star$.
\end{lemma}

\red{
Before providing a proof of this result, we make some remarks about the relationship between two--scale convergence and weak convergence.

\begin{remark}
  If we choose an admissible test function $\psi$ in Definition~\ref{def:two-scale} which is deterministic, then as a consequence of the definition, any measurable representative of the map $(t,x)\mapsto\int_\Omega u^\eps(t,\omega,x)\dPr(\omega)$ weakly--converges to $(t,x)\mapsto\int_\Omega u^\star(t,\omega,x)\dPr(\omega)$ in $\LL^2([0,T]\times\Dom)$. Two--scale convergence thus entails weak convergence of expectations, but does not necessarily imply weak convergence in $\LL^2([0,T]\times\Dom\times\Omega)$. Viewing these functions as expectations conditioned on the space and time variables explains the use of the terminology `two--scale convergence in the mean' used in~\cite{BMW94}.
\end{remark}
}

\begin{proof}
We follow the proof of Theorem~3.4 of~\cite{BMW94}. Since $\LL^2(\Omega)$ is separable, there exists a countable and dense set $S$ of admissible functions in $\LL^2\b([0,T]\times\Omega\times\Dom\b)$. For each $\psi \in S \subset \LL^2\b([0,T]\times\Omega\times\Dom\b)$ (which is admissible), we have, employing the Cauchy--Schwarz inequality and the invariance of $\Pr$ with respect to $T$, that
\begin{equation*}
  \bg|\int_0^T\!\!\blue{\int_\Omega\int_\Dom}u^\eps(t,\omega,x) \, \psi\left(t,T\left(\frac{x}{\eps}\right)\omega,x\right)\dx\, \mathrm{d}\Pr(\omega)\dt\bg|\leq C \|\psi\|_{\LL^2([0,T]\times\Omega\times\Dom)}
\end{equation*}
where $C = \sup_\eps \|u^\eps\|_{\LL^2([0,T]\times\Omega\times\Dom)}$. We can thus find a subsequence $\{ u^{\eps'} \}$ (which depends on $\psi$) such that
\begin{equation} \label{eq:chicago}
  \lim_{\eps' \to 0} \int_0^T\!\!\blue{\int_\Omega\int_\Dom}u^{\eps'}(t,\omega,x) \, \psi\left(t,T\left(\frac{x}{\eps'}\right)\omega,x\right)\dx\, \mathrm{d}\Pr(\omega)\dt
\end{equation}
exists. Since $S$ is countable, we can infer from a diagonalization argument that there exists a subsequence $\{ u^{\eps'} \}$ (independent of the elements in $\psi$) such that, for any $\psi \in S$, the limit~\eqref{eq:chicago} exists. Let $X$ be the subspace of admissible functions $\psi \in \LL^2\b([0,T]\times\Omega\times\Dom\b)$ such that the limit~\eqref{eq:chicago} (along the subsequence we have just defined) exists. Then $X$ is a vector subspace of $\LL^2\b([0,T]\times\Omega\times\Dom\b)$ and the limit~\eqref{eq:chicago} defines a bounded linear functional (denoted $g$) on $X$. Recalling that $S \subset X$ and that $S$ is dense in $\LL^2\b([0,T]\times\Omega\times\Dom\b)$, we can extend $g$ to a bounded linear functional on $\LL^2\b([0,T]\times\Omega\times\Dom\b)$. Using the Riesz theorem, this functional $g$ may be identified with some $u^\star \in \LL^2\b([0,T]\times\Omega\times\Dom\b)$. This concludes the proof of Lemma~\ref{th:2scale_L2}.
\end{proof}

We also have the following result, which is the one we use below. Again, this is a slight generalization of Proposition~1.14(i) of~\cite{Allaire94} and of Theorem~3.7(b) of~\cite{BMW94}.
  
\begin{lemma}
\label{th:2scale_H1}
Let $u^\eps$ be a bounded sequence of functions in $\HH^1\b([0,T];\LL^2(\Omega;\HBC)\b)$. Then there exist a scalar-valued function $u^\star\in\HH^1([0,T];\HBC)$, a vector-valued function $\theta\in\HH^1\b([0,T];\LL^2\b(\Dom;\L2Pot(\Omega)\b)\b)$ and a subsequence along which:
\begin{enumerate}
\item $u^\eps$ and $\dot{u}^\eps$ respectively two--scale converge to $u^\star$ and $\dot{u}^\star$, \label{i}
\item $D_xu^\eps$ and $D_x\dot{u}^\eps$ two--scale converge to $D_xu^\star+\theta$ and $D_x\dot{u}^\star+\dot{\theta}$, \label{ii}
\end{enumerate}
where each of these two--scale limits \blue{is identified with a representative in} $\LL^2([0,T]\times\Omega\times\Dom)$ (or $\LL^2([0,T]\times\Omega\times\Dom)^d$ respectively);
\begin{enumerate}[resume]
\item $u^\eps(0,\cdot\,,\cdot\,)$ and $u^\eps(T,\cdot\,,\cdot\,)$ respectively two--scale converge to $u^\star(0,\cdot\,)$ and $u^\star(T,\cdot\,)$, \label{iii}
\item $D_xu^\eps(0,\cdot\,,\cdot\,)$ and $D_xu^\eps(T,\cdot\,,\cdot\,)$ respectively two--scale converge to $D_xu^\star(0,\cdot\,)+\theta(0,\cdot\,,\cdot\,)$ and $D_xu^\star(T,\cdot\,)+\theta(T,\cdot\,,\cdot\,)$, \label{iv}
\end{enumerate}
where each of these two--scale limits \blue{is identified with a representative in} $\LL^2(\Omega\times\Dom)$ (or $\LL^2(\Omega\times\Dom)^d$ respectively), and two--scale convergence holds in the sense of \blue{Definition~3.3 of~\cite{BMW94}.}

Moreover, if $u^\eps(0,\omega,x)=u_0(x)$ for any $\eps$ and for almost every $(\omega,x)\in \Omega \times \Dom$ with $u_0 \in \HBC$, then $u^\star(0,x)=u_0(x)$ in $\HH^1(\Dom)$ and $\theta(0,\omega,x)=0$ in $\LL^2(\Dom\times\Omega)^d$.
\end{lemma}

\blue{Before proceeding to the proof of this lemma, we recall that, as mentioned above, Definition~3.3 of~\cite{BMW94} is identical to Definition~\ref{def:two-scale} given above when applied to $t$--independent functions.}

\medskip

\begin{proof}
The proof is organised into 5 steps. The proof of the assertions~(\ref{i}) and~(\ref{ii}) is an adaptation of the proof of Theorem~3.7(b) of~\cite{BMW94}.

\medskip

\noindent
\emph{Step 1. Two--scale convergence of $u^\eps$ and $\dot{u}^\eps$.}
\blue{We first note that applying the Lemma proved in~III.11.17 of~\cite{DunfordSchwartzI} allows us to represent $u^\eps$ and all of its weak derivatives in space and time in $\LL^2([0,T]\times\Omega\times\Dom)^n$ with $n=1$ or $n=d$. In an abuse of notation, we will refer to both representations in the same way.} Lemma~\ref{th:2scale_L2} then entails that we may extract a subsequence such that $u^\eps$ and $\dot{u}^\eps$
%, $D_xu^\eps$ and $D_x\dot{u}^\eps$ all 
two--scale converge to some $u^\star$ and $v^\star$ respectively, which both belong to $\LL^2([0,T]\times\Omega\times\Dom)$. For any $f\in\CC_0^\infty([0,T])$, $g\in\CC^1_0(\Dom)$ and $h\in\mathscr{D}^\infty(\Omega)$, by integrating by parts, we have
\begin{equation*}
  \int_0^T\!\!\blue{\int_\Omega\int_\Dom} u^\eps(t,\omega,x) \, \dot{f}(t) \, g(x) \, h\left(T\left(\frac x\eps\right)\omega\right) \dx \dPr(\omega) \dt
  = -\int_0^T\!\!\blue{\int_\Omega\int_\Dom} \dot{u}^\eps(t,\omega,x) \, f(t) \, g(x) \, h\left(T\left(\frac x\eps\right)\omega\right) \dx \dPr(\omega) \dt.
\end{equation*}
Passing to the two--scale limit, we obtain
\begin{equation*}
  \int_0^T\!\!\blue{\int_\Omega\int_\Dom} u^\star(t,\omega,x) \, \dot{f}(t) \, g(x) \, h(\omega) \dx \, \dd\Pr(\omega) \dt = -\int_0^T\!\!\blue{\int_\Omega\int_\Dom} v^\star(t,\omega,x) \, f(t) \, g(x) \, h(\omega) \dx \dPr(\omega)\dt.
\end{equation*}
Since the tensor product $\mathscr{D}^\infty(\Omega)\otimes\CC^1_0(\Dom)$ is dense in $\LL^2(\Omega\times\Dom)$ (because the component spaces endowed with appropriate $\LL^2$ norms are respectively dense in $\LL^2(\Omega)$ and $\LL^2(\Dom)$), we find that
\begin{equation*}
  \blue{\int_\Omega\int_\Dom} \int_0^T \!\!\dot{f}(t) u^\star(t,\omega,x) \dt\; k(\omega,x) \dx \, \dd\Pr(\omega) = - \blue{\int_\Omega\int_\Dom}\int_0^T\!\! f(t) v^\star(t,\omega,x) \dt \; k(\omega,x) \dx \dPr(\omega)
\end{equation*}
for any $f\in\CC^\infty_0([0,T])$ and any $k\in\LL^2(\Omega\times\Dom)$. It follows that \blue{$v^\star(\cdot,\omega,x)=\dot{u}^\star(\cdot,\omega,x)$ for almost every $(\omega,x)\in\Omega\times\Dom$}, and identifying the dual of $\LL^2(\Omega\times\Dom)$ with the space itself, we deduce $u^\star\in\HH^1([0,T];\LL^2(\Omega\times\Dom))$.

\medskip

\noindent
\emph{Step 2. $u^\star$ is deterministic.} 
Recalling Lemma~2.1 of~\cite{BMW94}, there exists $S_0\subseteq \mathscr{D}^\infty(\Omega)$ which is dense in $\LL^2(\Omega)$ such that
\begin{equation*}
  (D_\omega h)(T(x)\omega) = D_x \Big(h(T(x)\omega)\Big) \qquad \text{for any $h\in S_0$}.
\end{equation*}
Taking $f$ and $g$ as above and $h\in S_0$, by ``integrating by parts'', we find
\begin{align*}
  & \eps\int_0^T\!\!\blue{\int_\Omega\int_\Dom}D_xu^\eps(t,\omega,x) \, f(t) \, g(x) \, h\left(T\left(\frac x\eps\right)\omega\right) \dx \dPr(\omega)\dt
  \\
  &= -\eps\int_0^T\!\!\blue{\int_\Omega\int_\Dom} u^\eps(t,\omega,x) \, f(t) \, D_xg(x) \, h\left(T\left(\frac x\eps\right)\omega\right) \dx \dPr(\omega)\dt
  \\
  & \qquad - \int_0^T\!\!\blue{\int_\Omega\int_\Dom} u^\eps(t,\omega,x) \, f(t) \, g(x) \, (D_\omega h)\left(T\left(\frac x\eps\right) \omega\right) \dx\dPr(\omega)\dt.
\end{align*}
Passing to the two--scale limit, and noting that the first and second integrals remain bounded by assumption, we see that
\begin{equation*}
  0 = \int_0^T\!\!\blue{\int_\Omega\int_\Dom} u^\star(t,\omega,x) \, f(t) \, g(x) \, D_\omega h(\omega)\dx\dPr(\omega)\dt.
\end{equation*}
It follows that, for any $h\in S_0$,
\begin{equation*}
  0 = \int_\Omega u^\star(t,\omega,x) \, D_\omega h(\omega) \dPr(\omega) \qquad \text{for almost every $(t,x)\in[0,T]\times\Dom$}.
\end{equation*}
Since $S_0$ is dense in $\mathscr{D}^\infty(\Omega)$, we obtain that $D_\omega u^\star(t,\omega,x)=0$ for almost every $(t,\omega,x)\in[0,T]\times\Omega\times\Dom$. Since $T$ is ergodic, the discussion in Section~2 of~\cite{BMW94} demonstrates that any $u\in\LL^2(\Omega)$ such that $D_\omega u=0$ is constant. We may thus choose a representative such that $u^\star=u^\star(t,x)$, and we have demonstrated that $u^\star\in\HH^1([0,T];\LL^2(\Dom))$.

\medskip

\noindent
\emph{Step 3. Two--scale convergence of $D_xu^\eps$ and $D_x\dot{u}^\eps$.} 
Since $D_xu^\eps$ is bounded in $\LL^2(\blue{[0,T]\times\Omega\times\Dom})^d$, there exists a further subsequence along which $D_xu^\eps$ two--scale converges to some $\xi\in \LL^2([0,T]\times\Omega\times\Dom)^d$. Since $D_x\dot{u}^\eps$ is also bounded in $\LL^2(\blue{[0,T]\times\Omega\times\Dom})^d$, we can require that it also two--scale converges. Using the same arguments as in Step 1, we can show that $D_x\dot{u}^\eps$ two--scale converges to $\dot{\xi}$ and therefore that $\xi$ can be identified with a function in $\HH^1([0,T];\LL^2(\Omega\times\Dom)^d)$.

\smallskip

Taking $f\in\CC^\infty_0([0,T])$ and $g\in\CC^1_0(\Dom\cup\GDir)^d\cap\L2Pot(\Dom)^\perp$, i.e. $g$ such that $-\div_xg=0$ in $\HBCPrime$, we obtain, applying the definition of two--scale convergence, that
\begin{align*}
  0&=-\lim_{\eps\to0} \int_0^T\!\!\int_\Omega\int_\Dom u^\eps(t,\omega,x) \, f(t) \, \div_x g(x) \dx\dPr(\omega)\dt
  \\
  &=\lim_{\eps\to0} \int_0^T\!\!\int_\Omega\int_\Dom D_xu^\eps(t,\omega,x) \cdot g(x) \, f(t) \dx\dPr(\omega)\dt
  \\
  &=\int_0^T\!\!\int_\Omega\int_\Dom \xi(t,\omega,x) \cdot g(x) \, f(t) \dx\dPr(\omega)\dt.
\end{align*}
By the density of such $g$ in $\L2Pot(\Dom)^\perp$, we have that $\dps \int_\Omega \xi(t,\omega,\cdot)\dPr(\omega)\in\L2Pot(\Dom)$ for almost every $t\in[0,T]$, and therefore there exists $w\in\LL^2([0,T];\HBC)$ such that
\begin{equation*}
  D_xw(t,x) = \int_\Omega \xi(t,\omega,x)\dPr(\omega)\qquad\text{for almost every $(t,x)\in[0,T]\times\Dom$}.
\end{equation*}
Since $\xi \in \HH^1([0,T];\LL^2(\Omega\times\Dom))^d$, we further obtain that $D_xw\in\HH^1([0,T];\LL^2(\Dom)^d)$, and hence $w\in\HH^1([0,T];\HBC)$.

\smallskip

Taking now a more general space test function, we have, for any $\varphi\in\CC^\infty_0(\Dom)$, that
\begin{align*}
  \int_0^T\!\!\int_\Dom D_xw(t,x) \, f(t) \, \varphi(x)\dx\dt
  &=
  \lim_{\eps\to0}\int_0^T\!\!\int_\Omega\int_\Dom D_xu^\eps(t,\omega,x) \, f(t) \, \varphi(x)\dx\dPr(\omega)\dt
  \\
  &=-\lim_{\eps\to0}\int_0^T\!\!\int_\Omega\int_\Dom u^\eps(t,\omega,x) \, f(t) \, D_x\varphi(x)\dx\dPr(\omega)\dt
  \\
  &=-\int_0^T\!\!\int_\Dom u^\star(t,x) \, f(t) \, D_x\varphi(x)\dx\dt.
\end{align*}
This shows that $D_x w= D_x u^\star$, and hence that we may take $u^\star\in\HH^1([0,T];\HH^1(\Dom))$.

\medskip

We now show that $u^\star(t,\cdot) \in \HBC$ for almost any $t$. Let $f\in\CC^\infty_0([0,T])$ and $g\in\CC^1_0(\Dom\cup\GNeu)$. The function $(t,x,\omega) \mapsto f(t) \, g(x)$ is admissible. Since $u^\eps$ two--scale converges to $u^\star$ which is deterministic, we have
$$
\int_0^T\!\!\int_\Dom \Ex\left[u^\eps(t,\cdot,x)\right] \, f(t) \, g(x) \dx \dt
=
\int_0^T\!\!\int_\Dom\int_\Omega u^\eps(t,\omega,x) \, f(t) \, g(x) \dx \dPr(\omega) \dt
\underset{\eps \to 0}{\to}
\int_0^T\!\!\int_\Dom u^\star(t,x) \, f(t) \, g(x) \dx \dt.
$$
By density of such functions $f(t) \, g(x)$ in $\LL^2([0,T] \times\Dom)$, we hence get that $\Ex[u^\eps]$ weakly converges in $\LL^2([0,T] \times\Dom)$ to $u^\star$. In addition, since $\Ex[u^\eps]$ is bounded in $\HH^1([0,T];\HBC)$, we know that, up to a subsequence extraction, it converges (weakly in $\HH^1([0,T];\HBC)$ and strongly in $\LL^2([0,T] \times\Dom)$) to some $u_1 \in \HH^1([0,T];\HBC)$. By uniqueness of the limit in $\LL^2([0,T] \times\Dom)$, we get that $u_1 = u^\star$, and hence $u^\star \in \HH^1([0,T];\HBC)$.

\medskip

We know that $D_xu^\eps$ two--scale converges to $\xi$, that we write in the form $\xi = D_xu^\star+\theta$ for some $\theta\in\HH^1([0,T]; \LL^2(\Omega\times\Dom)\red{^d})$. We claim that $\theta(t,\cdot,x)\in\L2Pot(\Omega)$ for almost every $(t,x)\in[0,T]\times\Dom$. \blue{Following the construction made in the proof of Lemma~2.3(b) in~\cite{BMW94}, there exists a set $S_1\subset \mathscr{D}^\infty(\Omega)^d$ which is dense in the closure of the kernel of $\div_\omega$ (or equivalently, in $\L2Pot(\Omega)^\perp$) such that
  \begin{equation*}
    \eps\,\div_x \big( \zeta\left(T(x/\eps)\omega\right) \big) = \big(\div_\omega\zeta\big)\left(T(x/\eps)\omega\right)=0.
  \end{equation*}
}
%\fl{in the proof of Lemma~2.3(b) in~\cite{BMW94}, they consider some $g \star k_m$, and the set of these is $S_1$. They prove page 27 that this set is dense in $\L2Pot(\Omega)^\perp$; this set is built in the same way as the set $S_0$, see just above (2.6) in~\cite{BMW94}, and hence satisfies the above relation as for functions in $S_0$}
Let $f\in\CC^\infty_0([0,T])$, $g\in\CC^1_0(\Dom)$ and $\zeta\in S_1$. We have that
\begin{multline*}
  \int_0^T\!\!\int_\Omega\int_\Dom f(t) \, g(x) \, \left[D_xu^\eps(t,\omega,x)-D_xu^\star(t,x)\right] \cdot \zeta\left(T\left(\frac x\eps\right)\omega\right)\dx\dPr(\omega)\dt\\
    =-\int_0^T\!\!\int_\Omega\int_\Dom\left[u^\eps(t,\omega,x)-u^\star(t,x)\right] f(t) \, D_xg(x) \cdot \zeta\left(T\left(\frac x\eps\right)\omega\right)\dx\dPr(\omega)\dt.
\end{multline*}
Passing to the two--scale limit, we get
\begin{equation*}
  \int_0^T\!\!\int_\Omega\int_\Dom f(t) \, g(x) \, \theta(t,\omega,x) \cdot \zeta(\omega)\dx\dPr\dt = 0.
\end{equation*}
\blue{Now, by the previously asserted density of $S_1$ in $\L2Pot(\Omega)^\perp$,} we deduce that $\theta(t,\cdot,x)\in\L2Pot(\Omega)$ for almost every $(t,x)\in[0,T]\times\Dom$. \blue{Applying the result of~III.11.17 in~\cite{DunfordSchwartzI}, we may require that $\theta\in\HH^1\b([0,T];\LL^2(\Dom;\L2Pot(\Omega))\b)$, which} concludes the proof of assertions~(\ref{i}) and~(\ref{ii}) of the lemma.

\medskip

\noindent
\emph{Step 4. Two--scale convergence of $u^\eps(0)$, $D_xu^\eps(0)$, $u^\eps(T)$ and $D_xu^\eps(T)$.} 
Since $\HH^1([0,T];\LL^2(\Omega;\HH^1(\Dom)))$ is embedded in $\CC([0,T];\LL^2(\Omega;\HH^1(\Dom)))$, we may select continuous representatives of $u^\eps$ which have well--defined values at $t=0$ and $t=T$. In addition, $u^\eps(0)$ and $u^\eps(T)$ are bounded in $\LL^2(\Omega;\HH^1(\Dom))$. Using Theorem~3.7(b) of~\cite{BMW94}, we deduce that (up to extracting a further subsequence) $u^\eps(0)$ two--scale converges to some $U_0 \in\HH^1(\Dom)$ and $D_xu^\eps(0)$ two--scale converges to $D_xU_0+\Theta_0$ for some $\Theta_0 \in \LL^2(\Dom;\L2Pot(\Omega))$, in the sense that, for any $\psi \in \LL^2(\Omega\times\Dom)$ and $\Psi \in\LL^2(\Omega\times\Dom)^d$ which are admissible in the sense of~\cite{BMW94},
\begin{align*}
  \lim_{\eps\to 0}\int_\Omega\int_\Dom u^\eps(0,\omega,x) \, \psi\left(T\left(\frac{x}{\eps}\right)\omega,x\right) \dx \dPr(\omega)
  &=
  \int_\Omega\int_\Dom U_0(x) \, \psi(\omega,x) \dx \dPr(\omega),
  \\
  \lim_{\eps\to 0}\int_\Omega\int_\Dom D_xu^\eps(0,\omega,x) \cdot \Psi \left(T\left(\frac{x}{\eps}\right)\omega,x\right) \dx \dPr(\omega)
  &=
  \int_\Omega\int_\Dom \left[D_xU_0(x)+\Theta_0(\omega,x)\right] \cdot \Psi(\omega,x) \dx \dPr(\omega).
\end{align*} 
Likewise, $u^\eps(T)$ two--scale converges to some $U_T \in \HH^1(\Dom)$ and $D_xu^\eps(T)$ two--scale converges to $D_xU_T+\Theta_T$ for some $\Theta_T \in \LL^2(\Dom;\L2Pot(\Omega))$.

\smallskip

Next, we wish to check that the operations of taking the two--scale limit and selecting the initial or final values commute, in order to identify $U_0$, $\Theta_0$, $U_T$ and $\Theta_T$ with respect to $u^\star$ and $\theta$. To do so, consider $\varphi_\delta(t) = \varphi(t/\delta)/\delta$ for some $\varphi\in\CC^\infty_0(\R)$ which is even, non-negative and satisfies
\begin{equation*}
  \int_0^\infty \varphi(t)\dt = 1.
\end{equation*}
Let $X$ be a Banach space. For any $u\in\HH^1([0,T];X)$ and any $0 \leq t \leq T$, we have
$$
\left\| u(t)-u(0) \right\|_X
\leq
\int_0^t \|\dot{u}(s)\|_X \ds
\leq
\sqrt{t} \, \|\dot{u}\|_{\LL^2([0,T];X)}.
$$
We hence deduce that\red{, when $\delta$ is sufficiently small,}
\begin{align}
  \left\|\int_0^T\varphi_\delta(t) \, u(t) \dt - u(0)\right\|_X
  &=
  \left\|\int_0^T\varphi_\delta(t) \, [u(t)-u(0)] \dt \right\|_X
  \nonumber
  \\
  &\leq \int_0^T \sqrt{t} \, \|\dot{u}\|_{\LL^2([0,T];X)} \, \varphi_\delta(t)\dt
  \nonumber
  \\
  & \leq \sqrt{\delta} \, \|\dot{u}\|_{\LL^2([0,T];X)} \int_0^\infty \sqrt{t} \, \varphi(t)\dt
  \nonumber
  \\
  & =C\sqrt{\delta} \, \|\dot{u}\|_{\LL^2([0,T];X)}.
  \label{eq:PointEvaluation}
\end{align}
Taking $g\in\LL^2(\Omega\times\Dom)$ which is admissible as above, and applying the definition of two--scale convergence in $\LL^2([0,T]\times\Omega\times\Dom)$, we find that
\begin{align*}
  & \bigg|\int_0^T\!\!\int_\Omega\int_\Dom u^\star(t,x) \, \varphi_\delta(t) \, g(\omega,x) \dx \dPr(\omega) \dt - \int_\Omega\int_\Dom U_0(x) \, g(\omega,x) \dx \dPr(\omega)\bigg|
  \\
  &= \lim_{\eps\to 0} \bigg|\int_0^T\!\!\int_\Omega\int_\Dom u^\eps(t,\omega,x) \, \varphi_\delta(t) \, g\left(T\left(\frac{x}{\eps}\right)\omega,x\right) \dx \dPr(\omega) \dt
  -\int_\Omega\int_\Dom u^\eps(0,\omega,x) \, g\left(T\left(\frac{x}{\eps}\right)\omega\right) \dx \dPr(\omega)\bigg|
  \\
  & \leq C\sqrt{\delta},
\end{align*}
where the final estimate follows by using the estimate~\eqref{eq:PointEvaluation} in combination with the fact that the sequence $u^\eps$ is uniformly bounded in $\HH^1([0,T];\LL^2(\Omega;\HH^1(\Dom)))$. Letting $\delta \to 0$, and recalling that $u^\star$ belongs to $\HH^1([0,T];\HH^1(\Dom))$ and thus has a representative which is continuous in time, we obtain that $u^\star(0,x)=U_0(x)$ for almost every $x\in\Dom$.

A similar argument applies to show that $D_xu^\star(0,x)+\theta(0,\omega,x) = D_xU_0(x)+\Theta_0(\omega,x)$ a.e. in $\Omega \times \Dom$, and likewise at the final time $T$. This concludes the proof of assertions~(\ref{iii}) and~(\ref{iv}) of the lemma.

\medskip

\noindent
\emph{Step 5. Deterministic two--scale limit at initial time.}
We now assume that $u^\eps(0,\omega,\cdot) = u_0$ in $\HBC$ for almost every $\omega\in\Omega$. In view of assertion~(\ref{iii}), we see that $u^\star(0,\cdot)=u_0$. Taking now $g\in\CC^\infty_0(\Dom)$ and $h\in\mathscr{D}^\infty(\Omega)$, we have, in view of assertion~(\ref{iv}) and of the fact that $D_xu^\star(0,\cdot) = D_x u_0$, that
\begin{equation*}
  0
  =
  \int_\Omega\int_\Dom \left[D_xu^\eps(0,\omega,x)-D_xu_0(x)\right] g(x) \, h\left(T\left(\frac{x}{\eps}\right)\omega\right)\dx\dPr(\omega)
  \underset{\eps \to 0}{\to}
  \int_\Omega\int_\Dom \theta(0,x,\omega) \, g(x) \, h(\omega)\dx\dPr(\omega).
\end{equation*}
By density, we find that $\theta(0,\omega,x)=0$ for almost every $(\omega,x)\in\Omega\times\Dom$. This completes the proof of Lemma~\ref{th:2scale_H1}.
\end{proof}

We note that a key consequence of the previous lemma is that gradient fields $D_xu^\eps$ which are uniformly bounded in $\LL^2(\Omega \times \Dom)^d$ have two--scale convergent subsequences with limits in $\L2Pot(\Dom)+\LL^2(\Dom;\L2Pot(\Omega))$. In the following result, we show that divergence--free fields satisfy a similar property.

%% I (=Fred) remove old stuff not needed anymore
%% \input{comment_out_tom5.tex}

\begin{lemma}
\label{th:2scale_divfree_L2}
Suppose that \blue{$\sigma^\eps \in \LL^2\left([0,T]\times\Omega;\L2Pot(\Dom)^\perp\right)$} is a uniformly bounded sequence. Then there exist two functions $\sigma^\star \in \LL^2\left([0,T];\L2Pot(\Dom)^\perp\right)$ and \blue{$\zeta \in \LL^2\left([0,T]\times\Dom;\L2Pot(\Omega)^\perp\right)$} and a subsequence along which $\sigma^\eps$ two--scale converges to $\sigma^\star+\zeta$. Furthermore, the expectation of $\zeta$ vanishes: $\dps \int_\Omega \zeta(t,\omega,x) \dPr(\omega) = 0$ for almost every $(t,x)\in[0,T]\times\Dom$.
\end{lemma}

\begin{proof}
\blue{As in the proof of Lemma~\ref{th:2scale_H1}, we first note that applying the result of~III.11.17 in~\cite{DunfordSchwartzI} allows us to represent $\sigma^\eps$ in $\LL^2([0,T]\times\Omega\times\Dom)$, for which we use the same notation.} Applying Lemma~\ref{th:2scale_L2}, we may extract a subsequence such that $\sigma^\eps$ two--scale converges to some $\Sigma\in\LL^2([0,T]\times\Omega\times\Dom)^d$. Taking $f\in\CC^\infty([0,T])$ and $g\in\CC^1_0(\Dom\cup\GNeu)$, we have that
\begin{equation*}
  0 = \lim_{\eps\to0}\int_0^T\!\!\int_\Dom\int_\Omega f(t) \, \sigma^\eps(t,\omega,x) \cdot D_xg(x) \dPr(\omega)\dx\dt = \int_0^T\!\!\int_\Dom\int_\Omega f(t) \, \Sigma(t,\omega,x) \cdot D_xg(x)\dx\dPr(\omega)\dt.
\end{equation*}
Defining $\dps \sigma^\star(t,x) = \int_\Omega \Sigma(t,\omega,x)\dPr(\omega)$, it follows that $\sigma^\star(t,\cdot)\in\L2Pot(\Dom)^\perp$ for almost every $t\in[0,T]$. Using the integrability in time of $\Sigma$, we have $\sigma^\star \in \LL^2\left([0,T];\L2Pot(\Dom)^\perp\right)$.

As in the proof of Lemma~\ref{th:2scale_H1}, we recall from Section~2 of~\cite{BMW94} that there exists $S_0\subseteq\mathscr{D}^\infty(\Omega)$ which is dense in $\LL^2(\Omega)$ such that
\begin{equation*}
  D_x \Big( h(T(x)\omega) \Big) = (D_\omega h)(T(x)\omega)\qquad\text{for any $h\in S_0$}.
\end{equation*}
Taking $f$ as before, $g\in\CC^1_0(\Dom)$ and $h \in S_0$, we get
\begin{multline*}
  -\eps \int_0^T\!\!\int_\Dom \int_\Omega f(t) \, h\left(T\left(\frac{x}{\eps}\right)\omega\right) \, \b[\sigma^\eps(t,\omega,x)-\sigma^\star(t,x)\b] \cdot D_xg(x) \dx\dPr(\omega)\dt\\
  =\int_0^T\!\!\int_\Dom\int_\Omega f(t) \, g(x) \, \b[\sigma^\eps(t,\omega,x)-\sigma^\star(t,x)\b] \cdot (D_\omega h)\left(T\left(\frac{x}{\eps}\right)\omega\right)\dx\dPr(\omega)\dt.
\end{multline*}
Defining $\zeta = \Sigma-\sigma^\star$, and passing to the two--scale limit in the above expression, we obtain
\begin{equation*}
  0=\int_0^T\!\!\int_\Omega\int_\Dom f(t) \, g(x) \, \zeta(t,\omega,x) \cdot D_\omega h(\omega)\dx\dPr(\omega)\dt.
\end{equation*}
Since $S_0$ is dense in $\LL^2(\Omega)$, this implies that $\zeta(t,\cdot,x) \in \L2Pot(\Omega)^\perp$ for almost every $(t,x)\in[0,T]\times\Dom$. We have thus shown that \blue{$\zeta \in \LL^2\left([0,T]\times\Dom;\L2Pot(\Omega)^\perp\right)$.} Furthermore, in view of the definition of $\sigma^\star$, we observe that the expectation of $\zeta$ vanishes. This concludes the proof of Lemma~\ref{th:2scale_divfree_L2}.
\end{proof}

\subsection{Characterising $(\Psi^\eta_{t,\omega})^*$}
\label{sec:dual_representation}

\blue{In this section, we provide a preliminary result (namely Lemma~\ref{th:dual_char} below) which characterises various properties of a uniformly convex functional (precisely defined by~\eqref{eq:def_precise} below) defined on $\HBC$ and which only depends on $D_xu$. In order to do so, we make use of an approximation result, Lemma~\ref{th:convex_approx} below, which provides a smooth, measurable approximation of the dissipation potential density $\psi$. This approximation result is also needed for the proof of our main result. Since both lemmas rely crucially on the various structural assumptions made on $\psi$, we recall here these assumptions (made in Section~\ref{sec:dissipation_assumptions}) for the reader's convenience.

\medskip

\begin{center}
  \framed{0.9\textwidth}{
  {\bf Assumptions on $\psi$.}
  \begin{enumerate}
  \item[$(\psi1)$] For any $\xi\in\R^d$, the function $(\omega,x)\in\Omega\times\R^d\mapsto\psi(\omega,x,\xi)$ is measurable with respect to the sigma--algebra generated by $\Sigma\times\Leb^d$, and is stationary in the sense of~\eqref{eq:Stationarity}, i.e. there exists $\psi_0:\Omega\times\R^d\to\R$ such that
    \begin{equation*}
      \text{for any $\xi\in\R^d$}, \quad \psi(\omega,x,\xi) = \psi_0\big(T(x)\omega,\xi\big) \quad\text{for any $x\in\R^d$ and almost every $\omega\in\Omega$}.
    \end{equation*}
  \item[$(\psi2)$] $\psi(\omega,x,\xi)\geq 0$ for any $(\omega,x,\xi) \in \Omega\times\R^d\times\R^d$, and $\psi(\omega,x,0)=0$ for any $(\omega,x)\in \Omega\times\R^d$.
  \item[$(\psi3)$] $\psi$ is uniformly strongly convex in its final variable, i.e. there exists $c>0$ such that
      \begin{equation*}
        \xi \in \R^d \mapsto \psi(\omega,x,\xi)-c|\xi|^2\quad\text{is convex for almost every $(\omega,x)\in\Omega\times\R^d$}.
      \end{equation*}
      %    Without loss of generality, we can assume that $c \leq \underline{A}/2$, where $\underline{A}$ is defined in Assumption~$(A3)$. 
    \item[$(\psi4)$] There exists $C>0$ such that, for any $\xi\in\R^d$,
    \begin{equation*}
      \psi(\omega,x,\xi)\leq C(1+|\xi|^2)\quad\text{for almost every $(\omega,x)\in\Omega\times\R^d$}.
    \end{equation*}
  \end{enumerate}}
\end{center}

\medskip

\noindent
Under these assumptions, we prove the following result.

\begin{lemma}\label{th:convex_approx}
For any $\eta>0$, there exists a function $\psi_\eta:\Omega\times\R^d\times\R^d\to\R$, defined to be the Moreau envelope
\begin{equation}\label{eq:psi_eta}
  \psi_\eta(\omega,x,\xi) := \inf_{p\in\R^d} \left\{ \psi(\omega,x,p) + \frac{1}{2\eta} |\xi-p|^2 \right\},
\end{equation}
which satisfies the following properties:
\begin{enumerate}
\item For almost every $(\omega,x)\in\Omega\times\R^d$, the function $\xi\mapsto\psi_\eta(\omega,x,\xi)$ is $\CC^1$ and convex, with $\xi\mapsto D_\xi\psi_\eta(\omega,x,\xi)$ being uniformly Lipschitz on compact sets. \label{ii_sec5}
%\fl{on ne fait pas la preuve explicite de ceci, mais on s'appuie sur la ref~\cite{BC17}, où ils enoncent que la derivee est lipschitz; on reprend donc ici la meme formulation que dans la ref, pour faciliter la lecture et la comparaison entre notre papier et la ref; cependant, bien noter qu'on ne se sert pas de ce caractere lipschitz ici}   
\item For any $\xi\in\R^d$, the function $(\omega,x)\mapsto\psi_\eta(\omega,x,\xi)$ is measurable. It is stationary in the sense of~\eqref{eq:Stationarity}.\label{i_sec5}
\item For all $\xi\in\R^d$ and for almost every $(\omega,x)\in\Omega\times\R^d$, we have\label{iii_sec5}
\begin{equation} \label{eq:borne_psi_eta}
  0 \leq \psi_\eta(\omega,x,\xi) \leq \psi(\omega,x,\xi) \leq \psi_\eta(\omega,x,\xi)+8\eta\, C^2(1+|\xi|^2)
\end{equation}
where $C$ is the constant appearing in Assumption~$(\psi4)$.
\item There exists a constant $\mathcal{C}$, which only depends on the constants $c$ and $C$ appearing in Assumptions~$(\psi3)$ and~$(\psi4)$, such that, for almost every $(\omega,x,\xi) \in \Omega \times \R^d \times \R^d$,\label{iv_sec5}
\begin{equation} \label{eq:bound_d_psi_eta}
  \left| D_\xi \psi_\eta(\omega,x,\xi) \right| \leq \mathcal{C} \, (1+|\xi|).
\end{equation}
% \item For almost every $(\omega,x,\xi) \in \Omega \times \R^d \times \R^d$, we have that $\zeta^\eta(\omega,x,\xi):=D_\xi\psi_\eta(\omega,x,\xi)$ converges when $\eta \to 0$ to the unique element $\zeta(\omega,x,\xi)\in\partial \psi(\omega,x,\xi)$ satisfying \label{v_sec5}
% \begin{equation*}
%   \zeta(\omega,x,\xi) = \argmin_{\theta \in \partial_\xi\psi(\omega,x,\xi)} \, |\theta|.
% \end{equation*}
% %\fl{the set $\partial_\xi\psi(\omega,x,\xi)$ is convex, and thus the above argmin is indeed unique}
% \fl{I changed the wording: you said ``for every $\xi\in\R^d$ and almost every $(\omega,x)\in\Omega\times\R^d$'', I wrote ``for almost every $(\omega,x,\xi) \in \Omega \times \R^d \times \R^d$.} 
% \fl{this statement does not seem to be used now; to be checked; if true, remove it?}
\item Let $(\psi_\eta)^*$ be the Legendre--Fenchel transform of $\psi_\eta$ with respect to its third variable. For $\eta$ sufficiently small (e.g. whenever $\eta \leq c/(16 C^2)$), we have \label{vi_sec5}
\begin{equation} \label{eq:utile}
  \psi(\omega,x,\xi) \leq \frac{\psi_\eta(\omega,x,\xi)}{m_\eta} + \eta \frac{8 C^2}{m_\eta}
  \quad\text{for any $(\omega,x,\xi)\in \Omega\times\R^d\times \R^d$}
\end{equation}
and
\begin{equation} \label{eq:borne_psi_eta_star}
  (\psi_\eta)^*(\omega,x,\zeta)
  \geq
  \psi^*(\omega,x,\zeta)
  \geq
  \frac{(\psi_\eta)^*(\omega,x,m_\eta \, \zeta)}{m_\eta} - \eta \frac{8 C^2}{m_\eta}
  \quad\text{for any $(\omega,x,\zeta)\in \Omega\times\R^d\times \R^d$}
\end{equation}
with
\begin{equation} \label{eq:def_m_eta}
m_\eta := 1 - 8 \eta \, \frac{C^2}{c} \geq \frac{1}{2} > 0.
\end{equation}
Furthermore, for almost any $(\omega,x) \in \Omega\times\R^d$, the function $\zeta \mapsto (\psi_\eta)^*(\omega,x,\zeta)$ is convex.
\item The function $\psi_\eta(\omega,x,\xi)$ is (uniformly in $(\omega,x)$) strongly convex in its final variable, in the sense that the function $\dps \xi \mapsto \psi_\eta(\omega,x,\xi) - \frac{c}{1+2c\eta} |\xi|^2$ is convex. \label{vii_sec5}
\item For almost any $(\omega,x) \in \Omega\times\R^d$, the function $\zeta \mapsto (\psi_\eta)^*(\omega,x,\zeta)$ is $\CC^1$. Furthermore, there exists a constant $\mathcal{C}^*$, which only depends on the constants $c$ and $C$ appearing in Assumptions~$(\psi3)$ and~$(\psi4)$, such that, for almost every $(\omega,x,\zeta) \in \Omega \times \R^d \times \R^d$ and any $\eta \leq 1$, \label{viii_sec5}
\begin{equation} \label{eq:bound_d_psi_eta_star}
  \left| D_\zeta (\psi_\eta)^*(\omega,x,\zeta) \right| \leq \mathcal{C}^* \, (1+|\zeta|).
\end{equation}
\end{enumerate}
\end{lemma}

\begin{proof}
The definition~\eqref{eq:psi_eta} of $\psi_\eta$ follows that of the Moreau envelope defined in Section~12.4 of~\cite{BC17}, and property~\eqref{ii_sec5} follows from Propositions~12.15 and~12.29 in the same reference.
%\fl{ils demontrent dans ces propositions que c'est convexe, pas strictement convexe}

\smallskip

We note that the function $(\omega,x)\mapsto\psi_\eta(\omega,x,\xi)$ is measurable for each $\xi$, since the infimum in~\eqref{eq:psi_eta} can be taken over $\overline{p}\in\Q^d$ without changing the definition, thereby making the function an infimum over a countable collection of measurable functions, and hence itself measurable. The stationarity of $\psi$ assumed in~$(\psi1)$ directly implies that of $\psi_\eta$, since
\begin{equation}\label{eq:StationaryReg}
\begin{gathered}
  \psi_\eta(\omega,x,\xi) = \inf_{\overline{p}\in\R^d} \left\{ \psi_0\big(T(x)\omega,\overline{p}\big) + \frac{1}{2\eta} \, |\xi-\overline{p}|^2 \right\} = \psi_{0,\eta}\big(T(x)\omega,\xi\big),
  \\[1mm]
  \text{where} \qquad \psi_{0,\eta}(\omega,\xi) := \inf_{\overline{p}\in\R^d} \left\{ \psi_0(\omega,\overline{p}) + \frac{1}{2\eta} |\xi-\overline{p}|^2 \right\} \quad \text{for all $(\omega,\xi)\in\Omega\times\R^d$}.
\end{gathered}
\end{equation}
This proves property~\eqref{i_sec5}.

\smallskip

We now turn to proving property~\eqref{iii_sec5}. Since this property is expected to hold uniformly for almost every $(\omega,x)\in\Omega\times\R^d$, for convenience we set $f(\xi):= \psi(\omega,x,\xi)$ and $f_\eta(\xi):=\psi_\eta(\omega,x,\xi)$. Applying assumption~$(\psi2)$, it is clear that $\dps 0=\inf_{\xi \in \R^d} f(\xi) \leq f_\eta(\xi) \leq f(\xi)$, with the upper inequality being a consequence of the fact that $p=\xi$ is a competitor in the minimisation problem~\eqref{eq:psi_eta} defining $\psi_\eta$. It thus remains to show the upper bound in~\eqref{eq:borne_psi_eta}. To that aim, we consider
\begin{equation*}
  f(\xi)-f_\eta(\xi) = \sup_{p \in \R^d}\left\{ f(\xi)-f(p) -\frac{1}{2\eta}|\xi-p|^2\right\}.
\end{equation*}
The convexity of $f$ and the bound~\eqref{eq:subdiff_bnd}, which itself is a consequence of assumption~$(\psi4)$, together entail that, for any $\xi$ and $p$,
\begin{equation*}
  f(\xi)- f(p)\leq 4C\sqrt{1+|\xi|^2} \ |\xi-p|.
\end{equation*}
Using this estimate and explicitly solving to obtain an upper bound, it follows that
\begin{equation*}
  f(\xi)-f_\eta(\xi)
  \leq
  \sup_{p\in\R^d}\left\{ 4C\sqrt{1+|\xi|^2} \ |\xi-p| -\frac{1}{2\eta}|\xi-p|^2\right\}
  \leq
  8\eta\,C^2\b(1+|\xi|^2\b),
\end{equation*}
which implies the upper bound in~\eqref{eq:borne_psi_eta}. This proves property~\eqref{iii_sec5}.

\smallskip

We now establish property~\eqref{iv_sec5}. Let $p_\xi^\eta\in\R^d$ be the unique minimiser of the problem~\eqref{eq:psi_eta} defining $\psi_\eta$. We have that
\begin{equation*}
  f_\eta(\xi) = f(p_\xi^\eta)+\frac{1}{2\eta}|\xi-p_\xi^\eta|^2.
\end{equation*}
Applying Proposition~12.29 of~\cite{BC17} and using the fact that $p_\xi^\eta$ is a minimiser, we deduce that
\begin{equation*}
  D f_\eta(\xi) = \frac{1}{\eta}\b(\xi-p_\xi^\eta\b) \in \partial f(p_\xi^\eta).
\end{equation*}
Applying the bound~\eqref{eq:subdiff_bnd}, we therefore have
\begin{equation}\label{eq:gradpsieta_bnd}
\b|Df_\eta(\xi)\b| = \frac{1}{\eta}|\xi-p_\xi^\eta| \leq 4 \, C \, \sqrt{1+|p_\xi^\eta|^2}.
\end{equation}
Now, using~\eqref{eq:psi_lower_bound}, property~$(\psi4)$ along with the definition~\eqref{eq:psi_eta} of the Moreau envelope, we find that
\begin{equation*}
  c|p_\xi^\eta|^2\leq f(p_\xi^\eta)\leq f_\eta(\xi)\leq f(\xi)\leq C\b(1+|\xi|^2\b).
\end{equation*}
Combining this observation with~\eqref{eq:gradpsieta_bnd}, we obtain property~\eqref{iv_sec5}.

% \smallskip

% We note that property~\eqref{v_sec5} follows directly from Theorem~2.1 in~\cite{AA93}. \fl{if property~\eqref{v_sec5} is not used, then remove this sentence}

\smallskip

We next establish property~\eqref{vi_sec5}. Due to the ordering property of the Legendre--Fenchel transform, we deduce from~\eqref{eq:borne_psi_eta} the following upper bound:
\begin{equation} \label{eq:upper_bound}
(\psi_\eta)^*(\omega,x,\zeta)
\geq
\psi^*(\omega,x,\zeta)
\quad\text{for any $(\omega,x,\zeta)\in \Omega\times\R^d\times \R^d$}.
\end{equation}
We now establish a lower bound on $\psi^*$. In view of~\eqref{eq:borne_psi_eta} and~\eqref{eq:psi_lower_bound}, we write that
$$
\psi(\omega,x,\xi) \leq \psi_\eta(\omega,x,\xi) + 8 \eta \, C^2 (1+|\xi|^2) \leq \psi_\eta(\omega,x,\xi) + 8 \eta \, C^2 + 8 \eta \, \frac{C^2}{c} \, \psi(\omega,x,\xi).
$$
Whenever $\eta \leq c/(16 C^2)$, we see that $m_\eta = 1 - 8 \eta \, C^2/c \geq 1/2 > 0$, and thus
$$
\psi(\omega,x,\xi) \leq \frac{\psi_\eta(\omega,x,\xi)}{m_\eta} + \eta \frac{8 C^2}{m_\eta},
$$
which is~\eqref{eq:utile}. We deduce from the above inequality that 
\begin{equation} \label{eq:lower_bound}
  \psi^*(\omega,x,\zeta)
  \geq
  \frac{(\psi_\eta)^*(\omega,x,m_\eta \, \zeta)}{m_\eta} - \eta \frac{8 C^2}{m_\eta}
  \qquad \text{for any $(\omega,x,\zeta)\in \Omega\times\R^d\times \R^d$}.
\end{equation}
Collecting~\eqref{eq:upper_bound} and~\eqref{eq:lower_bound}, we obtain~\eqref{eq:borne_psi_eta_star}. In addition, the function $(\psi_\eta)^*(\omega,x,\zeta)$ is convex in $\zeta$ because it is the Legendre--Fenchel transform, and so a supremum of convex functions (see~\eqref{eq:LFEquivalence} and the discussion in Section~\ref{sec:InitialExample}).

\smallskip

We now turn to establishing property~\eqref{vii_sec5}. In view of property~\eqref{ii_sec5}, we know that $\psi_\eta$ is convex. Following the proof of Proposition 8.26 of~\cite{BC17}, we now establish a more precise statement. For convenience, we again set $f(\xi) = \psi(\omega,x,\xi)$ and $f_\eta(\xi) = \psi_\eta(\omega,x,\xi)$. Consider $\xi_1$ and $\xi_2$ in $\R^d$ and some $\alpha \in (0,1)$. For $i=1,2$, we choose some $\Lambda_i > f_\eta(\xi_i)$, so that there exists some $p_i \in \R^d$ such that $\dps f(p_i) + |\xi_i-p_i|^2/(2\eta) < \Lambda_i$. By definition of $f_\eta\b(\alpha \xi_1 + (1-\alpha) \xi_2\b)$ and using the strong convexity of $f$ and $\cdot \mapsto | \cdot |^2$, we have
\begin{align}
  & f_\eta\b(\alpha \xi_1 + (1-\alpha) \xi_2\b)
  \nonumber
  \\
  &\leq
  f\b(\alpha p_1 + (1-\alpha) p_2\b) + \frac{1}{2\eta} \big| \alpha (\xi_1-p_1) + (1-\alpha) (\xi_2-p_2) \big|^2
  \nonumber
  \\
  &\leq
  \alpha f(p_1) + (1-\alpha) f(p_2) - c \, \alpha(1-\alpha) \, |p_1 - p_2|^2 + \frac{\alpha}{2\eta} |\xi_1-p_1|^2 + \frac{1-\alpha}{2\eta} |\xi_2-p_2|^2 - \frac{\alpha(1-\alpha)}{2\eta} \big| \xi_1-p_1 - \xi_2+p_2 \big|^2
  \nonumber
  \\
  &\leq
  \alpha \Lambda_1 + (1-\alpha) \Lambda_2 - \frac{\alpha(1-\alpha)}{2} R
  \label{eq:combettes}
\end{align}
with
\begin{align*}
  R
  &=
  2 c |p_1 - p_2|^2 + \frac{1}{\eta} \big| (\xi_1-p_1) - (\xi_2-p_2) \big|^2
  \\
  &=
  \left(2c+\frac{1}{\eta} \right) |p_1 - p_2|^2 + \frac{1}{\eta} |\xi_1 - \xi_2|^2 - \frac{2}{\eta} (\xi_1-\xi_2) \cdot (p_1-p_2)
  \\
  & \geq
  \left(2c+\frac{1}{\eta} \right) |p_1 - p_2|^2 + \frac{1}{\eta} |\xi_1 - \xi_2|^2 - \frac{1}{2 \tau \eta} |\xi_1 - \xi_2|^2 - \frac{2 \tau}{\eta} |p_1 - p_2|^2,
\end{align*}
where we have used Young's inequality in the last line for some $\tau > 0$. Choosing $\tau$ such that $2 \tau / \eta = 2c + 1/\eta$, we deduce that
$$
R \geq \frac{1}{\eta} \left( 1 - \frac{1}{2\tau} \right) |\xi_1 - \xi_2|^2 = \frac{2c}{1+2c \eta} |\xi_1 - \xi_2|^2.
$$
Introducing this lower bound in~\eqref{eq:combettes} and passing to the limit $\Lambda_i \to f_\eta(\xi_i)$ for $i=1,2$, we deduce that
$$
f_\eta\b(\alpha \xi_1 + (1-\alpha) \xi_2\b)
\leq
\alpha f_\eta(\xi_1) + (1-\alpha) f_\eta(\xi_2) - \alpha(1-\alpha) \, \frac{c}{1+2c \eta} \, |\xi_1 - \xi_2|^2.
$$
The function $\psi_\eta(\omega,x,\xi)$ is thus strongly convex with respect to $\xi$, in the sense that the function $\dps \xi \mapsto \psi_\eta(\omega,x,\xi) - \frac{c}{1+2c\eta} |\xi|^2$ is convex.

\smallskip

Finally, we establish property~\eqref{viii_sec5}. 
Since $\xi \mapsto \psi_\eta(\omega,x,\xi)$ is $\CC^1$ (see property~\eqref{ii_sec5}), we deduce from the strong convexity of $\psi_\eta$ that, for any $\xi_1$ and $\xi_2$ in $\R^d$, we have
\begin{equation} \label{eq:combettes2}
\frac{2c}{1+2c \eta} | \xi_1 - \xi_2 | \leq | D_\xi \psi_\eta(\omega,x,\xi_1) - D_\xi \psi_\eta(\omega,x,\xi_2) |.
\end{equation}
Now consider $\zeta \in \R^d$. The Legendre--Fenchel transform of $\psi_\eta$ is
$$
(\psi_\eta)^*(\omega,x,\zeta) = \sup_{p \in \R^d} \{ p \cdot \zeta - \psi_\eta(\omega,x,p) \},
$$
and we denote $p^\zeta \in \R^d$ the unique supremizer of the above problem, which satisfies $\zeta = D_\xi \psi_\eta(\omega,x,p^\zeta)$. We note that $(\psi_\eta)^*$ is differentiable by Proposition~18.9 in \cite{BC17}, and we have $D_\zeta (\psi_\eta)^*(\omega,x,\zeta) = p^\zeta$. Thus, using~\eqref{eq:combettes2} with $\xi_1 = p^\zeta$ and $\xi_2 = 0$, we obtain
$$
\left| D_\zeta (\psi_\eta)^*(\omega,x,\zeta) \right| \leq \frac{1+2c\eta}{2c} | \zeta - D_\xi \psi_\eta(\omega,x,0) |.
$$
Using~\eqref{eq:bound_d_psi_eta} and restricting ourselves to the case $\eta \leq 1$, we deduce~\eqref{eq:bound_d_psi_eta_star}.

We use a similar argument to show that $D_\zeta (\psi_\eta)^*$ is Lipschitz-continuous with respect to its third variable. Consider $\zeta_i \in \R^d$ with $i=1,2$. As above, we define $p^{\zeta_i} \in \R^d$ which satisfies $\zeta_i = D_\xi \psi_\eta(\omega,x,p^{\zeta_i})$. We then have $D_\zeta (\psi_\eta)^*(\omega,x,\zeta_i) = p^{\zeta_i}$, and thus, using~\eqref{eq:combettes2} with $\xi_1 = p^{\zeta_1}$ and $\xi_2 = p^{\zeta_2}$, we obtain
$$
\left| D_\zeta (\psi_\eta)^*(\omega,x,\zeta_1) - D_\zeta (\psi_\eta)^*(\omega,x,\zeta_2) \right| \leq \frac{1+2c\eta}{2c} | \zeta_1 - \zeta_2 |.
$$
We hence get that $D_\zeta (\psi_\eta)^*(\omega,x,\cdot)$ is indeed Lipschitz continuous (with a Lipschitz constant uniform in $\eta$ and $(\omega,x)$). This of course implies that $(\psi_\eta)^*(\omega,x,\cdot)$ is $\CC^1$. This completes the proof of Lemma~\ref{th:convex_approx}.
\end{proof}
}

\blue{
\begin{lemma}\label{th:dual_char}
Let $\eta > 0$, and suppose that $\Psi^\eta_{t,\omega}:\HBC\to\R$ takes the form
\begin{equation} \label{eq:def_precise}
  \Psi^\eta_{t,\omega}[v] = \int_\Dom \psi_\eta\left(\omega,x,D_x\dot{\overline{y}}(t)+D_xv\right)\dx,
\end{equation}
where $\overline{y}\in\HH^1([0,T];\HH^1(\Dom))$ and $\psi_\eta$ is defined in Lemma~\ref{th:convex_approx}, assuming that $\psi$ satisfies the assumptions of Section~\ref{sec:dissipation_assumptions}. For any $u\in\LL^2([0,T]\times\Omega;\HBC)$, consider $f = \div_x \big( A(\omega,x) \left(D_x\overline{y}(t)+D_xu(t)\right) \big) \in \LL^2\left([0,T]\times\Omega;\HBCPrime\right)$, where $A$ satisfies the assumptions of Section~\ref{sec:elasticity_assumptions}. Assume furthermore that $\eta$ is sufficiently small (e.g. $\eta \leq c/(16 C^2)$), so that $m_\eta$ defined by~\eqref{eq:def_m_eta} is positive and bounded away from 0. 

Then there exists some $\sigma_f^\eta \in \LL^2([0,T]\times\Omega\times\Dom)^d$ such that $\sigma_f^\eta(t,\omega) \in \L2Pot(\Dom)^\perp$ for almost every $(t,\omega)\in[0,T]\times\Omega$, which satisfies the bound
\begin{equation} \label{eq:bound_sigma}
  \| \sigma_f^\eta \|_{\LL^2([0,T]\times\Omega\times\Dom)^d} \leq K \Big( 1+\left\| D_x\overline{y} \right\|_{\HH^1([0,T];\LL^2(\Dom))^d} + \|D_xu\|_{\LL^2([0,T]\times\Omega\times\Dom)^d} \Big)
\end{equation}
for some constant $K$ which depends only on $\Dom$, $T$ and on the bounds assumed on $A$ and $\psi$, and for which we can write
\begin{equation} \label{eq:Psi_star}
  (\Psi^\eta_{t,\omega})^*[f] = \int_\Dom (\psi_\eta)^*\big(\omega,x,\sigma^\eta_f-A(\omega,x)D_x\overline{y}-A(\omega,x)D_xu\big) -\b(\sigma_f^\eta-A(\omega,x)D_x\overline{y}-A(\omega,x)D_xu\b) \cdot D_x\dot{\overline{y}} \, \dx
\end{equation}
for almost every $(t,\omega)\in[0,T]\times\Omega$.
\end{lemma}

\begin{proof}
We recall that, for any $f\in\HBCPrime$, we have
\begin{align}
    (\Psi^\eta_{t,\omega})^*[f] &= \sup\left\{ \< f, v\>_{\HBCPrime,\HBC} - \int_\Dom \psi_\eta\left(\omega,x,D_x\dot{\overline{y}}(t)+D_xv\right) \dx, \quad v\in\HBC\right\}\notag
    \nonumber
    \\
    &=-\inf\left\{\int_\Dom \psi_\eta\left(\omega,x,D_x\dot{\overline{y}}(t)+D_xv\right)\dx-\<f,v\>_{\HBCPrime,\HBC}, \quad v\in\HBC\right\}.
    \label{eq:sup/inf_prob}
\end{align}
A straightforward application of the Direct Method of the Calculus of Variations entails that a unique solution to the latter minimisation problem exists for almost every $(t,\omega)\in[0,T]\times\Omega$, which we denote $v_f^\eta(t,\omega,\cdot) \in \HBC$. The Euler-Lagrange equations read as follows: for almost every $(t,\omega)\in[0,T]\times\Omega$, the function $\zeta_f^\eta(t,\omega,x) = D_\xi\psi_\eta\left(\omega,x,D_x\dot{\overline{y}}(t,x)+D_xv_f^\eta(t,\omega,x)\right)$ satisfies
$$
\int_\Dom \zeta_f^\eta(t,\omega,x) \cdot D_x w(x) \dx = \<f,w\>_{\HBCPrime,\HBC} \qquad\text{for any $w\in\HBC$}.
$$
Using the regularity of $v_f^\eta$ and~\eqref{eq:bound_d_psi_eta}, we see that $\zeta_f^\eta(t,\omega,\cdot) \in \LL^2(\Dom)^d$ for almost every $(t,\omega)\in[0,T]\times \Omega$. Since $f = \div_x \big( A(\omega,x) \left(D_x\overline{y}(t)+D_xu(t)\right) \big)$ we infer that for almost every $(t,\omega)\in[0,T]\times\Omega$
\begin{equation} \label{eq1_a} 
0 = \int_\Dom \b[\zeta_f^\eta(t,\omega,x) + A(\omega,x) \left(D_x\overline{y}(t) + D_xu(t) \right) \b] \cdot D_xw(x)\dx \qquad\text{for any $w \in \HBC$}.
\end{equation}
We now establish uniform bounds on $v_f^\eta$ and $\zeta_f^\eta$. To get a bound on $v_f^\eta$, we use the test function $v=0$ in~\eqref{eq:sup/inf_prob}, the bounds~\eqref{eq:utile} and~\eqref{eq:psi_lower_bound} and the specific choice of $f$:
\begin{align*}
\int_\Dom \psi_\eta\left(\omega,x,D_x\dot{\overline{y}}\right)\dx
&\geq
\int_\Dom \psi_\eta\left(\omega,x,D_x\dot{\overline{y}}+D_xv_f^\eta\right) \dx - \<f,v_f^\eta\>
\\
&\geq
m_\eta \int_\Dom \psi\left(\omega,x,D_x\dot{\overline{y}}+D_xv_f^\eta\right) \dx - \<f,v_f^\eta\> - 8 \eta C^2 \Leb^d(\Dom)
\\
& \geq
c \, m_\eta \left\| D_x\dot{\overline{y}}+D_xv_f^\eta \right\|_{\LL^2(\Dom)}^2 + \int_\Dom D_xv_f^\eta \cdot A(\omega,x) \left(D_x\overline{y}+D_xu\right) - 8 \eta C^2 \Leb^d(\Dom)
\\
& \geq
c \, m_\eta \left\| D_x\dot{\overline{y}}+D_xv_f^\eta \right\|_{\LL^2(\Dom)}^2 - \overline{A} \, \left\| D_xv_f^\eta \right\|_{\LL^2(\Dom)} \| D_x\overline{y}+D_xu \|_{\LL^2(\Dom)} - 8 \eta C^2 \Leb^d(\Dom),
\end{align*}
where $\Leb^d(\Dom)$ denotes the volume of the domain $\Dom$.

We hence deduce, using~$(\psi4)$, \eqref{eq:borne_psi_eta} and Young's inequality, that
\begin{align*}
& C \left( \Leb^d(\Dom) + \left\| D_x\dot{\overline{y}} \right\|_{\LL^2(\Dom)}^2 \right)
\\
&\geq
\int_\Dom \psi\left(\omega,x,D_x\dot{\overline{y}}(t)\right)\dx
\\
&\geq
\int_\Dom \psi_\eta\left(\omega,x,D_x\dot{\overline{y}}(t)\right)\dx
\\
&\geq
\frac{c m_\eta}{2} \, \| D_xv_f^\eta \|_{\LL^2(\Dom)}^2 - c \, m_\eta \left\| D_x\dot{\overline{y}} \right\|_{\LL^2(\Dom)}^2 - \overline{A} \, \| D_xv_f^\eta \|_{\LL^2(\Dom)} \| D_x\overline{y}+D_xu \|_{\LL^2(\Dom)} - 8 \eta C^2 \Leb^d(\Dom)
\\
&\geq
\frac{c m_\eta}{4} \, \| D_xv_f^\eta \|_{\LL^2(\Dom)}^2 - c \, m_\eta \left\| D_x\dot{\overline{y}} \right\|_{\LL^2(\Dom)}^2 - \frac{\overline{A}^2}{c m_\eta} \, \| D_x\overline{y}+D_xu \|^2_{\LL^2(\Dom)} - 8 \eta C^2 \Leb^d(\Dom).
\end{align*}
Integrating over $(t,\omega)$ and using that $m_\eta$ is both bounded above and away from 0, we find that
$$
\|D_x v_f^\eta\|_{\LL^2([0,T]\times\Dom\times\Omega)^d}\leq K \Big( 1+\left\| D_x \overline{y} \right\|_{\HH^1([0,T]; \LL^2(\Dom))^d} + \|D_xu\|_{\LL^2([0,T]\times\Omega\times\Dom)^d} \Big)
$$
for some constant $K$ which depends only on $\Dom$, $T$ and on the bounds assumed on $A$ and $\psi$. Next, using~\eqref{eq:bound_d_psi_eta} and the above estimate, we find that
\begin{align}
  \|\zeta_f^\eta\|^2_{\LL^2([0,T]\times\Dom\times\Omega)^d}
  &\leq 2 \mathcal{C}^2\left(T\Leb^d(\Dom)+\|D_x\dot{\overline{y}}+D_xv_f^\eta\|^2_{\LL^2([0,T]\times\Omega\times\Dom)^d}\right)
  \nonumber
  \\
  &\leq K \left( 1 + \|D_x\dot{\overline{y}}\|^2_{\LL^2([0,T]\times\Dom)^d} + \left\| D_x \overline{y} \right\|^2_{\HH^1([0,T]; \LL^2(\Dom))^d} + \|D_xu\|^2_{\LL^2([0,T]\times\Omega\times\Dom)^d} \right),
  \label{eq:estim_zeta_f}
\end{align}
and hence $\zeta^\eta_f\in\LL^2([0,T]\times\Omega\times\Dom)$.
We now introduce $\sigma_f^\eta := \zeta_f^\eta + A(\omega,x) \left(D_x\overline{y} + D_xu \right)$, which lies in $\LL^2([0,T]\times\Omega\times\Dom)$ as a consequence of the fact that $\zeta_f^\eta \in \LL^2([0,T]\times\Omega\times\Dom)$ and the properties of $A$, $\overline{y}$ and $u$. In view of~\eqref{eq1_a}, we have that $\sigma_f^\eta(t,\omega) \in \L2Pot(\Dom)^\perp$ for almost every $(t,\omega)\in[0,T]\times\Omega$. In view of the bound~\eqref{eq:estim_zeta_f} on $\zeta_f^\eta$, we get that $\sigma_f^\eta$ satisfies~\eqref{eq:bound_sigma}.

\smallskip

Finally, we show~\eqref{eq:Psi_star}. Since $\dps \zeta_f^\eta(t,\omega,x) = D_\xi\psi_\eta\b(\omega,x,D_x\dot{\overline{y}}(t)+D_xv_f^\eta\b)$ for almost every $(t,\omega,x)$, we infer from the property~\eqref{eq:LFEquivalence} of the Legendre--Fenchel transform that
\begin{equation} \label{eq:ralf}
  \psi_\eta\left(\omega,x,D_x\dot{\overline{y}}+D_x v_f^\eta\right)-\zeta_f^\eta \cdot \b(D_x\dot{\overline{y}}+D_xv_f^\eta\b) = -(\psi_\eta)^*(\omega,x,\zeta_f^\eta) \qquad \text{a.e. in $(t,\omega,x)$}.
\end{equation}
Recalling that $v_f^\eta$ achieves the infimum sought in~\eqref{eq:sup/inf_prob}, we find that, for almost every $(t,\omega)$,
\begin{align*}
  -(\Psi^\eta_{t,\omega})^*[f]
  &=
  \int_\Dom\psi_\eta\left(\omega,x,D_x\dot{\overline{y}}+D_x v_f^\eta\right) - \<f,v_f^\eta\>_{\HBCPrime,\HBC}
  \\
  &=
  \int_\Dom\psi_\eta\left(\omega,x,D_x\dot{\overline{y}}+D_x v_f^\eta\right) + \int_\Dom D_xv_f^\eta \cdot A(\omega,x) \left(D_x\overline{y}+D_xu\right)
  \\  
  &= \int_\Dom \zeta_f^\eta \cdot \b(D_x\dot{\overline{y}}+D_xv_f^\eta\b) - (\psi_\eta)^*(\omega,x,\zeta_f^\eta) + \int_\Dom D_xv_f^\eta \cdot (\sigma_f^\eta - \zeta_f^\eta), 
\end{align*}
where, in the last line, we have used~\eqref{eq:ralf} and the relation between $\sigma_f^\eta$ and $\zeta_f^\eta$. Using that $\sigma_f^\eta(t,\omega) \in \L2Pot(\Dom)^\perp$, we obtain~\eqref{eq:Psi_star}. This concludes the proof of Lemma~\ref{th:dual_char}.
\end{proof}
}

\begin{remark}
\blue{We note that, when $\Dom=(a,b)\subset\R$ and $\GDir=\{a,b\}$, there is a particularly precise characterisation of $\sigma_f^\eta$. In this case, $\L2Pot(\Dom)$ is the space of mean--zero functions, and $\L2Pot(\Dom)^\perp$ is simply the space of constant functions. By definition, $\zeta_f^\eta = D_\xi\psi_\eta(\omega,x,D_x\dot{\overline{y}}+D_xv_f^\eta)$, hence $D_x\dot{\overline{y}}+D_xv_f^\eta = D_\sigma(\psi_\eta)^*(\omega,x,\zeta_f^\eta)$ (see~\eqref{eq:LFEquivalence}). We thus obtain that $D_\sigma(\psi_\eta)^*(\omega,\cdot,\zeta_f^\eta) \in \L2Pot(\Dom)$ a.e. in $(t,\omega)$. Using the relation between $\zeta_f^\eta$ and $\sigma_f^\eta$, it follows that, for almost every $(t,\omega)\in[0,T]\times\Omega$, $\sigma_f^\eta(t,\omega)\in\R$ is the constant such that 
\begin{equation*}
  \int_\Dom D_\sigma(\psi_\eta)^*\big(\omega,x,\sigma_f^\eta(t,\omega)-A \, D_x\overline{y} - A \, D_xu \big) \dx = 0.
\end{equation*}
This characterisation of $\sigma_f^\eta$ (in the regime $\eta \to 0$) was exploited in the numerical scheme presented in Section~\ref{sec:numerics}.}
\end{remark}

\section{Existence of solutions for fixed $\eps$}
\label{sec:wellposedness}

With the preliminaries of Section~\ref{sec:setup} now in place, we restate Theorem~\ref{th:eps_existence} in a precise form, which asserts that problem~\eqref{eq:EvolutionProblem} (namely, the evolution problem at fixed $\eps > 0$) is well--posed.

\begin{theorem}[Rigorous statement of Theorem~\ref{th:eps_existence}]
For any $\eps>0$ and any $\overline{y} \in \HH^1\b([0,T];\HH^1(\Dom)\b)$, there exists a unique solution $u^\eps \in \HH^1\big([0,T];\LL^2(\Omega;\HBC)\big)$ such that~\eqref{eq:EvolutionProblem} holds for almost every $(t,\omega)\in[0,T]\times\Omega$.
\end{theorem}  

\medskip
  
The proof of this result is given over the course of the present section. The strategy of the proof is to apply the Banach fixed point theorem in a similar manner to its use in the proof of the Cauchy--Lipschitz existence theorem for ODEs.

\subsection{The velocity operator}
\label{sec:velocity-operator}

We first observe that the functionals $u \mapsto \Phi^\eps_{t,\omega}[u] = \Phi^\eps_\omega[\overline{y}(t)+u]$ and $v \mapsto \Psi^\eps_{t,\omega}[v] = \Psi^\eps_\omega[\dot{\overline{y}}(t)+v]$ are uniformly strongly convex on $\HBC$, since $\Phi^\eps_\omega$ and $\Psi^\eps_\omega$ are. In order to assess their convexity constants, we proceed as follows. We note that, for any $(\omega,x,\xi)$, the function
\begin{equation*}
  \zeta\mapsto W(\omega,x,\xi+\zeta) - \frac12 \, \underline{A} \, |\zeta|^2 = W(\omega,x,\xi+\zeta) -\frac12 \, \underline{A} \, |\xi+\zeta|^2 + \underline{A}\,\xi\cdot\zeta + \frac12 \, \underline{A} \, |\xi|^2
\end{equation*}
is convex, since it is the sum of two convex functions by Assumption~$(A3)$. The same property holds for the function $\zeta\mapsto\psi(\omega,x,\xi+\zeta)-c|\zeta|^2$, in view of Assumption~$(\psi3)$. It therefore follows that $\Phi^\eps_{t,\omega}[u]$ and $\Psi^\eps_{t,\omega}[v]$ are uniformly strongly convex on $\HBC$, since
\begin{align*}
u \mapsto \Phi^\eps_{t,\omega}[u] - \frac12 \, \underline{A} \, \|u\|_{\HBC}^2 &= \int_\Dom W\left(\omega,\frac{\cdot}{\eps},D_x\overline{y}(t)+D_xu\right) - \frac12 \, \underline{A} \, |D_xu|^2\dx
\\
v \mapsto \Psi^\eps_{t,\omega}[v] -c \|v\|_{\HBC}^2 &= \int_\Dom \psi\left(\omega,\frac{\cdot}{\eps},D_x\dot{\overline{y}}(t)+D_xv\right)-c|D_xv|^2\dx
\end{align*}
are both convex functions on $\HBC$.

\medskip

Our first step towards proving existence of solutions is to consider the problem of finding $v\in\HBC$ such that, for any fixed $u\in\HBC$,
\begin{equation}
  0\in\partial\Psi^\eps_{t,\omega}[v]+\nabla \Phi^\eps_{t,\omega}[u].
  \label{eq:Vel_eqn}
\end{equation}
We note first that~\eqref{eq:Vel_eqn} may be viewed as a necessary condition for $v$ being a minimizer of the functional $\mathcal{I}^\eps_{t,\omega}:\HBC\to\R$ defined by
\begin{equation*}
  \mathcal{I}^\eps_{t,\omega}[v] := \Psi^\eps_{t,\omega}[v] +\<\nabla \Phi^\eps_{t,\omega}[u],v\>_{\HBC}.
\end{equation*}
This functional is well--defined, and is uniformly strongly convex on $\HBC$, a property which it inherits from $\Psi^\eps_{t,\omega}$. Thus it has a unique minimizer, denoted $\Vel^\eps_{t,\omega}[u]$, which satisfies~\eqref{eq:Vel_eqn}.

\medskip
  
We now show that $\Vel^\eps_{t,\omega}:\HBC \to \HBC$ is a Lipschitz map. Let $u_1,u_2\in\HBC$: then
\begin{equation*}
  -\nabla \Phi^\eps_{t,\omega}[u_i]\in\partial \Psi^\eps_{t,\omega}\b[\Vel^\eps_{t,\omega}[u_i]\b]\quad\text{for both $i=1$ and $i=2$}.
\end{equation*}
Since $\Psi^\eps_{t,\omega}$ is strongly convex, it follows that $\partial\Psi^\eps_{t,\omega}$ is a strongly monotone set--valued mapping (see Chapter~11 of~\cite{Roubicek}), and hence
\begin{equation}
  \b\<-\nabla \Phi^\eps_{t,\omega}[u_1]+\nabla \Phi^\eps_{t,\omega}[u_2],\Vel^\eps_{t,\omega}[u_1]-\Vel^\eps_{t,\omega}[u_2]\b\>_{\HBC} \geq 2c \, \B\| \Vel^\eps_{t,\omega}[u_1]-\Vel^\eps_{t,\omega}[u_2] \B\|^2_{\HBC}.
  \label{eq:Sepst_lower_bnd}
\end{equation}
Using the explicit form~\eqref{eq:nabla_Phi} of $\nabla\Phi^\eps_\omega$, we obtain that
\begin{equation}
  \b\|\nabla \Phi^\eps_{t,\omega}[u_1]-\nabla\Phi^\eps_{t,\omega} [u_2]\b\|_{\HBCPrime} \leq \| A\|_{\LL^\infty(\Omega \times \R^d)}\|D_x u_1 - D_x u_2\|_{\LL^2(\Dom)}\leq \overline{A} \, \|u_1-u_2\|_{\HBC}.
  \label{eq:Sepst_upper_bnd}
\end{equation}
Collecting~\eqref{eq:Sepst_lower_bnd} and~\eqref{eq:Sepst_upper_bnd}, we obtain
\begin{equation}
  \B\|\Vel^\eps_{t,\omega}[u_2]-\Vel^\eps_{t,\omega}[u_1]\B\|_{\HBC} \leq \frac{\overline{A}}{2c} \ \|u_2-u_1\|_{\HBC},
  \label{eq:Vel_Lipschitz}
\end{equation}
which demonstrates that $\Vel^\eps_{t,\omega}$ is Lipschitz continuous, uniformly for almost every $(t,\omega)\in[0,T]\times\Omega$.
  
\subsection{A priori bound}
\label{sec:a_priori_bound}

We now establish an \emph{a priori} bound on $\Vel^\eps_{t,\omega}[u(t)]$. Note first that, if $u\in\LL^2([0,T]\times\Omega;\HBC)$, then $\Vel^\eps_{t,\omega}[u(t,\omega)]\in\HBC$ for almost every $(t,\omega)\in[0,T]\times\Omega$. The fact that $-\nabla\Phi^\eps_{t,\omega}[u(t,\omega)] \in\partial\Psi^\eps_{t,\omega}\b[\Vel^\eps_{t,\omega}[u(t,\omega)]\b]$ entails that
\begin{equation*}
  \Big\<-\nabla\Phi^\eps_{t,\omega}[u(t,\omega)],w\Big\>_{\HBC}\leq \Psi^\eps_{t,\omega}\Big[\Vel^\eps_{t,\omega}[u(t,\omega)]+w\Big]-\Psi^\eps_{t,\omega}\Big[\Vel^\eps_{t,\omega}[u(t,\omega)]\Big] \quad\text{for any $w\in\HBC$}.
\end{equation*}
By taking $w = -\Vel^\eps_{t,\omega}[u(t,\omega)]$ and rearranging, we obtain
\begin{equation*}
  \Psi^\eps_{t,\omega}\Big[\Vel^\eps_{t,\omega}[u(t,\omega)]\Big] \leq \Psi^\eps_{t,\omega}[0]-\Big\<\nabla\Phi^\eps_{t,\omega} [u(t,\omega)],\Vel^\eps_{t,\omega}[u(t,\omega)]\Big\>_{\HBC}.
\end{equation*}
Using the growth condition~\eqref{eq:psi_lower_bound} to estimate the left-hand side from below, and assumptions~$(A3)$ and~$(\psi4)$ to estimate the right-hand side from above, we find that
\begin{align*}
  & c \, \Big\| D_x\dot{\overline{y}}(t) + D_x\Vel^\eps_{t,\omega}[u(t,\omega)] \Big\|_{\LL^2(\Dom)}^2
  \\
  & \leq C \, \Big(1+\left\|D_x\dot{\overline{y}}(t)\right\|_{\LL^2(\Dom)}^2\Big) + \overline{A} \, \Big\| D_x\overline{y}(t)+D_xu(t,\omega)\Big\|_{\LL^2(\Dom)} \Big\| D_x\Vel^\eps_{t,\omega}[u(t,\omega)] \Big\|_{\LL^2(\Dom)}
  \\
  & \leq C \, \Big(1+\left\|D_x\dot{\overline{y}}(t)\right\|_{\LL^2(\Dom)}^2\Big)
  +\frac{\overline{A}^2}{c} \, \Big\| D_x\overline{y}(t)+D_xu(t,\omega)\Big\|^2_{\LL^2(\Dom)} + \frac{c}{4} \, \Big\|D_x\Vel^\eps_{t,\omega}[u(t,\omega)]\Big\|_{\LL^2(\Dom)}^2.
\end{align*}
Using Young's inequality, we write that
\begin{align*}
  & \frac{3c}{4} \, \Big\|D_x\Vel^\eps_{t,\omega}[u(t,\omega)]\Big\|^2_{\LL^2(\Dom)}
  \\
  & \leq
  c \, \Big\| D_x\dot{\overline{y}}(t) + D_x\Vel^\eps_{t,\omega}[u(t,\omega)] \Big\|_{\LL^2(\Dom)}^2 + 3c \, \Big\| D_x\dot{\overline{y}}(t) \Big\|_{\LL^2(\Dom)}^2
  \\
  & \leq
  C \Big[1 \! + \! \left\|D_x\dot{\overline{y}}(t)\right\|_{\LL^2(\Dom)}^2\Big]
  +\frac{2\overline{A}^2}{c} \! \left[ \left\| D_x\overline{y}(t) \right\|^2_{\LL^2(\Dom)} \! + \! \left\| D_xu(t,\omega)\right\|^2_{\LL^2(\Dom)}\right] \! + \frac{c}{4} \Big\|D_x\Vel^\eps_{t,\omega}[u(t,\omega)]\Big\|_{\LL^2(\Dom)}^2 \! + 3c \Big\| D_x\dot{\overline{y}}(t) \Big\|_{\LL^2(\Dom)}^2.
\end{align*}
Upon rearranging, we obtain
\begin{equation*}
\|\Vel^\eps_{t,\omega}[u(t,\omega)]\b\|_{\HBC}^2
\leq
C_0 \B( 1 + \left\| D_x\dot{\overline{y}}(t) \right\|_{\LL^2(\Dom)}^2 + \left\| D_x\overline{y}(t) \right\|_{\LL^2(\Dom)}^2 + \|u(t,\omega)\|_{\HBC}^2 \B)
\end{equation*}
for some deterministic constant $C_0$ independent of $\eps$ and $t$. We hence get that, for almost every $(t,\omega)\in[0,T]\times\Omega$,
\begin{equation}
  \b\|\Vel^\eps_{t,\omega}[u(t,\omega)]\b\|_{\HBC}
  \leq 
  C_0 \B( 1 + \left\| D_x\dot{\overline{y}}(t) \right\|_{\LL^2(\Dom)} + \left\|D_x\overline{y}(t) \right\|_{\LL^2(\Dom)} + \|u(t,\omega)\|_{\HBC} \B).
  \label{eq:apriori_pointwise}
\end{equation}
Since $C_0$ is deterministic and independent of $t$, we can integrate the above bound in $t \in [0,\tau]$ and over $\omega\in\Omega$:
\begin{equation}
\int_0^\tau\!\!\int_\Omega\b\| \Vel^\eps_{t,\omega}[u(t)]\b\|^2_{\HBC}\dPr(\omega)\dt \leq C_0\B( \tau+\|D_x\overline{y}\|^2_{\HH^1([0,\tau];\LL^2(\Dom))}+\|u\|^2_{\LL^2([0,\tau]\times\Omega;\HBC)} \B) \qquad \text{for any }\tau\in[0,T].
  \label{eq:apriori_integrated}
\end{equation}
It follows that, if $u\in\LL^2\b([0,T]\times\Omega;\HBC\b)$, then the mapping $(t,\omega)\mapsto\Vel^\eps_{t,\omega}[u(t,\omega)]$ is well--defined, and also lies in $\LL^2\b([0,T]\times\Omega;\HBC\b)$.

\subsection{Existence of solutions for $\eps>0$}
\label{sec:existence_eps}
  
Define now the operator $\mathcal{T}^\eps:\LL^2([0,\tau]\times\Omega;\HBC)\to\LL^2([0,\tau]\times\Omega;\HBC)$ for some $0<\tau \leq T$ (which will be fixed later) by
\begin{equation*}
  \forall t \in [0,\tau], \qquad \mathcal{T}^\eps[u](t,\omega) = u_0(\omega) +\int_0^t \Vel^\eps_{s,\omega}[u(s,\omega)]\ds,
\end{equation*}
where $u_0 \in \LL^2(\Omega;\HBC)$ is given. This operator is well--defined in view of the arguments leading up to~\eqref{eq:apriori_integrated}. Moreover, we infer from~\eqref{eq:Vel_Lipschitz} that, for any $u$ and $v$ in $\LL^2([0,\tau]\times\Omega;\HBC)$,
\begin{align*}
  \B\|\mathcal{T}^\eps[u]-\mathcal{T}^\eps[v] \B\|^2_{\LL^2([0,\tau]\times\Omega;\HBC)}
  &=
  \int_0^\tau\!\!\int_\Omega\bg\| \int_0^t \left( \Vel^\eps_{s,\omega}[u(s,\omega)]-\Vel^\eps_{s,\omega}[v(s,\omega)] \right) \ds \bg\|_{\HBC}^2\dPr(\omega)\dt
  \\
  &\leq \int_0^\tau\!\!\int_\Omega t \int_0^t \B\| \Vel^\eps_{s,\omega}[u(s,\omega)]-\Vel^\eps_{s,\omega}[v(s,\omega)] \B\|^2_{\HBC}\ds\dPr(\omega)\dt
  \\
  &\leq \frac{\overline{A}^2\tau^2}{4c^2} \, \|u-v\|^2_{\LL^2([0,\tau]\times\Omega;\HBC)}.
\end{align*}
The map $\mathcal{T}^\eps$ is thus a contraction mapping on $\LL^2([0,\tau]\times\Omega;\HBC)$ whenever $\tau<2c/\overline{A}$, and so has a unique fixed point by the Banach Fixed Point Theorem, which we denote $u^\eps \in \LL^2([0,\tau]\times\Omega;\HBC)$. \blue{Moreover, applying the Lemma proved in~III.11.16 of~\cite{DunfordSchwartzI},} it is straightforward to check that this fixed point \blue{has a weak time derivative $\dot{u}^\eps$ in $\LL^2([0,\tau]\times\Omega;\HBC)$} by applying~\eqref{eq:apriori_integrated}. Choosing $u_0 \equiv 0$, we see that this fixed point satisfies the time-evolution equation in~\eqref{eq:EvolutionProblem} for almost every $t\in[0,\tau]$. \blue{Moreover, applying Tonelli's theorem to deduce that $u^\eps:[0,T]\to\LL^2(\Omega;\HBC)$ is a measurable map when we choose a representative such that
\begin{equation*}
  u^\eps(t) = \int_0^t\dot{u}^\eps(s)\ds \in\LL^2(\Omega;\HBC),
\end{equation*}
we deduce that we may take $u^\eps\in\HH^1\b([0,T];\LL^2(\Omega;\HBC)\b)$.} Recalling that functions in $\HH^1\b([0,T];\LL^2(\Omega;\HBC)\b)$ have unique representatives in $\CC\b([0,T];\LL^2(\Omega;\HBC)\b)$, we note that this representative satisfies $u^\eps(0)=u_0=0$. We thus have built a solution to~\eqref{eq:EvolutionProblem} on $[0,\tau]\times\Omega$ for any $\tau<2c/\overline{A}$. Since the argument given above is independent of the initial condition $u_0$, we may apply the same argument iteratively to show that a solution to~\eqref{eq:EvolutionProblem} exists on $[0,T]\times\Omega$. The uniqueness of that solution is a consequence of~\eqref{eq:Vel_Lipschitz}. This therefore completes the proof of Theorem~\ref{th:eps_existence}.
  
\subsection{Boundedness in $\HH^1([0,T];\LL^2(\Omega;\HBC))$}
\label{sec:boundedness}

We now check that the mapping $u^\eps$ has properties sufficient for us to pursue our subsequent analysis. As noted above, since $u^\eps\in\HH^1([0,T];\LL^2(\Omega;\HBC))$, there is a well--defined representative of $u^\eps$ in $\CC([0,T];\LL^2(\Omega;\HBC))$, which satisfies $\dps u^\eps(\tau,\omega) = \int_0^\tau \dot{u}^\eps(t,\omega)\dt$ since $u^\eps(0,\omega)=0$. Since $\dot{u}^\eps(t,\omega)=\Vel^\eps_{t,\omega}[u^\eps(t,\omega)]$ for almost every $(t,\omega)\in[0,T]\times\Omega$, we infer from~\eqref{eq:apriori_pointwise} that
\begin{align*}
  \|u^\eps(\tau,\omega)\|_{\HBC}
  &=
  \bg\| \int_0^\tau\dot{u}^\eps(t,\omega)\dt\bg\|_{\HBC}
  \\
  &\leq \int_0^\tau\|\dot{u}^\eps(t,\omega)\|_{\HBC}\dt
  \\
  &\leq C_0 \int_0^\tau \Big( 1+\left\| D_x\dot{\overline{y}}(t) \right\|_{\LL^2(\Dom)} + \left\|D_x\overline{y}(t)\right\|_{\LL^2(\Dom)} + \|u^\eps(t,\omega) \|_{\HBC} \Big) \dt.
\end{align*}
Upon applying Gr\"onwall's inequality, we get, for any $\tau \in [0,T]$, that
\begin{equation}
  \label{eq:redite_pre}
  \|u^\eps(\tau,\omega)\|_{\HBC}\leq C_0\tau\mathrm{e}^{C_0\tau} + C_0\mathrm{e}^{C_0\tau}\int_0^\tau \b( \left\|D_x\dot{\overline{y}}(t)\right\|_{\LL^2(\Dom)}+\left\|D_x\overline{y}(t)\right\|_{\LL^2(\Dom)} \b) \dt\leq C_1(T),
\end{equation}
where $C_1$ is independent of $\tau$, $\eps$ and $\omega$ (but depends on $T$). Squaring the bound~\eqref{eq:apriori_pointwise} and integrating in \blue{$\omega$ and }$t$, we obtain
\begin{equation}
  \label{eq:redite}
  \blue{\int_0^T\!\!\int_\Omega\|\dot{u}^\eps(t,\omega)\|^2_{\HBC}\dPr(\omega)\dt \leq C_2,}
\end{equation}
where $C_2$ is independent of $\eps$. Collecting~\eqref{eq:redite_pre} and~\eqref{eq:redite}, we have thus established that the family of solutions $u^\eps$ is uniformly bounded in $\HH^1([0,T];\LL^2(\Omega;\HBC))$: there exists $C$ independent of $\eps$ such that
\begin{equation} \label{eq:bound_u_eps}
  \| u^\eps \|_{\HH^1([0,T];\LL^2(\Omega;\HBC))} \leq C.
\end{equation}

%% I (=Fred) remove old stuff commented out by Tom and move it to another file
%% \input{comment_out_tom3.tex}

\section{Obtaining a homogenized limit}
\label{sec:limit}

The following theorem now gives a precise statement of our main result, Theorem~\ref{th:main}, namely the identification of the homogenized limit of~\eqref{eq:EvolutionProblem} as $\eps\to0$. We use the notion of stochastic two--scale convergence introduced in Section~\ref{sec:stoc2scale}.

\begin{theorem}[Rigorous statement of Theorem~\ref{th:main}]%\label{th:mainRig}
Assume that $\overline{y}$ belongs to $\HH^1\b([0,T];\HH^1(\Dom)\b)$, and let $u^\eps\in\HH^1\b([0,T];\LL^2(\Omega;\HBC)\b)$ be the sequence of solutions to the evolution problem~\eqref{eq:EvolutionProblem} with the initial condition $u^\eps(0) = 0$. Then, as $\eps\to0$, $u^\eps$ two--scale converges to $u^\star$, and $D_xu^\eps$ two--scale converges to $D_x u^\star+\theta$, where $u^\star\in\HH^1([0,T];\HBC)$ and $\theta\in\HH^1\b([0,T];\LL^2(\Dom;\L2Pot(\Omega))\b)$ are the unique solutions to the system of inclusions
\begin{equation}
  -\div_x\bg[\int_\Omega \partial_\xi\psi_0\b(\omega,D_x\dot{\overline{y}}+D_x\dot{u}^\star+\dot{\theta}\b)+ D_\xi W_0\b(\omega,D_x\overline{y}+D_xu^\star+\theta\b)\dPr(\omega)\bg]\ni0,
  \tag{\ref{eq:homogenised_eqn}}
\end{equation}
\begin{equation}
  -\div_\omega \b[\partial_\xi\psi_0(\omega,D_x\dot{\overline{y}}+D_x\dot{u}^\star +\dot{\theta})+D_\xi W_0(\omega,D_x\overline{y}+D_xu^\star+\theta)\b]\ni 0,
  \tag{\ref{eq:corrector_eqn}}
\end{equation}
with the initial conditions $u^\star(0)=0$ and $\theta(0)=0$. 

The inclusions~\eqref{eq:homogenised_eqn} and~\eqref{eq:corrector_eqn} respectively hold in $\HBCPrime$ and $\LL^2\b(\Dom;(\L2Pot(\Omega))'\b)$ for almost every $t\in[0,T]$. Note that $u^\star$ is deterministic and that $\Ex[\theta]=0$.
\end{theorem}
\blue{As pointed out below Theorem~\ref{th:main}, Equations~\eqref{eq:homogenised_eqn} and~\eqref{eq:corrector_eqn} should be understood as follows: there exists some vector valued function $G^\star \in \LL^2\b( [0,T] \times \Omega \times \Dom \b)^d$, satisfying $G^\star(t,\omega,x) \in \partial_\xi\psi_0\left(\omega,D_x\dot{\overline{y}}+D_x\dot{u}^\star +\dot\theta\right)$ for almost any $(t,\omega,x) \in [0,T] \times \Omega \times \Dom$, such that~\eqref{eq:homogenised_eqn_bis} and~\eqref{eq:corrector_eqn_bis} hold.}

\medskip

We have stated in Theorem~\ref{th:eps_existence} the well-posedness of~\eqref{eq:EvolutionProblem} with the initial condition $u^\eps(0)=u_0 \equiv 0$, but the proof shows that the same result holds for any $u_0 \in \LL^2(\Omega;\HBC)$. Similarly, the proof we give for Theorem~\ref{th:main} can be adapted to show that the result holds also when $u^\eps(0)=u_0 \in \LL^2(\Omega;\HBC)$ for any $\eps$ (note that the initial condition is independent from $\eps$), with the initial conditions for the homogenized problem becoming $u^\star(0) = \Ex[u_0]$ and \red{$\theta(0) = 0$. The details of this adaptation are left to the reader.}

\medskip
  
The proof of Theorem~\ref{th:main} is given over the remainder of this section, and proceeds by compactness. The main idea is relatively standard: in view of the uniform \emph{a priori} bounds~\eqref{eq:bound_u_eps} on $u^\eps$, we can extract a two--scale convergent subsequence using Lemma~\ref{th:2scale_H1}. We next identify an equation satisfied by the limit. We eventually demonstrate that that homogenized equation has a unique solution. The whole sequence hence two--scale converges to the derived limit.

\begin{remark}
Theorem~\ref{th:main} implies that
\begin{equation}
  \label{eq:vendredi2}
\lim_{\eps \to 0} \int_0^T\!\!\int_\Dom \left| u^\star - \Ex(u^\eps) \right|^2 = 0.
\end{equation}
Indeed, we know that $u^\eps$ two--scale converges to $u^\star$, which is independent of $\omega$. Taking test functions in Definition~\ref{def:two-scale} that are independent of $\omega$, we obtain that, for any $\psi\in \LL^2\b([0,T];\CC\b(\overline{\Dom}\b)\b)$, we have
\begin{equation}
  \label{eq:vendredi}
  \lim_{\eps\to 0} \int_0^T\!\!\int_\Dom \Ex(u^\eps(t,\cdot,x)) \, \psi \left(t,x\right) \dx \dt
  =
  \int_0^T\!\!\int_\Dom u^\star(t,x) \, \psi(t,x) \dx \dt.
\end{equation}
In addition, we have shown in Section~\ref{sec:boundedness} that, almost surely, $u^\eps(\omega)$ is bounded in $\HH^1([0,T];\HBC)$ by a constant $C$ independent of $\eps$ and $\omega$ (see~\eqref{eq:redite_pre} and~\eqref{eq:redite}). We thus have
$$
\int_0^T\!\!\int_\Dom \B( \Ex(u^\eps) \B)^2 \leq \int_0^T\!\!\int_\Dom \Ex\b(|u^\eps|^2\b) \leq C 
$$
and likewise for $\Ex(D_x u^\eps)$. We hence have that $\overline{u}^\eps = \Ex(u^\eps)$ is bounded in $\HH^1([0,T];\HBC)$. There hence exists $\overline{u}^\star \in \HH^1([0,T];\HBC)$ such that, up to the extraction of a subsequence, $\overline{u}^\eps$ converges to $\overline{u}^\star$, weakly in $\HH^1([0,T];\HBC)$ and strongly in $\LL^2([0,T];\LL^2(\Dom))$. Collecting this result with~\eqref{eq:vendredi}, we get~\eqref{eq:vendredi2}. We furthermore obtain that $\Ex(u^\eps)$ weakly converges to $u^\star$ in $\HH^1([0,T];\HBC)$.
\end{remark}

\subsection{Convergence of subsequences}
\label{sec:conv_subseq}

The \emph{a priori} bound~\eqref{eq:bound_u_eps} on $u^\eps$ allows us to apply Lemma~\ref{th:2scale_H1}. There hence exist some $u^\star \in \HH^1([0,T];\HBC)$, some $\theta\in\HH^1\b([0,T];\LL^2\b(\Dom;\L2Pot(\Omega)\b)\b)$ and a subsequence such that, along that subsequence, $u^\eps$ and $\dot{u}^\eps$ respectively two--scale converge to $u^\star$ and $\dot{u}^\star$, and $D_x u^\eps$ and $D_x \dot{u}^\eps$ respectively two--scale converge to $D_xu^\star+\theta$ and $D_x\dot{u}^\star+\dot{\theta}$. Using these properties, we now demonstrate that $u^\star$ and $\theta$ satisfy a system of nonlinear evolutionary inclusions, namely~\eqref{eq:corrector_eqn} and~\eqref{eq:homogenised_eqn}. The proof falls in 5 steps.

\medskip

\noindent
\emph{Step 1. Reformulation of~\eqref{eq:EvolutionProblem}.}
We claim that the statement that~\eqref{eq:EvolutionProblem} holds up to a set of $\Leb^1$--negligible times $t\in[0,T]$ for $\Pr$--almost every $\omega$ is equivalent to the statement that
\begin{equation}
  0=\int_0^T\!\!\int_\Omega \Big( \Psi^\eps_{t,\omega}[\dot{u}^\eps(t,\omega)] + (\Psi^{\eps}_{t,\omega})^*\b[-\nabla\Phi^\eps_{t,\omega}[u^\eps(t,\omega)]\b]
  -\b\<-\nabla\Phi^\eps_{t,\omega}[u^\eps(t,\omega)],\dot{u}^\eps(t,\omega)\b\>_{\HBC} \Big) \dPr(\omega)\dt ,
  \label{eq:zero_energy}
\end{equation}
where $\b(\Psi^\eps_{t,\omega})^*$ denotes the Legendre--Fenchel transform of $\Psi^\eps_{t,\omega}$ with respect to the duality product $\<\cdot,\cdot\>_{\HBC}$. This equivalence is discussed in greater detail in~\cite{MielNotes,Roubicek}. We briefly recall here the main idea, which was also discussed in Section~\ref{sec:InitialExample} (see~\eqref{eq:LFEquivalence}). For any \blue{proper,} convex and lower semicontinuous function $F:X\to\R$ defined on a Banach space $X$, define the Legendre--Fenchel transform $F^*:X'\to\R\cup\{+\infty\}$ by
\begin{equation*}
  F^*(\sigma):=\sup_{\xi\in X}\b\{\<\sigma,\xi\>_X-F(\xi)\b\}.  
\end{equation*}
Using the fact that $F$ is convex, it may be deduced that $F^*$ is also convex, and additionally $F(\xi)+F^*(\sigma)\geq \<\sigma,\xi\>_X$ for any $\xi\in X$ and $\sigma\in X'$. Moreover, the statements
\begin{equation} \label{eq:convex_duality_inclusion}
(1)\quad F(\xi)+F^*(\sigma)=\<\sigma,\xi\>_X, \qquad (2) \quad \sigma\in \partial_\xi F(\xi) \quad \text{and} \qquad (3) \quad \xi\in\partial_\sigma F^*(\sigma)
\end{equation}
are equivalent.

If~\eqref{eq:EvolutionProblem} holds for almost every $(t,\omega)\in[0,T]\times\Omega$, then, using~\eqref{eq:convex_duality_inclusion} and integrating with respect to $t$ and $\omega$, we obtain~\eqref{eq:zero_energy}. Conversely, the integrand in~\eqref{eq:zero_energy} is always non-negative. The equation~\eqref{eq:zero_energy} thus implies that the integrand vanishes for almost every $(t,\omega)\in[0,T]\times\Omega$. The equivalence~\eqref{eq:convex_duality_inclusion} then implies that~\eqref{eq:EvolutionProblem} is satisfied for almost every $(t,\omega)\in[0,T]\times\Omega$. We have thus proved our claim.

\smallskip

In the sequel of the proof, we use the integral formulation~\eqref{eq:zero_energy} to pass to the limit $\eps \to 0$. 
  
\medskip

\noindent
\emph{Step 2. Passing to the limit in the first term of~\eqref{eq:zero_energy}.}
For any $\xi\in\CC^\infty_0\b([0,T];\CC^1_0(\Dom\cup\GNeu;\mathscr{D}^\infty(\Omega))\b)^d$, set $\dps \xi^\eps(t,\omega,x):= \xi\left(t,T\left(\frac{x}{\eps}\right)\omega,x\right)$. In view of the discussion below Definition~\ref{def:admissible}, the function $\xi$ is admissible (i.e. the function $\xi^\eps$ is measurable and square-integrable). We use Lemma~\ref{th:convex_approx} to introduce a measurable and $\CC^1$ approximation $\psi_\eta$ of $\psi$. Using property~\eqref{iii_sec5} of Lemma~\ref{th:convex_approx} and the fact that $\psi_\eta$ is convex and differentiable, we write
\begin{align}
  \psi\left(\omega,\frac{\cdot}{\eps},D_x\dot{\overline{y}}+D_x\dot{u}^\eps \right)
  &\geq
  \psi_\eta\left(\omega,\frac{\cdot}{\eps},D_x\dot{\overline{y}}+D_x\dot{u}^\eps \right)
  \nonumber
  \\
  &\geq
  \psi_\eta\left(\omega,\frac{\cdot}{\eps},D_x\dot{\overline{y}}+\xi^\eps \right)
  +
  D_\xi \psi_\eta\left(\omega,\frac{\cdot}{\eps},D_x\dot{\overline{y}} + \xi^\eps\right)\cdot [D_x\dot{u}^\eps-\xi^\eps]. \label{eq:maison2}
\end{align}
Consider the function $\alpha(t,\omega,x) = D_\xi \psi_{0,\eta}\left(\omega,D_x\dot{\overline{y}}(t)+ \xi(t,\omega,x)\right)$, where $\psi_{0,\eta}$ is defined in~\eqref{eq:StationaryReg}, and set
$$
\alpha^\eps(t,\omega,x)
:=
\alpha\left(t,T\left(\frac{x}{\eps}\right)\omega,x\right)
=
D_\xi \psi_\eta\left(\omega,\frac{x}{\eps},D_x\dot{\overline{y}} + \xi^\eps\right).
$$
Noting that $\psi_\eta$ is $\CC^1$ in its third argument, that $\overline{y}\in\HH^1([0,T];\HH^1(\Dom))$ and that $\xi$ and $\xi^\eps$ are measurable, we have that $\alpha$ and $\alpha^\eps$ are measurable.
%\fl{cette mesurabilite n'est pas evidente. Il est certain que, si on travaille avec $\psi$, ca ne marche pas, pas assez de regularite dans le troisieme argument. D'o\`u l'idee de travailler avec $\psi_\eta$. La fonction $D_\xi \psi_\eta$ est de Caratheodory, ie mesurable vs $\omega$ et continue vs $\xi$, cf. https://www.encyclopediaofmath.org/index.php/Carath\%C3\%A9odory\_conditions, et du coup $\alpha$ est bien mesurable et idem pour $\alpha^\eps$}
Furthermore, using~\eqref{eq:bound_d_psi_eta} and the regularity of $\xi$, $\xi^\eps$ and $\overline{y}$, we see that $\alpha$ and $\alpha^\eps$ belong to $\LL^2([0,T]\times\Omega\times\Dom)$. The function $\alpha$ is hence an admissible test function. Likewise, introduce $\beta(t,\omega,x) = \psi_{0,\eta}\left(\omega,D_x\dot{\overline{y}}(t)+ \xi(t,\omega,x)\right)$, where we again recall that $\psi_{0,\eta}$ is defined in~\eqref{eq:StationaryReg}, and
$$
\beta^\eps(t,\omega,x)
:=
\beta\left(t,T\left(\frac{x}{\eps}\right)\omega,x\right)
=
\psi_\eta\left(\omega,\frac{x}{\eps},D_x\dot{\overline{y}} + \xi^\eps\right).
$$
Using the same arguments as above, we have that $\beta$ and $\beta^\eps$ are measurable. We next infer from property~\eqref{iii_sec5} of Lemma~\ref{th:convex_approx}, Assumption~$(\psi4)$ and the regularity of $\xi$, $\xi^\eps$ and $\overline{y}$ that $\beta$ and $\beta^\eps$ belong to $\LL^{\red{1}}([0,T]\times\Omega\times\Dom)$.

Integrating~\eqref{eq:maison2} and taking the liminf, we have that
\begin{multline}
  \liminf_{\eps\to0} \int_0^T\!\!\int_\Omega \Psi^\eps_{t,\omega}[\dot{u}^\eps_\omega(t)] \dPr(\omega)\dt
  \geq
  \liminf_{\eps\to0}\int_0^T\!\!\int_\Omega\int_\Dom \beta^\eps(t,\omega,x) \dx \dPr(\omega)\dt
  \\
  +
  \liminf_{\eps\to0} \int_0^T\!\!\int_\Omega \int_\Dom \alpha^\eps \cdot [D_x\dot{u}^\eps-\xi^\eps].
  \label{eq:maison3}
\end{multline}
We are now in position to pass to the limit $\eps \to 0$ in the right-hand side of~\eqref{eq:maison3}. Since the ergodic dynamical system $T$ preserves the measure, we have
$$
\int_0^T\!\!\int_\Omega \int_\Dom \! \beta^\eps\left(t,\omega,x\right)\dx\dPr(\omega)\dt
=
\int_0^T\!\!\int_\Dom \left[ \int_\Omega \! \beta\left(t,T\left(\frac{x}{\eps}\right)\omega,x\right)\dPr(\omega) \right] \! \dx\dt
=
\int_0^T\!\!\int_\Dom \left[ \int_\Omega \! \beta\left(t,\omega,x\right) \! \dPr(\omega) \right] \! \dx\dt
$$
and likewise for $\alpha^\eps \cdot \xi^\eps$. Using that $\alpha$ is admissible and passing to the two--scale limit in the remainder term of the right-hand side of~\eqref{eq:maison3}, we get
\begin{multline} \label{eq:maison4}
  \liminf_{\eps\to0}\int_0^T\!\!\int_\Omega\Psi^\eps_{t,\omega}[\dot{u}^\eps_\omega(t)]\dPr(\omega)\dt\\
  \geq 
  \int_0^T\!\!\int_\Omega\int_\Dom \left( \psi_{0,\eta}\left(\omega,D_x\dot{\overline{y}} +\xi\right)
  +D_\xi\psi_{0,\eta}\left(\omega,D_x\dot{\overline{y}}+\xi\right)\cdot \b[D_x\dot{u}^\star+\dot{\theta}-\xi\b] \right)\dx\dPr(\omega)\dt.
\end{multline}
Next, by density of $\CC^\infty_0([0,T];\CC^1_0(\Dom\cup\GNeu;\mathscr{D}^\infty(\Omega)))$ in $\LL^2([0,T]\times\Omega\times\Dom)$, we let $\xi\to D_x\dot{u}^\star+\dot{\theta}$ in $\LL^2([0,T]\times\Omega\times\Dom)^d$. Using the bound~\eqref{eq:bound_d_psi_eta} on $D_\xi\psi_{0,\eta}$, we deduce that
\begin{align*}
  \liminf_{\eps\to0}\int_0^T\!\!\int_\Omega\Psi^\eps_{t,\omega}[\dot{u}^\eps_\omega(t)]\dPr(\omega)\dt
  &\geq \int_0^T\!\!\int_\Omega\int_\Dom \psi_{0,\eta}\left(\omega,D_x\dot{\overline{y}}+D_x\dot{u}^\star+\dot{\theta}\right)\dx\dPr\dt.
\end{align*}
\blue{Now, using property~\eqref{iii_sec5} of $\psi_\eta$ proved in Lemma~\ref{th:convex_approx}, we deduce that
\begin{align*}
  \liminf_{\eps\to0}\int_0^T\!\!\int_\Omega\Psi^\eps_{t,\omega}[\dot{u}^\eps_\omega(t)]\dPr(\omega)\dt
  &\geq \int_0^T\!\!\int_\Omega\int_\Dom \psi_0\left(\omega,D_x\dot{\overline{y}}+D_x\dot{u}^\star+\dot{\theta} \right)\dx\dPr\dt\\
  &\qquad\qquad- 8\eta C^2\int_0^T\!\!\int_\Omega\int_\Dom \left( 1+|D_x\dot{\overline{y}}+D_x\dot{u}^\star+\dot{\theta}|^2 \right) \dx\dPr\dt.
\end{align*}
Taking the limit $\eta\to0$, the latter term on the right--hand side vanishes, and} we obtain
\begin{equation}
  \liminf_{\eps\to0}\int_0^T\!\!\int_\Omega\Psi^\eps_{t,\omega}[\dot{u}^\eps_\omega(t)]\dPr(\omega)\dt
  \geq
  \int_0^T\!\!\int_\Omega\int_\Dom\psi_0\left(\omega,D_x\dot{\overline{y}}+D_x\dot{u}^\star+\dot{\theta}\right)\dx\dPr\dt.
  \label{eq:liminf_term1}
\end{equation}

%% I (=Fred) remove old stuff commented out by Tom and move it to another file
%% \input{comment_out_tom4.tex}

\noindent
\emph{Step 3. Passing to the limit in the second term of~\eqref{eq:zero_energy}.}
\blue{Consider the regularisation $\psi_\eta$ of $\psi$ introduced in Lemma~\ref{th:convex_approx}. In view of~\eqref{eq:utile}, we have, for any $v \in \HBC$, that
\begin{multline*}
\Psi^{\eps,\eta}_{t,\omega}[v]
:=
\int_\Dom \psi_\eta\left(\omega,\frac{x}{\eps},D_x\dot{\overline{y}}(t)+D_xv\right)\dx
\geq
m_\eta \int_\Dom \psi\left(\omega,\frac{x}{\eps},D_x\dot{\overline{y}}(t)+D_xv\right)\dx - 8 \eta C^2 \Leb^d(\Dom)
\\
=
m_\eta \Psi^\eps_{t,\omega}[v] - 8 \eta C^2 \Leb^d(\Dom),
\end{multline*}
where $\Leb^d(\Dom)$ denotes the volume of the domain $\Dom$. We thus deduce that, for any $f \in \HBCPrime$, we have
\begin{equation} \label{eq:samedi}
(\Psi^\eps_{t,\omega})^*[f] \geq \frac{1}{m_\eta} (\Psi^{\eps,\eta}_{t,\omega})^*[m_\eta \, f] - \eta \frac{8 C^2 \Leb^d(\Dom)}{m_\eta}.
\end{equation}
For any $u \in \HBC$, we have
$$  
\Phi^\eps_{t,\omega}[u]
=
\frac{1}{2} \int_\Dom D_x(\overline{y}(t)+u)\cdot A^\eps D_x(\overline{y}(t)+u)\dx,
$$
where we have denoted $\dps A^\eps = A\left(\omega,\frac{x}{\eps}\right)$. For any $v \in \HBC$, we hence write
\begin{equation} \label{eq:nabla_Phi_t}
\<\nabla \Phi^\eps_{t,\omega}[u^\eps_\omega(t)],v\>_{\HH^1(\Dom)} = \int_\Dom A^\eps D_x(\overline{y}(t)+u^\eps_\omega(t))\cdot D_xv\dx
\end{equation}
and thus $\dps -\nabla\Phi^\eps_{t,\omega}[u^\eps(t)] = \div_x \left( A^\eps D_x(\overline{y}(t)+u^\eps_\omega(t)) \right)$. Combining~\eqref{eq:samedi} and Lemma~\ref{th:dual_char} therefore yield that
\begin{multline} \label{eq:maison}
  \eta \frac{8 C^2 T \Leb^d(\Dom)}{m_\eta} + \int_0^T\!\!\int_\Omega(\Psi^\eps_{t,\omega})^*\left[-\nabla\Phi^\eps_{t,\omega}[u^\eps_\omega(t)]\right]\dPr(\omega)\dt
  \\
  \geq \frac{1}{m_\eta} \int_0^T\!\!\int_\Omega\int_\Dom (\psi_\eta)^*\left(\omega,\frac{x}{\eps},\sigma^\eps_\eta-m_\eta A^\eps D_x(\overline{y}+u^\eps_\omega)\right)-\Big( \sigma^\eps_\eta - m_\eta A^\eps D_x\left(\overline{y}+u^\eps_\omega\right)\Big) \cdot D_x\dot{\overline{y}} \, \dx\dPr(\omega)\dt
\end{multline}
for some $\sigma^\eps_\eta \in \LL^2([0,T]\times\Omega\times\Dom)^d$ satisfying $\sigma^\eps_\eta(t,\omega) \in \L2Pot(\Dom)^\perp$ for almost every $(t,\omega)$. By combining the bound~\eqref{eq:bound_sigma} on $\sigma^\eps_\eta$ with the \emph{a priori} bounds on $u^\eps$ obtained in Section~\ref{sec:boundedness}, we note that $\sigma^\eps_\eta$ is uniformly bounded in $\LL^2([0,T]\times\Omega\times\Dom)^d$. Considering $\eta$ of the form $\eta = 1/n$ for $n \in \N^\star$, applying Lemma~\ref{th:2scale_divfree_L2} for any fixed $\eta$ of that form and using a diagonalization argument, we infer that there exists a subsequence $\{ \sigma^{\eps'}_\eta \}$ (with $\eps'$ independent of $\eta = 1/n$) along which $\sigma^{\eps'}_\eta$ two--scale converges to $\sigma^\star_\eta+\zeta_\eta$ when $\eps' \to 0$, where $\sigma^\star_\eta \in \LL^2\left([0,T];\L2Pot(\Dom)^\perp\right)$ and $\zeta_\eta \in \LL^2\left([0,T];\LL^2(\Dom;\L2Pot(\Omega)^\perp)\right)$. For simplicity of notation, we again index this subsequence by $\eps$.

\smallskip

Using the convexity of $(\psi_\eta)^*$ (note property~\eqref{vi_sec5} of Lemma~\ref{th:convex_approx}), we deduce from~\eqref{eq:maison} that
\begin{align*}
  & \eta \frac{8 C^2 T \Leb^d(\Dom)}{m_\eta} + \int_0^T\!\!\int_\Omega(\Psi^\eps_{t,\omega})^*\left[-\nabla\Phi^\eps_{t,\omega}[u^\eps_\omega(t)]\right]\dPr(\omega)\dt
  \\
  & \geq \frac{1}{m_\eta} \int_0^T\!\!\int_\Omega\int_\Dom (\psi_\eta)^*\left(\omega,\frac{x}{\eps},\xi^\eps - m_\eta A^\eps D_x\overline{y} \right) - \Big( \sigma^\eps_\eta - m_\eta A^\eps D_x\left(\overline{y}+u^\eps_\omega\right) \Big) \cdot D_x\dot{\overline{y}} \, \dx\dPr(\omega)\dt
  \\
  & \qquad +
  \frac{1}{m_\eta} \int_0^T\!\!\int_\Omega\int_\Dom D_\xi (\psi_\eta)^*\left(\omega,\frac{x}{\eps},\xi^\eps - m_\eta A^\eps D_x\overline{y} \right) \cdot [\sigma^\eps - \xi^\eps - m_\eta A^\eps D_xu^\eps_\omega] \, \dx\dPr(\omega)\dt
\end{align*}
for any $\xi\in\CC^\infty_0\b([0,T];\CC^1_0(\Dom\cup\GNeu;\mathscr{D}^\infty(\Omega))\b)^d$ and where $\dps \xi^\eps(t,\omega,x)= \xi\left(t,T\left(\frac{x}{\eps}\right)\omega,x\right)$. We are now in position to pass to the two--scale limit, in the same manner as used to arrive to~\eqref{eq:maison4}. We first note that $A_0(\omega) \, D_x\dot{\overline{y}}$ is an admissible test function (see e.g. the remarks before Proposition~3.1 of~\cite{BMW94}). Using the fact that both $D_\xi (\psi_{0,\eta})^*\Big(\omega,\xi- m_\eta A_0(\omega) D_x\overline{y} \Big)$ and $A_0(\omega) \, D_\xi (\psi_{0,\eta})^*\Big(\omega,\xi- m_\eta A_0(\omega) D_x\overline{y} \Big)$ are also admissible test functions (recall indeed that $(\psi_{0,\eta})^*$ is $\CC^1$ in its third argument and satisfies~\eqref{eq:bound_d_psi_eta_star}), we obtain that
\begin{align*}
  & \eta \frac{8 C^2 T \Leb^d(\Dom)}{m_\eta} + \liminf_{\eps\to0}\int_0^T\!\!\int_\Omega(\Psi^\eps_{t,\omega})^*\left[-\nabla \Phi^\eps_{t,\omega}[u^\eps_\omega(t)]\right]\dPr(\omega)\dt
  \\
  & \geq \frac{1}{m_\eta} \int_0^T\!\!\int_\Dom \int_\Omega (\psi_{0,\eta})^*\Big(\omega,\xi - m_\eta A_0(\omega) D_x\overline{y} \Big) - \Big( \sigma^\star_\eta+\zeta_\eta - m_\eta A_0(\omega)\b(D_x\overline{y}+D_xu^\star+\theta\b) \Big) \cdot D_x\dot{\overline{y}} \, \dx\dPr(\omega)\dt
  \\
  & \quad +
  \frac{1}{m_\eta} \int_0^T\!\!\int_\Omega\int_\Dom D_\xi (\psi_{0,\eta})^*\Big(\omega,\xi- m_\eta A_0(\omega) D_x\overline{y} \Big) \cdot \left[\sigma^\star_\eta + \zeta_\eta - \xi - m_\eta A_0(\omega)\b(D_xu^\star+\theta\b) \right] \dx\dPr(\omega)\dt.
\end{align*}
Letting $\xi \to \sigma^\star_\eta + \zeta_\eta - m_\eta A_0(\omega)\b(D_xu^\star+\theta\b)$ and using the bound~\eqref{eq:bound_d_psi_eta_star} on $D_\xi(\psi_{0,\eta})^*$, we infer that
\begin{align}
  & \eta \frac{8 C^2 T \Leb^d(\Dom)}{m_\eta} + \liminf_{\eps\to0}\int_0^T\!\!\int_\Omega(\Psi^\eps_{t,\omega})^*\left[-\nabla \Phi^\eps_{t,\omega}[u^\eps_\omega(t)]\right]\dPr(\omega)\dt
  \nonumber
  \\
  & \geq \frac{1}{m_\eta} \int_0^T\!\!\int_\Dom \int_\Omega (\psi_{0,\eta})^*\Big(\omega,\sigma^\star_\eta+\zeta_\eta-m_\eta A_0(\omega)\b[D_x\overline{y}+D_xu^\star+\theta\b] \Big) \, \dx\dPr(\omega)\dt
  \nonumber
  \\
  & \qquad - \frac{1}{m_\eta} \int_0^T\!\!\int_\Dom \int_\Omega \Big( \sigma^\star_\eta+\zeta_\eta-m_\eta A_0(\omega)\b(D_x\overline{y}+D_xu^\star+\theta\b) \Big) \cdot D_x\dot{\overline{y}} \, \dx\dPr(\omega)\dt
  \nonumber
  \\
  & \geq \frac{1}{m_\eta} \int_0^T\!\!\int_\Dom \int_\Omega \psi_0^*\Big(\omega,\sigma^\star_\eta+\zeta_\eta-m_\eta A_0(\omega)\b[D_x\overline{y}+D_xu^\star+\theta\b] \Big) \, \dx\dPr(\omega)\dt
  \nonumber
  \\
  & \qquad - \frac{1}{m_\eta} \int_0^T\!\!\int_\Dom \int_\Omega \Big( \sigma^\star_\eta+\zeta_\eta-m_\eta A_0(\omega)\b(D_x\overline{y}+D_xu^\star+\theta\b) \Big) \cdot D_x\dot{\overline{y}} \, \dx\dPr(\omega)\dt,
  \label{eq:samedi4}
\end{align}
where we have used the bound~\eqref{eq:borne_psi_eta_star} in the last inequality. We are now in position to pass to the limit $\eta = 1/n \to 0$. To that aim, we first establish some bounds on $\sigma^\star_\eta$ and $\zeta_\eta$. Using Proposition~3.5(c) of~\cite{BMW94}, we see that
\begin{equation} \label{eq:samedi2}
\| \sigma^\star_\eta + \zeta_\eta \|_{\LL^2([0,T]\times\Omega\times\Dom)} \leq \liminf_{\eps \to 0} \| \sigma^\eps_\eta \|_{\LL^2([0,T]\times\Omega\times\Dom)} \leq C
\end{equation}
for some $C$ independent of $\eta$. Since $\sigma^\star_\eta$ is deterministic and the expectation of $\zeta_\eta$ vanishes, we can write that
\begin{equation} \label{eq:samedi3}
\| \sigma^\star_\eta \|_{\LL^2([0,T]\times\Dom)} = \left\| \Ex[\sigma^\star_\eta + \zeta_\eta] \right\|_{\LL^2([0,T]\times\Dom)} \leq \| \sigma^\star_\eta + \zeta_\eta \|_{\LL^2([0,T]\times\Omega\times\Dom)}.
\end{equation}
Collecting~\eqref{eq:samedi2} and~\eqref{eq:samedi3}, we see that the sequence $\{ \sigma^\star_\eta \}$ (respectively $\{ \zeta_\eta \}$) is bounded, and thus weakly converges (up to the extraction of a subsequence) to some $\sigma^\star \in \LL^2\left([0,T];\L2Pot(\Dom)^\perp\right)$ (respectively $\zeta \in \LL^2\left([0,T]\times\Dom;\L2Pot(\Omega)^\perp\right)$) in $\LL^2([0,T]\times\Dom)^d$ (respectively $\LL^2([0,T]\times\Omega\times\Dom)^d$).

Using these weak limits, we deduce, by letting $\eta \to 0$ in~\eqref{eq:samedi4}, that
\begin{multline}
  \liminf_{\eps\to0}\int_0^T\!\!\int_\Omega(\Psi^\eps_{t,\omega})^*\b[-\nabla \Phi^\eps_{t,\omega}[u^\eps_\omega(t)]\b]\dPr(\omega)\dt
  \geq \int_0^T\!\!\int_\Dom \int_\Omega \psi^*_0\B(\omega,\sigma^\star+\zeta-A_0(\omega)\b[D_x\overline{y}+D_xu^\star+\theta\b]\B)
  \\
  -\B(\sigma^\star+\zeta-A_0(\omega)\b(D_x\overline{y}+D_xu^\star+\theta\b)\B)\cdot D_x\dot{\overline{y}} \, \dx\dPr(\omega)\dt.
  \label{eq:liminf_term2}
\end{multline}
To pass to the limit in the first term of the right-hand side of~\eqref{eq:samedi4}, we have used the fact that $\psi_0^*$ is convex and lower semicontinuous in its third variable (see Theorem~2.43(i) of~\cite{dacorogna}), the fact that $\psi_0^*$ is bounded from below by $-C$ (a direct consequence of Assumption~$(\psi4)$) and Theorem~3.20 of~\cite{dacorogna}.} \medskip

\noindent
\emph{Step 4. Passing to the limit in the third term of~\eqref{eq:zero_energy}.}
We note that, using~\eqref{eq:nabla_Phi_t} and the chain rule, we have
\begin{align}
  & \int_0^T\!\!\int_\Omega \b\<\nabla\Phi^\eps_{t,\omega}[u^\eps_\omega(t)],\dot{u}^\eps_\omega(t)\b\>_{\HBC}
  \nonumber
  \\
  &=
  \int_0^T\!\!\int_\Omega\int_\Dom A\left(\omega,\frac{x}{\eps}\right) \left(D_x\overline{y}+D_xu^\eps_\omega\right)\cdot\left(D_x\dot{\overline{y}}+D_x\dot{u}^\eps_\omega\right)\dx\dPr(\omega)\dt
  \nonumber
  \\
  & \qquad - \int_0^T\!\!\int_\Omega\int_\Dom A\left(\omega,\frac{x}{\eps}\right) \left(D_x\overline{y}+D_xu^\eps_\omega\right)\cdot D_x\dot{\overline{y}} \, \dx\dPr(\omega)\dt
  \nonumber
  \\
  &=\int_0^T\!\!\int_\Omega\frac{\mathrm{d}}{\dt}\Phi^\eps_{t,\omega}[u^\eps_\omega]\dPr(\omega)\dx
  - \int_0^T\!\!\int_\Omega\int_\Dom A\left(\omega,\frac{x}{\eps}\right) \left(D_x\overline{y}+D_xu^\eps_\omega\right)\cdot D_x\dot{\overline{y}} \, \dx\dPr(\omega)\dt
  \nonumber
  \\
  &=
  \int_\Omega \Big( \Phi^\eps_{T,\omega}[u^\eps_\omega(T)]-\Phi^\eps_{0,\omega}[u^\eps_\omega(0)] \Big) \dPr(\omega)
  - \int_0^T\!\!\int_\Omega\int_\Dom A\left(\omega,\frac{x}{\eps}\right) \left(D_x\overline{y}+D_xu^\eps_\omega\right)\cdot D_x\dot{\overline{y}} \, \dx\dPr(\omega)\dt.
  \label{eq:liminf_term6}
\end{align}
Since $\Phi^\eps_{T,\omega}$ is convex, we can again use the arguments of Steps 2 and 3 \red{(recalling that the sequence selected at the beginning of the proof by applying Lemma~\ref{th:2scale_H1} also two--scale converges at initial and final times, in the sense described in~\cite{BMW94}). Since the function $A_0 \, (D_x\overline{y}(T) + \xi)$ (for any $\xi \in \CC^1_0\b(\Dom\cup\GNeu;\mathscr{D}^\infty(\Omega)\b)^d$) is admissible}, we obtain
%\fl{je ne redis pas pourquoi admissibilite ok; c'est les memes arguments qu'avant: $A_0$ l'est, cf. les remarks page 30 de~\cite{BMW94}, et $(D_x\overline{y}(T) + \xi)$ l'est aussi car tres regulier}
\begin{align}
  &\liminf_{\eps\to 0} \int_\Omega\Phi^\eps_{T,\omega}[u^\eps_\omega(T)]\dPr(\omega)
  \nonumber
  \\
  &\geq
  \frac{1}{2} \int_\Omega \int_\Dom A_0(\omega)\B(D_x\overline{y}(T)+D_xu^\star(T)+\theta(T)\B) \cdot\B(D_x\overline{y}(T)+D_xu^\star(T)+\theta(T)\B)\dPr(\omega)\dx
  \nonumber
  \\
  &=
  \int_\Omega \int_\Dom W_0\B(\omega,D_x\overline{y}(T)+D_xu^\star(T)+\theta(T)\B)\dPr(\omega)\dx.
  \label{eq:liminf_term3}
\end{align}
Furthermore, since $A_0 \, D_x\dot{\overline{y}}$ is admissible, we see that
\begin{multline}
\lim_{\eps\to0} \int_0^T\!\!\int_\Omega\int_\Dom A\left(\omega,\frac{x}{\eps}\right) \left(D_x\overline{y}+D_xu^\eps_\omega\right)\cdot D_x\dot{\overline{y}} \dx\dPr(\omega)\dt
\\=
\int_0^T\!\!\int_\Omega \int_\Dom A_0(\omega) \left(D_x\overline{y}+D_xu^\star+\theta\right)\cdot D_x\dot{\overline{y}} \, \dx\dPr(\omega)\dt.
\label{eq:liminf_term5}
\end{multline}
Noting that $u^\eps_\omega(0)=0$ and next using that $\theta(0) = D_xu^\star(0) = 0$ (see Lemma~\ref{th:2scale_H1}), we get that
\begin{align}
\int_\Omega\Phi^\eps_{0,\omega}[u^\eps_\omega(0)]\dPr(\omega)
&=
\frac{1}{2} \int_\Omega \int_\Dom D_x\overline{y}(0)\cdot A\left(\omega,\frac{x}{\eps}\right) \, D_x\overline{y}(0)\dx\dPr(\omega)
\nonumber
\\
&=
\frac{1}{2} \int_\Omega \int_\Dom D_x\overline{y}(0)\cdot A_0(\omega) \, D_x\overline{y}(0)\dx\dPr(\omega)
\nonumber
\\
&=
\int_\Omega \int_\Dom W_0\B(\omega,D_x\overline{y}(0)+D_xu^\star(0)+\theta(0)\B)\dPr(\omega)\dx.
\label{eq:liminf_term4}
\end{align}
Inserting~\eqref{eq:liminf_term3}, \eqref{eq:liminf_term5} and~\eqref{eq:liminf_term4} in~\eqref{eq:liminf_term6}, we obtain that
\begin{align} 
  &\liminf_{\eps\to 0} \int_0^T\!\!\int_\Omega \b\<\nabla\Phi^\eps_{t,\omega}[u^\eps_\omega(t)],\dot{u}^\eps_\omega(t)\b\>_{\HBC}
  \nonumber
  \\
  &\geq
  \int_\Omega \int_\Dom \left( W_0\B(\omega,D_x\overline{y}(T)+D_xu^\star(T)+\theta(T)\B) - W_0\B(\omega,D_x\overline{y}(0)+D_xu^\star(0)+\theta(0)\B)\right) \dPr(\omega)\dx
  \nonumber
  \\
  &\qquad
  - \int_0^T\!\!\int_\Omega \int_\Dom A_0(\omega) \left(D_x\overline{y}+D_xu^\star+\theta\right)\cdot D_x\dot{\overline{y}} \dx\dPr(\omega)\dt
  \nonumber
  \\
  &=
  \int_0^T\!\!\int_\Omega \int_\Dom A_0(\omega) \left(D_x\overline{y}+D_xu^\star+\theta\right)\cdot \left( D_x\dot{u}^\star+\dot{\theta} \right) \dx\dPr(\omega)\dt.
  \label{eq:liminf_term7}
\end{align}

\medskip

\noindent
\emph{Step 5. Conclusion.}
Collecting~\eqref{eq:zero_energy} with the estimates~\eqref{eq:liminf_term1}, \eqref{eq:liminf_term2} and~\eqref{eq:liminf_term7}, we deduce that
\begin{align}
  0 &=
  \liminf_{\eps\to0}\int_0^T\!\!\int_\Omega \left( \Psi^\eps_{t,\omega}[\dot{u}^\eps_\omega(t)]+(\Psi^\eps_{t,\omega})^*\b[-\nabla\Phi^\eps_{t,\omega}[u^\eps_\omega(t)]\b]+\b\<\nabla\Phi^\eps_{t,\omega}[u^\eps_\omega(t)],\dot{u}^\eps_\omega(t)\b\>_{\HBC} \right)\dPr(\omega)\dt
  \nonumber
  \\
  & \geq
  \int_0^T\!\!\int_\Dom\int_\Omega \left\{ \psi_0\B(\omega,D_x\dot{\overline{y}} +D_x\dot{u}^\star +\dot{\theta}\B)+\psi^*_0\B(\omega,\sigma^\star+\zeta-A_0(\omega) [D_x\overline{y}+D_xu^\star+\theta]\B) \right.
  \nonumber
  \\
  & \qquad \qquad \qquad -\B(\sigma^\star+\zeta-A_0(\omega)\b(D_x\overline{y}+D_xu^\star +\theta\b)\B)\cdot D_x\dot{\overline{y}}
  \nonumber
  \\
  & \qquad \qquad \qquad \left. + A_0(\omega)\b(D_x\overline{y}+D_xu^\star+\theta\b) \cdot \left( D_x\dot{u}^\star+\dot{\theta} \right) \right\} \dPr(\omega)\dx\dt.
  \label{eq:Psi*_liminf}
\end{align}
Since $\sigma^\star\in\LL^2\left([0,T];\L2Pot(\Dom)^\perp\right)$, we have $\dps \int_\Dom \sigma^\star \cdot D_x\dot{u}^\star \dx = 0$. Since $\theta\in\HH^1\left([0,T];\LL^2(\Dom;\L2Pot(\Omega))\right)$ and $\zeta\in\LL^2\left([0,T];\LL^2(\Dom;\L2Pot(\Omega)^\perp)\right)$, we also see that $\dps \int_\Omega \zeta \cdot \blue{\dot{\theta}} \dPr(\omega) = 0$. \blue{Moreover, it is a consequence of Lemma~\ref{th:2scale_divfree_L2}, and respectively, Lemma~\ref{th:2scale_H1} and Remark~\ref{rem:evident2} together, that the expectation of $\zeta$ and the expectation of $\dot{\theta}$ vanishes.} It hence follows that
\begin{equation*}
  \int_0^T\!\!\int_\Omega\int_\Dom\b(\sigma^\star+\zeta\b)\cdot\b(D_x\dot{u}^\star+\dot{\theta}\b)\dx\dPr(\omega)\dt=0.
\end{equation*}
We can therefore \red{subtract this quantity from} the right--hand side of~\eqref{eq:Psi*_liminf} to obtain
\begin{multline*}
  0 \geq \int_0^T\!\!\int_\Omega \int_\Dom \left\{ \psi_0\left(\omega,D_x\dot{\overline{y}} + D_x\dot{u}^\star + \dot{\theta} \right) + \psi^*_0\B(\omega,\sigma^\star+\zeta-A_0(\omega) [D_x\overline{y}+D_xu^\star+\theta]\B) \right.
  \\
  \left. -\B(\sigma^\star+\zeta-A_0(\omega)\b(D_x\overline{y} + D_xu^\star + \theta\b)\B)\cdot \left(D_x\dot{\overline{y}}+D_x\dot{u}^\star+\dot{\theta}\right) \right\} \dx\dPr(\omega)\dt.
\end{multline*}
Since $\psi_0(\omega,\xi)+\psi^*_0(\omega,\sigma)\geq\sigma\cdot \xi$ for any two vectors $\xi$ and $\sigma\in\R^d$, we see that the integrand is non-negative for almost every $(t,\omega,x)\in[0,T]\times\Omega\times\Dom$, which therefore implies that
\begin{multline}
  0=\psi_0\left(\omega,D_x\dot{\overline{y}}+D_x\dot{u}^\star+\dot{\theta}\right) +\psi^*_0\B(\omega,\sigma^\star+\zeta-A_0(\omega) [D_x\overline{y}+D_xu^\star+\theta]\B)
  \\
  -\B(\sigma^\star+\zeta-A_0(\omega)\b(D_x\overline{y}+D_xu^\star+\theta\b) \B) \cdot \left(D_x\dot{\overline{y}}+D_x\dot{u}^\star+\dot{\theta}\right).
  \label{eq:limit_energy_equality}
\end{multline}
Using~\eqref{eq:convex_duality_inclusion}, we deduce from the above relation that
$$
\sigma^\star+\zeta-A_0(\omega)\b(D_x\overline{y}+D_xu^\star+\theta\b) \in \partial_\xi \psi_0\left(\omega,D_x\dot{\overline{y}}+D_x\dot{u}^\star+\dot{\theta}\right),
$$
which we can also write as
\begin{equation} \label{eq:matin}
\sigma^\star+\zeta \in \partial_\xi \psi_0\left(\omega,D_x\dot{\overline{y}}+D_x\dot{u}^\star+\dot{\theta}\right) + D_\xi W_0\left(\omega,D_x\overline{y}+D_xu^\star+\theta\right).
\end{equation}
Recall now (see Step 3) that $\sigma^\star\in\LL^2\left([0,T];\L2Pot(\Dom)^\perp\right)$ and $\zeta\in\LL^2\left([0,T];\LL^2(\Dom;\L2Pot(\Omega)^\perp)\right)$ and (see beginning of Step 5) that the expectation of $\zeta$ vanishes. Taking the expectation and next the divergence in $x$ in equation~\eqref{eq:matin}, we obtain~\eqref{eq:homogenised_eqn}. Taking the divergence in $\omega$ in~\eqref{eq:matin}, we obtain~\eqref{eq:corrector_eqn}.

\medskip

We hence have proved our claim stated at the beginning of Section~\ref{sec:conv_subseq} that $u^\star$ and $\theta$ satisfy the system of nonlinear evolutionary inclusions~\eqref{eq:corrector_eqn} and~\eqref{eq:homogenised_eqn}, along with the initial and boundary conditions. \blue{Setting $G^\star(t,\omega,x) = \sigma^\star + \zeta - D_\xi W_0\left(\omega,D_x\overline{y}+D_xu^\star+\theta\right)$, we deduce from~\eqref{eq:matin} that~\eqref{eq:corrector_eqn_bis} and~\eqref{eq:homogenised_eqn_bis} hold.}
 
\subsection{Existence and uniqueness for the limiting evolution}
\label{sec:exist-uniq-limit}

We now define an additional function space and the notation necessary to study~\eqref{eq:corrector_eqn} and~\eqref{eq:homogenised_eqn}. Consider the product space $\H:=\HBC\times\LL^2(\Dom;\L2Pot(\Omega))$. When endowed with the inner product
$$
\forall (u^\star,\theta) \in \H, \quad \forall (v^\star,\xi)\in\H, \qquad 
\b((u^\star,\theta),(v^\star,\xi)\b)_\H := \int_\Omega\int_\Dom D_xu^\star(x)\cdot D_xv^\star(x)+\theta(\omega,x)\cdot\xi(\omega,x)\dx\dPr(\omega),
$$
it is straightforward to check that $\H$ is a Hilbert space (recall that the two ``component'' inner products were defined in Section~\ref{sec:Deformations} and Section~\ref{sec:Correctors}). We note the following fact which we frequently use below:
\begin{equation}
  \forall (u^\star,\theta) \in \H, \quad \forall (v^\star,\xi)\in\H, \qquad 
  \int_\Omega\int_\Dom \b(D_xu^\star+\theta) \cdot(D_x v^\star+\xi) \dPr(\omega)\dx =\b((u^\star,\theta),(v^\star,\xi)\b)_\H.
  \label{eq:decoupling}
\end{equation}
This follows from the fact that the expectation of any function in $\L2Pot(\Omega)$ vanishes (see Remark~\ref{rem:evident2}).

Consider the functionals $\Psi^0_t:\H\to\R$ and $\Phi^0_t:\H\to\R$ defined by
\begin{align*}
  \Psi^0_t(v^\star,\xi)
  &:=\int_\Omega \int_\Dom \psi_0\B(\omega,D_x\dot{\overline{y}}(t)+D_xv^\star+\xi\B)\dx\dPr(\omega),
  \\
  \Phi^0_t(u^\star,\theta)
  &:=\int_\Omega \int_\Dom W_0\B(\omega,D_x\overline{y}(t)+D_xu^\star+\theta\B)\dx\dPr(\omega).
\end{align*}
These functionals are well--defined since assumption~$(A1)$ made in Section~\ref{sec:elasticity_assumptions}, assumption~$(\psi1)$ made in Section~\ref{sec:dissipation_assumptions} and the fact that $\overline{y}\in\HH^1([0,T];\HH^1(\Dom))$ ensure measurability, while assumptions~$(A3)$ and~$(\psi4)$ ensure integrability. The functionals $\Phi^0_t$ and $\Psi^0_t$ are strongly convex on $\H$, a property which they inherit from the strong convexity of $\psi$ and $W$ (see e.g. assumption~$(\psi3)$) and an application of~\eqref{eq:decoupling}.

We may characterise $f\in\partial\Psi^0_t(v^\star,\xi)\subset\H'$ as
\begin{equation}\label{eq:homogenized_subdiff_a}
  \forall (w^\star,\nu)\in\H, \qquad
  \<f,(w^\star,\nu)\>_\H
  =
  \int_\Omega\int_\Dom \left(\sigma^\star+\zeta\right)\cdot(D_xw^\star+\nu)\dx\dPr(\omega)
\end{equation}
where $\sigma^\star\in\LL^2(\Dom)^d$ and $\zeta\in\LL^2(\Dom\times\Omega)^d$ satisfy, for almost every $(\omega,x)\in\Omega\times\Dom$,
\begin{equation}\label{eq:homogenized_subdiff_b}
\sigma^\star(x)+\zeta(\omega,x) \in \partial_\xi\psi_0\b(\omega,D_x\dot{\overline{y}}(t)+D_xv^\star+\xi\b) \quad \text{and} \quad \int_\Omega \zeta(\omega,x)\dPr(\omega)=0.
\end{equation}
Applying~\eqref{eq:subdiff_bnd}, we see that there exists $K>0$ such that
\begin{equation*}
  \|f\|_{\H'} \leq K \B( \Leb^d(\Dom) + \left\|D_x\dot{\overline{y}}(t) \right\|_{\LL^2(\Dom)^d} + \|(v^\star,\xi)\|_\H \B) \qquad \text{for any $f\in\partial\Psi^0_t(v^\star,\xi)$}.
\end{equation*}
Moreover, for any $(u^\star,\theta)\in\H$, we have
\begin{equation} \label{eq:tata}
  \forall (w^\star,\nu)\in\H, \qquad
  \b\<\nabla\Phi^0_t(u^\star,\theta),(w^\star,\nu)\b\>_\H
=
\int_\Omega \int_\Dom A_0(\omega)\left(D_x\overline{y}(t)+D_x u^\star+\theta\right)\cdot \left(D_xw^\star+\nu\right)\dx\dPr(\omega),
\end{equation}
and therefore it is straightforward to show that
\begin{equation*}
  \left\|\nabla\Phi^0_t(u^\star,\theta)\right\|_{\H'} \leq \overline{A} \, \B( \left\|D_x\overline{y}(t)\right\|_{\LL^2(\Dom)^d} + \|(u^\star,\theta)\|_\H \B).
\end{equation*}

\medskip

We understand the equations~\eqref{eq:corrector_eqn} and~\eqref{eq:homogenised_eqn} as the inclusion
\begin{equation}
  \partial \Psi^0_t\left(\dot{u}^\star(t),\dot{\theta}(t)\right) +\nabla\Phi^0_t\left(u^\star(t),\theta(t)\right)\ni0\quad\text{in $\H'$ for almost every $t\in[0,T]$}.
  \label{eq:homogenised_system}
\end{equation}
Indeed, using~\eqref{eq:homogenized_subdiff_a} and~\eqref{eq:tata}, the inclusion~\eqref{eq:homogenised_system} can be recast as: for any $(w^\star,\nu)\in\H$,
\begin{equation} \label{eq:maison_seattle}
\int_\Omega\int_\Dom \Big( \sigma^\star+\zeta+A_0(\omega)\left(D_x\overline{y}(t)+D_x u^\star+\theta\right) \Big) \cdot(D_xw^\star+\nu)\dx\dPr(\omega) = 0
\end{equation}
where, in view of~\eqref{eq:homogenized_subdiff_b}, $\sigma^\star\in\LL^2(\Dom)^d$ and $\zeta\in\LL^2(\Dom\times\Omega)^d$ satisfy, for almost every $(\omega,x)\in\Omega\times\Dom$,
$$
\sigma^\star(x)+\zeta(\omega,x) \in \partial_\xi\psi_0\b(\omega,D_x\dot{\overline{y}}(t)+D_x\dot{u}^\star+\dot{\theta}\b) \quad \text{and} \quad \int_\Omega \zeta(\omega,x)\dPr(\omega)=0.
$$
Taking $\nu \equiv 0$ in~\eqref{eq:maison_seattle}, we obtain~\eqref{eq:homogenised_eqn}. Taking $w^\star \equiv 0$ and $\nu(x,\omega) = \nu_1(x) D_\omega \nu_2(\omega)$ in~\eqref{eq:maison_seattle} (for any $\nu_1 \in \LL^2(\Dom)$ and $\nu_2 \in \HH^1(\Omega)$), we obtain~\eqref{eq:corrector_eqn}. \blue{Conversely, it is clear that~\eqref{eq:homogenised_eqn}--\eqref{eq:corrector_eqn} (in the sense of~\eqref{eq:homogenised_eqn_bis}--\eqref{eq:corrector_eqn_bis}) imply~\eqref{eq:maison_seattle}.}

\medskip

The following lemma demonstrates the existence and uniqueness of solutions to the problem~\eqref{eq:homogenised_system}.
  
\begin{lemma}
\label{th:0_existence}
For any initial condition $(u^\star_0,\theta_0) \in \H$, there exists a unique solution $(u^\star,\theta)\in\HH^1([0,T];\H)$ to~\eqref{eq:homogenised_system} with $u^\star(0,\cdot)=u^\star_0$ and $\theta(0,\cdot,\cdot)=\theta_0$. Moreover, $(u^\star,\theta)$ satisfies
\begin{equation*}
  0 = \int_0^T \Psi^0_t\left(\dot{u}^\star(t),\dot{\theta}(t)\right) + (\Psi^0_t)^*\Big(-\nabla\Phi^0_t(u^\star(t),\theta(t))\Big) + \left\<\nabla\Phi^0_t\left(u^\star(t),\theta(t)\right),\left(\dot{u}^\star(t),\dot{\theta}(t)\right)\right\>\dt.
\end{equation*}
\end{lemma}

The proof proceeds along similar lines as that of Theorem~\ref{th:eps_existence}.
  
\begin{proof}
For any $(u^\star,\theta) \in \H$, consider the functional $\mathcal{I}^0_t:\H\to\R$ defined by
\begin{equation*}
  \mathcal{I}^0_t(v^\star,\xi)=\Psi^0_t(v^\star,\xi) +\b\<\nabla\Phi^0_t(u^\star,\theta),(v^\star,\xi)\b\>_\H.
\end{equation*}
Using~\eqref{eq:psi_lower_bound} and~\eqref{eq:decoupling}, and applying Young's inequality, we obtain
\begin{align*}
  \mathcal{I}^0_t(v^\star,\xi)
  &\geq c\int_\Omega\int_\Dom\b|D_x\dot{\overline{y}}(t)+D_xv^\star+\xi\b|^2\dx\dPr(\omega)-\b\|\nabla\Phi^0_t(u^\star,\theta)\b\|_{\H'}\b\|(v^\star,\xi)\b\|_\H
  \\
  & = c \int_\Omega\int_\Dom \b|D_x\dot{\overline{y}}(t)+D_xv^\star\b|^2 + |\xi|^2\dx\dPr(\omega) -\b\|\nabla\Phi^0_t(u^\star,\theta)\b\|_{\H'}\b\|(v^\star,\xi)\b\|_\H
  \\
  &\geq c\,\|(v^\star,\xi)\|^2_\H + c\left\|D_x\dot{\overline{y}}(t)\right\|^2_{\LL^2(\Dom)^d} - 2c\int_\Dom|D_x\dot{\overline{y}}(t) \cdot D_xv^\star|\dx-\frac1{c}\b\|\nabla\Phi^0_t(u^\star,\theta)\b\|_{\H'}^2-\frac{c}{4}\b\|(v^\star,\xi)\b\|^2_\H \\
  &\geq \frac{c}{2}\,\|(v^\star,\xi)\|^2_\H - 3c\left\|D_x\dot{\overline{y}}(t)\right\|^2_{\LL^2(\Dom)^d} - \frac1{c}\b\|\nabla\Phi^0_t(u^\star,\theta)\b\|_{\H'}^2.
\end{align*}
It follows that $\mathcal{I}^0_t$ is coercive on $\H$. We have noted above that $\Psi^0_t$ is strongly convex on $\H$, and so is $\mathcal{I}^0_t$. The functional $\mathcal{I}^0_t$ hence admits a unique minimizer in $\H$, denoted $\Vel^0_t(u^\star,\theta)$. Moreover, this minimizer satisfies
\begin{equation*}
  0\in\partial \Psi^0_t\b(\Vel^0_t(u^\star,\theta)\b) +\nabla\Phi^0_t(u^\star,\theta).
%  \label{eq:limit_equilibrium}
\end{equation*} 
Using the fact that $\nabla \Phi^0_t$ is a bounded affine map from $\H$ to $\H'$, we may use the strong convexity of $\Psi^0_t$ to argue as in Section~\ref{sec:velocity-operator}, and deduce (see~\eqref{eq:Vel_Lipschitz}) that
\begin{equation*}
  \forall (u^\star,\theta) \in \H, \quad \forall (v^\star,\xi) \in \H, \qquad \Big\|\Vel^0_t(u^\star,\theta)-\Vel^0_t(v^\star,\xi)\Big\|_\H \leq \frac{\overline{A}}{2c} \, \b\|(u^\star,\theta)-(v^\star,\xi)\b\|_\H,
\end{equation*}
which demonstrates that $\Vel^0_t:\H\to\H$ is a uniformly Lipschitz operator.
  
Next, for any $(u^\star,\theta)\in\LL^2([0,T];\H)$, a similar argument to that given in Section~\ref{sec:a_priori_bound} (see~\eqref{eq:apriori_pointwise}) entails that, for almost every $t\in[0,T]$,
\begin{equation*}
  \B\|\Vel^0_t(u^\star(t),\theta(t))\B\|_\H \leq C_0 \B( 1+ \|D_x\dot{\overline{y}}(t)\|_{\LL^2(\Dom)^d} + \|D_x\overline{y}(t)\b\|_{\LL^2(\Dom)^d} +\b\|(u^\star(t),\theta(t))\b\|_\H\B)
\end{equation*}
for some $C_0$ independent of $t$. Since $\overline{y}\in\HH^1([0,T];\HH^1(\Dom))$, we have that $t \mapsto \Vel^0_t(u^\star(t),\theta(t))$ belongs to $\LL^2([0,T];\H)$.

Finally, an argument identical to that given in Section~\ref{sec:existence_eps} entails the existence and uniqueness of a solution to~\eqref{eq:homogenised_system} in $\HH^1\b([0,T];\H\b)$ with the initial condition $(u^\star_0,\theta_0) \in \H$.

Moreover, \eqref{eq:homogenised_system} is satisfied for almost every $t\in[0,T]$ if and only if
\begin{equation*}
  0 = \int_0^T \Psi^0_t\left(\dot{u}^\star(t),\dot{\theta}(t)\right) + (\Psi^0_t)^*\B(-\nabla\Phi^0_t(u^\star(t),\theta(t))\B) + \left\<\nabla\Phi^0_t\left(u^\star(t),\theta(t)\right),\left(\dot{u}^\star(t),\dot{\theta}(t)\right)\right\>\dt.
\end{equation*}
The proof of this statement follows the same lines as those used in Step 1 of Section~\ref{sec:conv_subseq}, where we recast the evolution problem~\eqref{eq:EvolutionProblem} in the integral form~\eqref{eq:zero_energy}. This concludes the proof of Lemma~\ref{th:0_existence}.
\end{proof}

\subsection{Conclusion of the proof of Theorem~\ref{th:main}}
\label{sec:conc}

We have shown in Section~\ref{sec:conv_subseq} that any subsequence of $u^\eps$ contains a further subsequence which satisfies the conclusions of Lemma~\ref{th:2scale_H1}, where the corresponding limits $u^\star$ and $\theta$ satisfy~\eqref{eq:matin}, and hence~\eqref{eq:homogenised_eqn}--\eqref{eq:corrector_eqn}. We have next shown in Section~\ref{sec:exist-uniq-limit} that~\eqref{eq:homogenised_eqn}--\eqref{eq:corrector_eqn} can be understood as the inclusion~\eqref{eq:homogenised_system}. In view of Lemma~\ref{th:0_existence}, we observe that the solution to~\eqref{eq:homogenised_system} exists and is unique. Since the initial subsequence is arbitrary, and since $\Phi^0_t$ and $\Psi^0_t$ in~\eqref{eq:homogenised_system} do not depend on that subsequence, it follows that the entire sequence $u^\eps$ satisfies two--scale convergence in the sense stated in Lemma~\ref{th:2scale_H1} to the unique solution to~\eqref{eq:homogenised_system}. The proof of Theorem~\ref{th:main} is therefore complete.

\subsection*{Acknowledgements}

\noindent
{\bf Thanks:} The authors would like to thank Felix Otto and Patrick Le~Tallec for informative discussions while carrying out this work. They also thank the anonymous referees whose detailed comments helped to significantly improve this article.
      
%\medskip

% \noindent
% {\bf Funding:} The work of TH was funded by a public grant overseen by the French National Research Agency (ANR) as part of the ``Investissements d'Avenir'' program (reference: ANR-10-LABX-0098). TH was also funded by an Early Career Fellowship awarded by the Leverhulme Trust (ECF-2016-526). The work of FL and TL was supported by the European Research Council under the European Union's Seventh Framework Programme (FP/2007-2013) / ERC Grant Agreement number 614492.

\medskip

\noindent
{\bf Conflict of interest:} The authors declare that there is no conflict of interest regarding this work.
    
\bibliographystyle{plain}
\bibliography{michelin}
\end{document}